\documentclass{imsart}
\RequirePackage[OT1]{fontenc}
\RequirePackage{amsthm,amsmath}
\RequirePackage[numbers]{natbib}
\RequirePackage[colorlinks,citecolor=blue,urlcolor=blue]{hyperref}
\usepackage{graphicx}
\usepackage{natbib}
\usepackage[utf8]{inputenc}
\usepackage{graphics}
\usepackage{amsmath}
\usepackage{amssymb}
\usepackage{mathtools}
\usepackage{color}
\usepackage{babel}[english]
\usepackage{verbatim}
\usepackage{enumitem}
\usepackage[ruled,vlined]{algorithm2e}
\usepackage{upgreek}
\usepackage{marginnote}
\usepackage{multirow}
\usepackage{diagbox}
\usepackage{slashbox}
\usepackage{graphicx}
\usepackage{caption}
\usepackage{subcaption}

\newcommand{\gammap}{\gamma^\prime}
\newcommand{\deltap}{\delta^\prime}
\newcommand{\three}{\tilde{\chi}}
\newcommand{\four}{\hat{\chi}}
\newcommand{\ffour}{\check{\chi}}

\newcommand{\cC}{\mathcal{C}}
\newcommand{\cD}{\mathcal{D}}
\newcommand{\cF}{\mathcal{F}}

\newcommand{\cL}{\mathcal{L}} 
\newcommand{\cP}{\mathcal{P}}

\newcommand{\cS}{\mathcal{S}}

\newcommand{\cY}{\mathcal{Y}}

\newcommand{\bE}{\mathbb{E}}
\newcommand{\bP}{\mathbb{P}}
\newcommand{\Pro}{\mathsf{P}}
\newcommand{\Qro}{\mathrm Q}
\newcommand{\Exp}{\mathsf{E}}
\newcommand{\bS}{\mathbb{S}}
\newcommand{\bN}{\mathbb{N}}
\newcommand{\bR}{\mathbb{R}}

\newcommand{\SPRT}{\st^\prime}
\newcommand{\SPRTd}{\dr^\prime}
\newcommand{\TST}{\tilde{\st}}
\newcommand{\TSTd}{\tilde{\dr}}
\newcommand{\FST}{\hat{\st}}
\newcommand{\FSTd}{\hat{\dr}}
\newcommand{\FFST}{\check{\st}}
\newcommand{\FFSTd}{\check{\dr}}

\newcommand{\tgamma}{\tilde{\gamma}}
\newcommand{\tdelta}{\tilde{\delta}}
\newcommand{\hgamma}{\hat{\gamma}}
\newcommand{\hgammap}{\hat{\gammap}}

\newcommand{\lev}{\xi} 
\newcommand{\f}{\mathcal{U}}  
\newcommand{\ff}{\mathcal{V}}

\newcommand{\thetas}{\theta_*}

\newcommand{\myremark}{\noindent \underline{\textbf{Remark}}: }

\newcommand{\st}{\tau} 
\newcommand{\Stat}{T}  

\newcommand{\n}{n^*}     
\newcommand{\ka}{\kappa}  
\newcommand{\Ka}{K}
\newcommand{\thre}{\kappa^*} 
\newcommand{\dr}{d}

\newcommand{\phis}{\phi^*}

\newtheorem{theorem}{Theorem}[section]
\newtheorem{corollary}{Corollary}[theorem]
\newtheorem{lemma}{Lemma}[section]

 \theoremstyle{plain}

\pubyear{2022}
\volume{0}
\issue{0}
\firstpage{1}
\lastpage{51}

\startlocaldefs
\numberwithin{equation}{section}
\theoremstyle{plain}

\endlocaldefs

\begin{document}

\begin{frontmatter}
\title{3-stage and 4-stage  tests  \\
with deterministic stage sizes \\
and non-iid data\support{This research was supported in part by the US National Science
Foundation under grant ATD-1737962 through the University of
Illinois at Urbana-Champaign.}}
\runtitle{3-stage and 4-stage tests}

\begin{aug}
\author{\fnms{Yiming} \snm{Xing}
\ead[label=e1]{yimingx4@illinois.edu}}
\and
\author{\fnms{Georgios} \snm{Fellouris}%
\ead[label=e2]{fellouri@illinois.edu}}

\address{725 S. Wright St. Champaign, IL 61822, USA \\
University of Illinois,  Urbana-Champaign\\
\printead{e1,e2}}

\runauthor{Y. Xing and G. Fellouris}

\affiliation{University of Illinois at Urbana-Champaign}

\end{aug}

\begin{abstract} Given a  fixed-sample-size test that  controls the  error probabili-ties  under two specific, but arbitrary, distributions, a 3-stage and two 4-stage tests  are proposed and  analyzed. For each of them,  a novel, concrete, non-asymptotic, non-conservative design is specified,  which  guarantees the same error control as the given fixed-sample-size test.  Moreover,  first-order asymptotic approximation are established on their expected sample sizes under the two prescribed distributions as the error probabilities go to zero.   As a corollary, it is shown that the proposed multistage tests can achieve, in  this asymptotic sense, the  optimal expected sample size  under these  two  distributions  in the class of all sequential tests with the same error control.   Furthermore, they  are shown to be much more robust than Wald's  SPRT when applied to  one-sided testing problems  and the error probabilities under control are small enough.    These  general results are applied to  testing problems in the iid setup and beyond, such as testing  the correlation coefficient of a first-order autoregression, or   the transition matrix of a finite-state Markov chain, and are illustrated in various numerical studies. 
\end{abstract}




\tableofcontents
\end{frontmatter}

\section{Introduction}
A typical motivation for employing a  \textit{sequential} test, i.e., a testing procedure whose sample  size depends on the collected observations,
 is that its \text{average}  sample size can be much  smaller   than that of the corresponding  \textit{fixed-sample-size} test.  
The first  test of this kind in the literature   was  the \textit{double  sampling}  procedure  of Dodge and Romig \cite{Dodge_Romig_1929}, 
a   precursor to  Wald's  Sequential Probability Ratio Test (SPRT)   \cite{Wald1948OptimumCO} and the field of ``sequential analysis''. 
However, the implementation of the  SPRT, as well as   of many  sequential tests in the literature (see, e.g.,    \cite{Tartakovsky_Book}),  requires  continuous monitoring of the data collection process, which is often inconvenient, or even infeasible,  in  application areas such   as   sampling inspection and   clinical trials  \cite{Jennison_Turnbull_Book,Bartroff_Book_Clinicaltrials}.  As a result, the emphasis in such applications   has been  on   \emph{group-sequential} tests,  like the  one  in \cite{Dodge_Romig_1929},  i.e., sequential tests whose implementation  requires  the collection of only a  small  number of  \textit{groups of samples}. 
An equivalent terminology, which we  use in this work,  is   \emph{multistage tests},   in which the  \text{groups of samples} are referred to as  \textit{stages}.

Most works about multistage tests,  e.g.,  \cite{Armitage_1969,Pocock_1977,OBrien_Fleming_1979,Pocock_1982,Wang_Tsiatis_1987,Emerson_Fleming_1989, Eales_Jennison_1992,PAMPALLONA1994,Barber_Jennison_2002}, (i) focus on testing the mean of iid Gaussian observations with known variance, (ii)  are  designed to  control prescribed  type-I and type-II error probabilities under two specific distributions, and (iii) require equal stage sizes.  Free parameters, if any, as in \cite{Wang_Tsiatis_1987}, are selected to optimize the expected sample size under a certain distribution, such as the one under which the type-II error probability is controlled. 
This  optimization is  performed via  dynamic programming in \cite{Eales_Jennison_1992, Barber_Jennison_2002}.   

Multistage tests with unequal and random stage sizes are considered in  \cite{Lan_DeMets_1983, Kim_DeMets_1987,Jennison_1987}, as well as in  \cite{Lai_Shih_2004}. In the latter,   more general testing problems, regarding the parameters of an exponential family, are also  studied.

In  all the  above works the stage sizes are  treated as user-specified inputs.   Lorden in  \cite{Lorden_1983} showed that  3-stage  tests, \textit{with properly selected  stage sizes},   achieve asymptotically the optimal expected sample size, under both hypotheses, among all sequential tests with the same or smaller  error probabilities  as the latter  go to 0.   In the case of simple hypotheses for  iid data,  this was shown for tests with   \textit{deterministic} stage sizes  \cite[Section 2]{Lorden_1983}.   On the other hand,  in the case of composite hypotheses for the one-sided testing problem in a one-parameter exponential family, this was shown for  tests whose   stage sizes are \textit{adaptive}, i.e., they  can  depend on the  data from the previous  stages \cite[Section 3]{Lorden_1983}. Such multistage tests were also  considered in \cite{Bartroff_2008_adaptive, Bartroff_2008_anotheradaptive}, where they were designed to  be less conservative than in \cite[Section 3]{Lorden_1983}.  All these asymptotic optimality results require  certain  assumptions on  the decay rates of the prescribed error probabilities, which are not allowed to  go to 0 very asymmetrically.

In the present work we focus on the  design and analysis of multistage tests with \textit{deterministic} stage sizes, and we strengthen, extend and generalize   the results in \cite[Section 2]{Lorden_1983}. First of all,  unlike all the above mentioned works, we do not require  that the observations  be either independent or identically distributed.  Instead, we only assume that a  fixed-sample-size test  is given, which can control the  type-I and type-II error probabilities under two   specific  distributions below arbitrary levels.   Given such a test,  we introduce and analyze    a \textit{3-stage test}, that generalizes the one in \cite[Section 2]{Lorden_1983}, as well as   two novel \textit{4-stage tests}. For each of them, 
 we   propose a  novel, concrete, non-asymptotic, non-conservative  specification, 
 which  guarantees the same error control as the  fixed-sample-size test.     This specification  only requires  knowledge of  the number of observations  and the threshold the fixed-sample-size test requires for its error control. 
 While there are not, in general, explicit formulas for these quantities, they can be  estimated via  simulation. In the case of very small error probabilities, in which plain Monte-Carlo  is not efficient or even feasible (see, e.g., \cite{Bucklew_Book}), we  propose  a simulation approach  via importance sampling.

In order to obtain theoretical insights regarding the proposed multistage tests,  we  impose  some  structure on the above general  setup. Essentially, we assume  that there are thresholds for which the error probabilities of the given  fixed-sample-size test,   under the two prescribed distributions,  decay  exponentially fast in  the sample size.  Using 
the  G\"artner-Ellis theorem  from large deviation theory (see, e.g., \cite{Dembo_Zeitouni_LDPBook}), we show that the required conditions are satisfied in various testing problems beyond the iid setup. Two specific examples,  which we work out in detail,  are  testing the  correlation coefficient of a first-order  autoregression, and  testing the transition matrix of an irreducible and recurrent finite-state  Markov chain.

Assuming that the above conditions hold,    we   establish  first-order  asymptotic approximations to the expected sample sizes of the proposed  multistage tests under the distributions  with respect to which we control the error probabilities, as the latter  go to 0. For the 3-stage test, the relative decay of the  error probabilities is   allowed to be much more asymmetric than the one  required in  \cite[Section 2]{Lorden_1983}. Even more asymmetric rates are allowed for each of the two  4-stage tests.   As a corollary,    we   extend  the asymptotic optimality of the 3-stage test  in \cite[Section 2]{Lorden_1983}, beyond the iid setup and  for  more asymmetric error probabilities. Moreover,
we  show that  the two proposed 4-stage tests are asymptotically  optimal,  in the same setup as the 3-stage test, with even  more asymmetric  error probabilities.  These   results are also  illustrated in a numerical study, where these  multistage tests are compared with the SPRT with respect to their average sample sizes under the two prescribed distributions. 

In order to  obtain a more  complete understanding of  how the  proposed multistage tests perform, especially in comparison to  the SPRT,  it is important to assess their  behavior  when the true distribution is different from those under which we control the error probabilities.  Indeed,  when the SPRT is applied to the  one-sided testing problem for the mean of iid Gaussian observations with known variance, as suggested in \cite[Chapter 7.5]{Wald_Book}, its  \text{expected}  sample size can be much larger \textit{even than that of the  corresponding optimal fixed-sample-size test} when the true mean is between the values used for the design of the SPRT (see, e.g.,  \cite{Bechhofer60}).   Motivated by this phenomenon, we establish  a \textit{distribution-free} asymptotic upper bound on the  expected sample sizes of the proposed multistage tests, as at least one of the two prescribed error probabilities goes to 0.  This  reveals that when the prescribed error  probabilities are small enough, the proposed multistage tests are much more  robust  than the SPRT, thus, they  may be preferable  not only because of  their practical advantages, but also based  on  statistical  considerations.


The remainder of this paper is organized as follows.
In Section \ref{sec:formulation} we formulate the testing setup and in Section  \ref{sec:multistage} we   introduce and analyze the proposed multistage tests.  In Section \ref{sec:asy} we state our  asymptotic results,
and in Section \ref{sec:sufficient} we state sufficient conditions for this asymptotic analysis.
In Section \ref{sec: IS} we  propose an importance sampling  approach for the   implementation of  the proposed tests when the error probabilities are small.
  In Section \ref{sec:examples} we illustrate the  general theory in   three specific testing problems.
In Section \ref{sec:numerical}  we present the results of our numerical studies. In Section \ref{sec:conclusions} we conclude and discuss potential extensions. The  proofs of most results  are presented in  Appendices 
\ref{app:A}, \ref{app:new}, \ref{app: proof from Assumptions B to Assumptions A}. 

Finally, we introduce some  notations that we use throughout the paper. We denote by $\bN$  the set of positive integers, i.e., $\bN \equiv \{1,2,\ldots\}$, and by $\bR$ the set of real numbers.   For a set $A$ 
we denote by $1\{A\}$ its indicator function and by $A^o$ its interior.   For a function $f:\bR\to (-\infty,\infty]$, we call $\{x\in\bR: \,f(x)<\infty\}$ the effective domain of $f$, and denote by $f(x+)$ the right limit and by $f(x-)$ the left limit of $f$ at $x\in \bR$,  when they exist. For $x, y \in \bR$ we set  $x\wedge y \equiv \min\{x,y\}$ and $x\vee y \equiv\max\{x,y\}$. For positive sequences $(x_n), (y_n)$,  we write  $x_n \sim y_n$ for  $\lim_n (x_n/y_n) =1$, $x_n \gtrsim y_n$ for $\underline\lim_n \, (x_n/y_n) \geq 1$,  $x_n\lesssim y_n$ for $\overline\lim_n \,  (x_n/y_n) \leq 1$, $x_n<<y_n$ for $ x_n/y_n \to 0$, and $x_n>>y_n$ for  $x_n/y_n  \to \infty$.

\section{Problem formulation}
  \label{sec:formulation}
We consider a sequence of $\bS$-valued random elements, $X\equiv \{X_n, n\in \mathbb N\}$,  where $(\bS, \cS)$ is  an arbitrary measurable space. For any  $n \in \bN$,  we denote by  $\cF_n$  the $\sigma$-algebra generated by the first $n$ terms of this sequence, i.e.,   $\cF_n \equiv \sigma(X_1,\ldots,X_n).$
Moreover, we denote by  $\Pro$ the distribution of $X$,   assume that it belongs  to some family, $\cP$,  and consider  the following hypotheses for it,
\begin{equation} \label{Binary testing problem formulation}
    H_0: \, \Pro \in \cP_0   \qquad \text{  versus  } \qquad   H_1: \Pro \in \cP_1,
\end{equation}  
where   $\cP_0$ and   $\cP_1$ are disjoint subsets of $\cP$.

\subsection{Tests}
We allow  the data to be collected sequentially so that,  after each observation,   the decision   whether to stop sampling or not and, in the former case, whether to select the null or the alternative hypothesis, can depend on  all the already collected data. Thus,  we say that  $ \chi \equiv (\st,\dr)$ is  a \textit{test}  for \eqref{Binary testing problem formulation} if  the random time, $\st$, that represents the
utilized sample size,  is  a stopping time  with respect to  the filtration 
$\{\cF_n,\, n \in \bN\}$,  and  the Bernoulli random variable, $d$, that represents the decision ($H_i$ being selected when $d=i$,  where  $i \in \{0,1\}$)  is  $\cF_\st$-measurable,  i.e., 
$$ \{\st=n\}, \, \{\st=n, \, \dr=i\} \in \cF_n, \quad \forall \; n \in \bN, \;  i \in \{0,1\}.
$$

We denote by  $\cC$ the family of all tests, and  we further introduce a subfamily of tests that control the two error probabilities under two specific,  but arbitrary,   distributions. To be specific,  we  fix $\Pro_0 \in \cP_0$, $\Pro_1\in \cP_1$ and,  for any  $\alpha, \beta \in (0,1)$,    we  denote by  $\cC(\alpha, \beta)$  the family of  tests whose  type-I error probability under $\Pro_0$ does not exceed $\alpha$ and whose type-II probability under $\Pro_1$  does not exceed  $\beta$, i.e., 
\begin{equation} \label{Class of binary tests}
 \cC(\alpha, \beta) \equiv  \{ (\st,\dr) \in \cC :  \;  \Pro_0(\dr=1)\leq \alpha \quad \text{and} \quad \mathbb \Pro_1(\dr=0)\leq \beta \}. 
\end{equation}

\subsection{The fixed-sample-size test}
Our only standing assumption throughout the paper is that there is a  sequence of test statistics,
  $\Stat \equiv \{\Stat_n,\, n\in \bN\}$,   such that $\Stat_n$ is $\cF_n$-measurable   for every $n\in \bN$ and, 
\textit{for any  $\alpha,\beta\in (0,1)$,  there exist  $n \in \bN$  and  $\ka \in \bR$ so that  the fixed-sample-size test that rejects $H_0$ if and only if  $\Stat_n>\ka $  belongs to $\cC(\alpha, \beta)$}.
Suppresing the dependence on $T$, we denote by $\n(\alpha,\beta)$ the smallest such sample size, i.e., 
\begin{equation} \label{n*(alpha,beta)}
     \n(\alpha,\beta)  \equiv  \inf \{n \in \bN: \exists \, \ka \in\bR: \; \Pro_0( \Stat_n >\ka)\leq \alpha \; \text{and} \;   \Pro_1( \Stat_n \leq \ka) \leq \beta\},
\end{equation}
and  by $\thre(\alpha,\beta)$ any of the corresponding thresholds.  In  Section \ref{sec: IS} we discuss  the estimation of these quantities via Monte-Carlo simulation when they do not admit  closed-form expressions.

\subsection{Goals}
The main goal of this work is  to design \textit{multistage tests with deterministic stage sizes} that 
\begin{itemize}
\item[(i)] belong to $\cC(\alpha, \beta)$, for any choice of $\alpha,\beta\in(0,1)$,
\item[(ii)] are robust,  in the sense that their expected sample sizes  under any plausible distribution are not much  larger than   $\n(\alpha,\beta)$,   when $\alpha, \beta$ are small enough,
\end{itemize}
and,  if additionally the test statistic $T$ is selected appropriately, 
\begin{itemize}
\item[(iii)]  achieve  asymptotically, as $\alpha, \beta \to 0$, the optimal expected sample size  in $\cC(\alpha, \beta)$ under both  $\Pro_0$ and $\Pro_1$,   $\cL_0(\alpha, \beta) $ and $\cL_1(\alpha, \beta)$, where  
\end{itemize}
\begin{align}  \label{optimal}
\cL_i(\alpha, \beta) \equiv    \inf \left\{ \Exp_i[\st]:  (\st,\dr)\in \cC(\alpha,\beta) \right\}, \quad   i \in \{0, 1\},
\end{align}
and   $\Exp$ and  $\Exp_i$  represent expectation under $\Pro$ and $\Pro_i$,  $ i \in \{0,1\}$. \\

The   error control in (i) and the asymptotic optimality property in (iii) are common goals in many sequential testing formulation, including \cite[Section 2]{Lorden_1983}. In order to explain the necessity and  importance of the  robustness property in (ii), it is useful to consider the special case of the  generic  one-sided testing problem.

\subsection{The one-sided testing problem}
\label{subsec: one-sided testing problem}
Consider the case where   the family of plausible distributions, $\cP$, is parametrized by a scalar parameter, $\mu$, taking values in  an open interval $M\subseteq \bR$. That is, if   $\bP_\mu$  denotes the distribution, and $\bE_\mu$ the expectation, of $X$ when the true parameter is $\mu$, then
$$\cP=\{\bP_\mu: \mu \in M\}.$$ 
Moreover, suppose that the testing problem of interest is whether the true parameter $\mu$ is smaller or larger than some user-specified value, $\mu_* \in M$, i.e., 
\begin{equation} \label{genreal framework of one-sided testing in one-dim parametric family1}
H_0: \mu < \mu_* \qquad \text{versus} \qquad H_0: \mu> \mu_*,
\end{equation}
or equivalently
\begin{equation} \label{genreal framework of one-sided testing in one-dim parametric family2}
 \cP_0 =\{\bP_\mu: \mu < \mu_*\},   \quad \cP_1=\{\bP_\mu: \mu > \mu_*\}.
\end{equation}
If also, it is required that  the type-I   error probability  be controlled below $\alpha$ when $\mu=\mu_0$ and the  type-II error probability  below 
  $\beta$ when $\mu=\mu_1$, where $\alpha,\beta\in (0,1)$ and
  $$\mu_0, \mu_1 \in M,  \quad \mu_0< \mu_1, \quad \mu_0\leq  \mu_* \leq \mu_1,$$
then this  is  a special case of the framework of this section,  with
\begin{equation} \label{genreal framework of one-sided testing in one-dim parametric family3}
 \Pro_i  =\bP_{\mu_i}, \quad i \in \{0,1\}.
\end{equation}

In this context, the  asymptotic optimality  property in (iii) guarantees that the expected sample size   \textit{when the true parameter  is in  $\{\mu_0, \mu_1\}$} will be relatively close to the optimal in $\cC(\alpha, \beta)$, at least when $\alpha, \beta$ are small enough.  However,   it is well known (see e.g., \cite{Bechhofer60})  that the expected sample size of such an asymptotically optimal test may be unacceptably  large  when  the true parameter  is \textit{between} $\mu_0$ and $\mu_1$ (see also  Subsection \ref{sec: upper bound in the middle} below).  This phenomenon motivates the design of sequential tests that are asymptotically  optimal even when the true parameter  is not   in  $\{\mu_0, \mu_1\}$,  (see, e.g.,   \cite[Chapter 16]{Chernoff_book}).  Such  an asymptotic optimality property 
has  been established for fully sequential tests  (see, e,.g., \cite[Chapter 5]{Tartakovsky_Book}) and  for  multistage tests with \textit{adaptive} stage sizes (see, e.g., \cite{Lorden_1983}, \cite{Bartroff_2007}, \cite{Bartroff_2008_adaptive}). However, it is not, in general, achievable by multistage tests with  \textit{deterministic} stage sizes, which cannot easily adapt to the true value of the parameter.  Thus, the  robustness property in (ii) guarantees that, even if it is asymptotically suboptimal,  the average sample size  of such a multistage test does not exceed, at least by much,  that of the corresponding  fixed-sample-size test, no matter what the true distribution is.   As a result, it is a necessary  complement to the  asymptotic optimality property in (iii), making sure that  the latter   does not come at the price of an inflated expected sample size  when  the  true parameter  is between $\mu_0$ and $\mu_1$.  \\

\myremark  In the context of the above one-sided testing problem,  it is desirable that a  test  $ \chi\equiv (\tau, d)$  in   $\cC(\alpha, \beta)$  controls  the  type-I error probability below $\alpha$ for \textit{every} $\mu \leq \mu_0$  and the  type-II error probability below $\beta$ for \textit{every} $\mu \geq \mu_1$, i.e., 
\begin{align} \label{uniform error control}
\bP_\mu(d=1) \leq \alpha  \quad \forall \; \mu  \leq \mu_0
\qquad \text{and} \qquad 
\bP_\mu(d=0) \leq \beta  \quad \forall \; \mu \geq \mu_1,
\end{align}
where $\alpha,\beta\in (0,1)$.
This is obviously the case for the given fixed-sample-size test that rejects $H_0$ if and only if $T_n>\ka$  when
\begin{equation} \label{stochastic mono}
\begin{split}
    & \bP_{\mu_0}(T_n>\ka)=\sup_{\mu\leq \mu_0} \bP_\mu(T_n>\ka), \\
    & \bP_{\mu_1}(T_n\leq \ka)=\sup_{\mu\geq \mu_1} \bP_\mu(T_n\leq \ka).
\end{split}
\end{equation}
If the monotonicity property in \eqref{stochastic mono} holds for every $n\in \bN$ and  $\ka\in\bR$, then the uniform error control in \eqref{uniform error control} will also  hold,  for every $\alpha,\beta\in (0,1)$,  for the proposed multistage tests in this work.

\section{The multistage tests}
  \label{sec:multistage}
In this section we  introduce and analyze the multistage tests that  we consider in this work. 

\subsection{The 3-stage test}\label{sec:3ST}
We next introduce and analyze a  test that offers two opportunities to accept  the null hypothesis and two to reject it.
Its implementation requires the specification of three positive integers, $n_0,n_1,N$,  and three real thresholds, $\ka_0,\ka_1,\Ka$,  so that
$$n_0\vee n_1<N \qquad \text{and} 
\qquad \ka_0\leq\ka_1.$$
Specifically, $n_0$ (resp. $n_1$) is  the   number of observations that need to be collected by the first opportunity to accept (resp. reject) $H_0$,  and  $N$ the maximum  number of observations that can be collected. Indeed, given these parameters, the test proceeds as follows:
\begin{enumerate}
\item[(i)]  $ n_0 \wedge  n_1$ observations are initially   collected.
\begin{itemize}
\item  If   $n_0 \leq n_1$ and   $\Stat_{n_0}\leq \ka_0$, then  $H_0$ is accepted.
\item If  $n_1 \leq n_0$ and   $\Stat_{n_1}>
\ka_1$,  then  $H_0$ is rejected.
\end{itemize}

\item[(ii)] If  the decision has not been reached yet,   
 $(n_0 \vee n_1)- (n_0 \wedge  n_1)$ additional observations are collected. 
\begin{itemize}
\item  If   $n_0 \leq n_1$ and   $\Stat_{n_1}>
\ka_1$,  then   $H_0$ is rejected.
\item  If   $n_1 \leq n_0$ and   $\Stat_{n_0}\leq \ka_0$, then $H_0$ is accepted.
\end{itemize}
\item[(iii)]   If the decision has not been reached yet,   $N-(n_0 \vee n_1)$ additional observations are collected and  $H_0$ is rejected if and only if $\Stat_N >\Ka$. 
\end{enumerate}

This testing procedure  can be implemented by collecting  at most three samples of deterministic sizes. Thus, in what follows  we refer to it  as the \textit{3-stage test} and denote it by $\three \equiv (\TST,\TSTd)$.

\subsubsection{Error control}
By the definition of the 3-stage test  it follows that, for any selection of its  parameters and any $\Pro \in \cP$, 
\begin{align}
\Pro(\TSTd=1) &\leq 
\Pro\left(\Stat_{n_1} > \ka_1 \right)
+\Pro\left(  \Stat_{N} >\Ka\right), \label{3ST, errors, rejecting}\\
\Pro(\TSTd=0) &\leq 
\Pro\left(\Stat_{n_0} \leq \ka_0\right)
        +\Pro \left( \Stat_{N}\leq \Ka \right). \label{3ST, errors, accepting}
\end{align}
Consequently,  by \eqref{3ST, errors, rejecting} with $\Pro=\Pro_0$ and  by \eqref{3ST, errors, accepting}  with $\Pro=\Pro_1$  we can see that 
 if the sample size and the threshold are  
\begin{equation} \label{n0,t0}
 n_0= \n(\gamma,\beta) \qquad \text{and} \qquad \ka_0=\thre(\gamma,\beta)  \qquad \text{for some} \quad \gamma \in (\alpha,1)
\end{equation}
in the first opportunity to accept $H_0$, 
\begin{equation} \label{n1,t1}
        n_1= \n(\alpha,\delta) \qquad \text{and} \qquad \ka_1=\thre(\alpha,\delta) \qquad \text{for some} \quad \delta \in (\beta,1)
\end{equation}
in the first opportunity to  reject $H_0$, and 
\begin{equation}  \label{N, c} 
    N =\n(\alpha,\beta) \qquad \text{and} \qquad  K=\thre(\alpha,\beta)
\end{equation}
in the final stage,  then 
\begin{equation*}  
\Pro_0(\TSTd=1) \leq 2\alpha \qquad \text{and} \qquad
\Pro_1(\TSTd=0) \leq 2\beta.
\end{equation*}

Thus, we have  shown the  following theorem.
\begin{theorem} \label{th:3ST_error_control}
Let $\alpha, \beta \in (0,1)$. If \eqref{n0,t0}-\eqref{N, c} hold with  $\alpha$ and $\beta$  replaced by   $\alpha/2$ and $\beta/2$ respectively,  then  $\three \in \cC(\alpha,\beta).$
\\
\end{theorem}

Theorem  \ref{th:3ST_error_control} specifies a design for  $\three \in \cC(\alpha, \beta)$  up to two  \textit{free} parameters,  $\gamma \in (\alpha/2,1)$ and $\delta \in (\beta/2,1)$.  Increasing the value of $\gamma$ (resp.  $\delta$) reduces the number of observations until the first opportunity to accept (resp. reject)  $H_0$, but  increases the probability of continuing to the final stage. To solve this  trade-off, we propose in Subsection \ref{subsec:3ST_free}   that   $\gamma$ (resp. $\delta$)  be selected to  minimize 
\textit{an upper bound on  $\Exp_0[\TST]$  (resp.  $\Exp_1[\TST]$) that is independent of $\delta$ (resp. $\gamma$)}.

\subsubsection{The average sample size}
By the definition of the 3-stage test it follows that, for any $\Pro \in \cP$,  
\begin{itemize}
\item if $n_0 \leq n_1<N$, then 
\begin{equation} \label{3ST, first case, ESS}
 \Exp[\TST] =    n_0 + (n_1-n_0) \cdot \Pro\left(\Stat_{n_0} > \ka_0\right) 
     + (N-n_1)\cdot \Pro\left(
     \begin{array}{c}
          \Stat_{n_0} > \ka_0  \\
          \Stat_{n_1} \leq \ka_1
     \end{array}
     \right),
     \end{equation}
\item if $n_1 \leq n_0<N$, then 
\begin{equation}
 \label{3ST, second case, ESS}
 \Exp[\TST] =    n_1 + (n_0-n_1) \cdot \Pro\left(\Stat_{n_1} \leq  
     \ka_1 \right) 
     + (N-n_0) \cdot \Pro\left(
     \begin{array}{c}
          \Stat_{n_1} \leq \ka_1  \\
          \Stat_{n_0} > \ka_0
     \end{array}
     \right).
     \end{equation}   
     \end{itemize}
Applying to these identities
the  basic inequalities:
$$\max \{0, \Pro(A)- \Pro(B^c)  \} \leq   \Pro(A \cap B)\leq \Pro(A),$$
we obtain, for any selection of the test parameters,  the following bounds:
\begin{equation} \label{3ST, ESS bound, null}
    \begin{split}
   \Exp[\TST]   &\geq  \; n_0 \cdot \Pro(\Stat_{n_1}\leq \ka_1) + (N-n_0) \cdot \left(\Pro(\Stat_{n_0}>\ka_0) - \Pro(\Stat_{n_1}>\ka_1) \right) ^+\\   \Exp[\TST] &\leq
     \; n_0 + (N-n_0) \cdot \Pro(\Stat_{n_0} > \ka_0) 
    \end{split}
\end{equation}
    and
\begin{equation} \label{3ST, ESS bound, alternative}
    \begin{split}
   \Exp[\TST]   &\geq     n_1 \cdot \Pro(\Stat_{n_0}>\ka_0) + (N-n_1) \cdot \big(\Pro(\Stat_{n_1}\leq \ka_1) - \Pro(\Stat_{n_0}\leq \ka_0) \big)^+ \\
 \Exp[\TST] 
      &  \leq \; n_1 + (N-n_1) \cdot \Pro(\Stat_{n_1} \leq \ka_1).
    \end{split}
\end{equation}

When, in particular,  the  test parameters   are selected as in Theorem \ref{th:3ST_error_control},   by \eqref{3ST, ESS bound, null} with $\Pro=\Pro_0$  we obtain 
\begin{align} \label{3ST, ESS bound, null, after plugin}
\begin{split}
    & n_0 \cdot (1-\alpha/2) + (N-n_0) \cdot  (\gamma-\alpha/2) \leq \Exp_0[\TST] \leq n_0 + (N-n_0) \cdot \gamma, \\
 & \text{where}   \quad    \gamma\in (\alpha/2,1),
\quad  n_0=\n(\gamma,\beta/2), \quad N=\n(\alpha/2,\beta/2),
\end{split}
 \end{align}
and   by  \eqref{3ST, ESS bound, alternative}  with $\Pro=\Pro_1$ we  obtain
\begin{align} \label{3ST, ESS bound, alternative, after plugin}
\begin{split}
  &  n_1 \cdot (1-\beta/2) + (N-n_1) \cdot (\delta - \beta/2 ) \leq \Exp_1[\TST] \leq n_1 + (N-n_1) \cdot \delta, \\ 
  & \text{where}  \quad \delta\in (\beta/2,1), \quad n_1=\n(\alpha/2,\delta), \quad N=\n(\alpha/2,\beta/2).
\end{split}
\end{align}

\subsubsection{Specification of the free parameters} \label{subsec:3ST_free}

For any selection of $\gamma$  (resp.  $\delta$) we can see that, at least when $\alpha$ (resp. $\beta$) is small, 
  the upper bound  in \eqref{3ST, ESS bound, null, after plugin} (resp. \eqref{3ST, ESS bound, alternative, after plugin})  is approximately equal to  the  lower bound and, as a result, it provides  an accurate approximation to $\Exp_0[\TST]$ (resp. $\Exp_1[\TST]$). Thus,  for any $\alpha, \beta \in (0,1)$,   we suggest selecting $\gamma$ and $\delta$ as
\begin{equation} \label{free_3ST}
\gamma=\tgamma \qquad \text{and} \qquad \delta=\tdelta,
\end{equation} 
where \textit{$\tgamma$ is a minimizer of the upper bound in  \eqref{3ST, ESS bound, null, after plugin}  and $\tdelta$  a minimizer of the upper bound  in \eqref{3ST, ESS bound, alternative, after plugin}}.

This selection of $\gamma$ (resp. $\delta$)  essentially minimizes $\Exp_0[\TST]$ (resp. $\Exp_1[\TST]$), at least when $\alpha$  (resp. $\beta$) is  small,  and it  is  practically convenient, as it requires the minimization  with respect to a single variable. Moreover, it  requires knowledge of only the function $n^*$,  defined in \eqref{n*(alpha,beta)},  which is also needed for the specification  of the other  test parameters according to Theorem \ref{th:3ST_error_control}.   \\

\myremark The  test of this section was proposed  in \cite{Lorden_1983} when $X$ is  an iid sequence  and  the test statistic, $T$, is the  corresponding average log-likelihood ratio. Our  setup here is essentially universal,  as  the only assumption  throughout this section about $X$ and $T$ is that  the corresponding fixed-sample-size test can control the error probabilities below arbitrary, user-specified levels, i.e., that $\n(\alpha, \beta)$ be  finite for any $\alpha, \beta \in (0,1)$.   At the same time,  we propose a concrete, non-asymptotic specification of the test parameters, which is novel and practically useful even in the setup of \cite[Section 2]{Lorden_1983}.

\subsection{The 4-stage tests} \label{sec:4ST}
Finally, we introduce and analyze two novel  tests, $\hat{\chi} \equiv (\FST,\FSTd)$ and  $\check{\chi} \equiv (\FFST,\FFSTd),$
which differ from that of the previous subsection only in that  the  first (resp. second) one allows for 
stopping  and  accepting (resp. rejecting)  the null hypothesis   if   the value of the test statistic, $T$,  after collecting  $N_0$  (resp. $N_1$) observations  is smaller (resp. larger) than  $\Ka_0$ (resp. $\Ka_1$), where
\begin{align*}
&n_0< N_0< N \qquad  \text{and} \qquad \Ka_0\leq \ka_1, \\
&n_1<N_1<N  \qquad \text{and} \qquad \ka_0\leq \Ka_1.
\end{align*}

Both these tests can be implemented by  collecting at most 4 samples of deterministic sizes, and for this reason we refer to  them as \textit{4-stage tests}.   
To avoid repetition, we present a detailed  analysis for  $\four$, and only state the corresponding results for  $\ffour$.  Thus,  given the above parameters,  $\four$ proceeds as follows:
\begin{enumerate}
    \item [(i)] $n_0\wedge n_1$ observations are initially collected.
    \begin{itemize}
        \item If $n_0\leq n_1$ and $\Stat_{n_0}\leq \ka_0$, then $H_0$ is accepted.
        \item If $n_1\leq n_0$ and $\Stat_{n_1}> \ka_1$, then $H_0$ is rejected.
    \end{itemize}
    \item [(ii)] If the decision has not been reached yet, $((n_0\vee n_1) \wedge N_0 )- (n_0\wedge n_1)$ additional observations are collected.
    \begin{itemize}
        \item If $n_0\leq n_1\leq N_0$ and $\Stat_{n_1}>\ka_1$, then $H_0$ is rejected.
        \item If $n_0\leq N_0\leq n_1$ and $\Stat_{N_0}\leq \Ka_0$, then $H_0$ is accepted.
        \item If $n_1\leq n_0\leq N_0$ and $\Stat_{n_0}\leq \ka_0$, then $H_0$ is accepted.
    \end{itemize}
    \item [(iii)]  If the decision has not been reached yet, $(n_1\vee N_0)-( (n_0\vee n_1) \wedge N_0)$ additional observations are collected. 
    \begin{itemize}
        \item If $n_1\leq N_0$ and $\Stat_{N_0}\leq \Ka_0$, then $H_0$ is accepted.
        \item If $N_0\leq n_1$ and $\Stat_{n_1}>\ka_1$, then $H_0$ is rejected.
    \end{itemize}
    \item [(iv)]  If the decision has not been reached yet,  
    $N-(n_1\vee N_0)$ additional observations are collected and $H_0$ is rejected if and only if $\Stat_N>\Ka$.
\end{enumerate}

\subsubsection{Error control}
By the definition of   $\hat{\chi}$  it follows that,
for any selection of its parameters  and any $\Pro \in \cP$,
\begin{align}
\Pro(\FSTd=1) & \leq \Pro\left(\Stat_{n_1}> \ka_1\right)
+
\Pro\left(\Stat_N>\Ka\right),  \label{4ST, error, I} \\
\Pro(\FSTd=0) & \leq \Pro\left(\Stat_{n_0} \leq \ka_0\right) +
\Pro\left(\Stat_{N_0} \leq \Ka_0\right) +
\Pro\left(\Stat_N \leq \Ka\right). \label{4ST, error, II}
\end{align} 
Therefore, if  $n_0, n_1, N$,  $\ka_0,  \ka_1, K$ are selected as in \eqref{n0,t0}--\eqref{N, c}  and we also set 
\begin{align} \label{N0, c0} 
         N_0 &=n^*(\gammap,\beta) \qquad \text{and} \qquad \Ka_0=c^*(\gammap,\beta) \qquad  \text{for some} \quad \gammap\in (\alpha,\gamma),  
\end{align}
 by \eqref{4ST, error, I}  with $\Pro=\Pro_0$ and by  \eqref{4ST, error, II} with $\Pro=\Pro_1$ we obtain
\begin{equation*}  
\Pro_0(\FSTd=1)\leq  2\alpha \qquad \text{and} \qquad \Pro_1(\FSTd=0) \leq 3\beta.
\end{equation*}

With a similar analysis we obtain 
$$
\Pro_0(\FFSTd=1)\leq  3\alpha \qquad \text{and} \qquad \Pro_1(\FFSTd=0) \leq 2\beta,
$$
when  $n_0, n_1, N$,  $\ka_0,  \ka_1, K$ are selected as in \eqref{n0,t0}--\eqref{N, c}  and also 
\begin{equation} \label{N1, c1}
    N_1=\n(\alpha,\deltap) \quad \text{ and } \quad \Ka_1=\thre(\alpha,\deltap) \quad \text{ for some } \deltap\in (\beta,\delta).
\end{equation}

Thus, we have shown  the following theorem.
\begin{theorem} \label{th:4ST_error_control}
Let $\alpha, \beta \in (0,1)$.
\begin{enumerate}
\item[(i)]  If \eqref{n0,t0}--\eqref{N, c} and \eqref{N0, c0}  hold with  $\alpha$ and $\beta$  replaced by   $\alpha/2$ and $\beta/3$ respectively,  then
$\four  \in \cC(\alpha,\beta)$. 
\item[(ii)] If  \eqref{n0,t0}--\eqref{N, c} and \eqref{N1, c1} hold with $\alpha$ and $\beta$ replaced by $\alpha/3$ and $\beta/2$ respectively,  then $\ffour \in \cC(\alpha,\beta)$. \\ 
\end{enumerate}
\end{theorem}

Theorem  \ref{th:4ST_error_control}   specifies  designs for  $\four$   and $\ffour$, which  guarantee the desired error control,  up  to three  \textit{free} parameters, $\gamma,\gammap,\delta$ and $\gamma,\delta,\deltap$ respectively. We  next propose a specific selection for these parameters,  similar to the one for the free parameters of the  3-stage test in Subsection \ref{subsec:3ST_free}.

\subsubsection{The average sample size}
By the definition of   $\four$ it follows that, for any  $\Pro \in\cP$,  
\begin{itemize}
\item if  $n_0\leq n_1 \leq  N_0 \leq N$, then 
\begin{align}
   \Exp[\FST] &=  n_0 +(n_1-n_0) \cdot \Pro\left(\Stat_{n_0} > \ka_0\right)  +(N_0-n_1)\cdot \Pro\left(
        \begin{array}{c}
             \Stat_{n_0} > \ka_0  \\
             \Stat_{n_1}\leq \ka_1
        \end{array}
        \right) \nonumber \\
        &+(N-N_0) \cdot \Pro \left(
        \begin{array}{c}
            \Stat_{n_0} > \ka_0 \\
            \Stat_{n_1} \leq \ka_1 \\
            \Stat_{N_0} > \Ka_0 \\
        \end{array}
        \right),
   \label{4ST, first case, ESS}
\end{align}
\item if  $n_0 \leq N_0 \leq n_1  \leq N$, then
\begin{align}
     \Exp [\FST] &=   n_0 +(N_0-n_0) \cdot \Pro\left(\Stat_{n_0} > \ka_0\right) +(n_1-N_0) \cdot \Pro\left(
     \begin{array}{c}
          \Stat_{n_0}> \ka_0  \\
          \Stat_{N_0}> \Ka_0
     \end{array}
     \right)  \nonumber \\
     &+(N-n_1) \cdot \Pro \left(
     \begin{array}{c}
          \Stat_{n_0} > \ka_0 \\
          \Stat_{N_0} > \Ka_0 \\
          \Stat_{n_1} \leq \ka_1
     \end{array}
     \right),
\end{align}
\item if $n_1\leq n_0 \leq  N_0 \leq N$, then 
\begin{align}
   \Exp  [\FST]  &=     n_1 +(n_0-n_1) \cdot \Pro\left(\Stat_{n_1} \leq \ka_1\right) +(N_0-n_0) \cdot\Pro\left(
        \begin{array}{c}
             \Stat_{n_1} \leq \ka_1  \\
             \Stat_{n_0} > \ka_0
        \end{array}
        \right) \nonumber \\
        &+(N-N_0) \cdot \Pro \left(
        \begin{array}{c}
            \Stat_{n_1} \leq \ka_1 \\
            \Stat_{n_0} > \ka_0 \\
            \Stat_{N_0} > \Ka_0 \\
        \end{array}
        \right).
\label{4ST, first case, ESS, n1<=n0}
\end{align}
\end{itemize}
Applying to the above identities 
 the following basic  inequalities:
$$
\max \{\Pro(A)- \Pro(B^c)- \Pro(C^c), 0\}  \leq \Pro(A \cap B \cap C) \leq \Pro(A),
$$
we obtain,  for any selection of the test parameters,    the following bounds: 
\begin{equation} \label{4ST, ESS bound, for the null}
    \begin{split}
        \Exp[\FST] & \geq  \; n_0 \cdot \Pro(\Stat_{n_1}\leq \ka_1) + (N_0-n_0) \cdot \big( \Pro(\Stat_{n_0}>\ka_0) - \Pro(\Stat_{n_1}>\ka_1) \big) \\
        & \quad + (N-N_0) \cdot \left( \Pro(T_{N_0}> K_0)-\Pro(T_{n_0}\leq \ka_0) - \Pro(T_{n_1}> \ka_1) \right)^+ , \\
        \Exp[\FST] &\leq  \; n_0 + (N_0-n_0) \cdot  \Pro(\Stat_{n_0} >\ka_0) + (N-N_0) \cdot \Pro(\Stat_{N_0}>\Ka_0)
    \end{split}
\end{equation}
and 
\begin{equation} \label{4ST, ESS bound, for the alternative}
    \begin{split}
       \Exp[\FST]    &\geq   \; n_1  \cdot \big( \Pro(\Stat_{n_0}>\ka_0) - \Pro(\Stat_{N_0}\leq \Ka_0) \big)  \\
        &+ \; (N-n_1) \cdot \big( \Pro(\Stat_{n_1}\leq \ka_1) - \Pro(\Stat_{n_0}\leq \ka_0) - \Pro(\Stat_{N_0}\leq \Ka_0) \big) \\
       \Exp[\FST] & \leq  \; n_1 + (N-n_1) \cdot \Pro(\Stat_{n_1}\leq \ka_1).
    \end{split}
\end{equation}

When, in particular, the parameters  of $\four$  are selected as in  Theorem \ref{th:4ST_error_control}.(i),   by \eqref{4ST, ESS bound, for the null} with $\Pro=\Pro_0$ we obtain
\begin{align} \label{4ST, ESS bound, after plugin, null}
    \begin{split}
\Exp_0[\FST]  &\leq  n_0 + (N_0-n_0) \cdot  \gamma + (N-N_0) \cdot \gammap ,\\
 \Exp_0[\FST]  &\geq n_0 \cdot  (1-\alpha/2) + (N_0-n_0) \cdot  (\gamma-\alpha/2) \\
 &+  (N-N_0)\cdot ( (1-\alpha/2) - ( 1-\gamma)-(1-\gamma') )^+ , \\
  \text{where}   \quad &\alpha/2<\gammap<\gamma<1, \quad  n_0=\n(\gamma,\beta/3), \; \\
&   N_0=\n(\gammap,\beta/3),  \quad  N=\n(\alpha/2,\beta/3),
    \end{split}
\end{align}
and   by \eqref{4ST, ESS bound, for the alternative} with $\Pro=\Pro_1$ we  obtain
\begin{align} \label{4ST, ESS bound, after plugin, alternative}
    \begin{split}
      & n_1 \cdot (1-2\beta/3) + (N-n_1) \cdot (\delta-2\beta/3)    \leq     \Exp_1[\FST]  \leq n_1 + (N-n_1) \cdot \delta
      \\
& \text{where} \quad  \delta\in (\beta/3,1), \quad 
n_1 =\n(\alpha/2,\delta), \quad
     N=\n(\alpha/2,\beta/3).
    \end{split}
\end{align}

With  a similar analysis it follows that when the parameters of  $\ffour$ are selected according to Theorem \ref{th:4ST_error_control}.(ii), then  
\begin{equation} \label{44ST, ESS bound, after plugin, null}
    \begin{split}
& n_0 \cdot  (1-2\alpha/3) + (N-n_0) \cdot (\gamma-2\alpha/3) \leq 
        \Exp_0[\FFST] \leq n_0 + (N-n_0) \cdot \gamma,\\
&\text{where}  \quad \gamma\in (\alpha/3,1), \quad 
n_0=\n(\gamma,\beta/2), \quad N=\n(\alpha/3,\beta/2),
    \end{split}
\end{equation}
and
\begin{align} \label{44ST, ESS bound, after plugin, alternative}
    \begin{split}
        \Exp_1[\FFST]         &\leq n_1 + (N_1-n_1) \cdot \delta + (N-N_1) \cdot \deltap,  \\
  \Exp_1[\FFST] &\geq     n_1 \, (1-\beta/2) + (N_1-n_1) \, (\delta-\beta/2) \\
  &+ (N-N_1)\cdot ((1-\beta/2)- (1-\delta)-(1-\delta') )^+, \\
\text{where} \quad & \beta/2<\deltap<\delta<1, \quad   n_1 =\n(\alpha/3,\delta), \\ 
    &N_1 =\n(\alpha/3,\deltap), \quad N=\n(\alpha/3,\beta/2).
    \end{split}
\end{align}

\subsubsection{Specification of the free parameters} \label{subsec:4ST_free}

 For any $\alpha, \beta \in (0,1)$, we propose selecting the free parameters of
  $\hat{\chi}$ as 
\begin{equation} \label{free_4ST_hat}\delta=\hat{\delta}, \quad  \gamma=\hat{\gamma}, \quad
 \gammap=\hat{\gammap}, 
 \end{equation}
where  \textit{$(\hat{\gamma}, \hat{\gammap})$ is
a minimizer of the upper bound in \eqref{4ST, ESS bound, after plugin, null} and $\hat{\delta}$  a minimizer of  the upper bound in \eqref{4ST, ESS bound, after plugin, alternative}}, and the free parameters of  $\check{\chi}$  as 
\begin{equation} \label{free_4ST_check}
\gamma=\check{\gamma}, \quad  \delta=\check{\delta}, \quad
 \deltap=\check{\deltap}, 
 \end{equation}
where  \textit{$\check{\gamma}$  is a minimizer of  the upper bound in \eqref{44ST, ESS bound, after plugin, null} and   ($\check{\delta},\check{\deltap}$)  a minimizer  of the upper bound in  \eqref{44ST, ESS bound, after plugin, alternative}.}  \\

\myremark 
Comparing  with the corresponding  results for the 3-stage test, we can see that,  at least when $\alpha$ (resp. $\beta$) is small,    $\hat{\delta}$  (resp. $\check{\gamma}$)  is close to    $\tdelta$  (resp. $\tgamma$), and  the  expected sample size of $\four$  (resp. $\ffour$) close to that of $\three$ under $\Pro_1$ (resp. $\Pro_0$).  Indeed, the additional stage in $\four$  (resp. $\ffour$)  is useful mainly for  reducing the expected sample size   under $\Pro_0$ (resp. $\Pro_1$).  This reduction is  illustrated  numerically  in  Figures  \ref{Figure: ratios} and \ref{Figure: performance against parameter}.  

\section{Asymptotic analysis.} \label{sec:asy}
In  this section we obtain asymptotic bounds and approximations, as $\alpha, \beta \to 0$,  to the  expected sample sizes of the multistage tests of the previous sections. 
For this analysis, we  need to   impose some structure on the 
 almost universal setup we have considered so far.
 

\subsection{Assumptions on the testing problem} \label{subsec: assumptions on the test prob}
Throughout this section, we assume that for every $ n \in \bN$, $\Pro_0$ and $\Pro_1$ are mutually absolutely continuous when restricted to $\cF_n$, and denote by $\Lambda\equiv\{\Lambda_n,\, n\in\bN\}$ and $\bar\Lambda\equiv\{\bar\Lambda_n,\, n\in\bN\}$ 
the corresponding  log-likelihood ratio and average log-likelihood ratio statistics, i.e., 
\begin{equation} \label{def: llr}  
\Lambda_n 
\equiv  \log \frac{d\Pro_1}{d\Pro_0} (\cF_n), \quad \bar\Lambda_n\equiv \Lambda_n/n, \quad n\in\bN.
\end{equation} 

We assume that there are  numbers  $I_0, I_1>0$  such that 
\begin{align} 
&     \Pro_0(\bar\Lambda_n\to -I_0)=\Pro_1(\bar\Lambda_n\to I_1)=1,  \label{SLLN on LLR} \\
\forall \;  \epsilon>0, \qquad   & \sum_{n=1}^\infty \Pro_0(\bar\Lambda_n >-I_0+\epsilon) +
  \sum_{n=1}^\infty \Pro_1(\bar\Lambda_n \leq I_1-\epsilon)  < \infty. \label{summability}
\end{align}
These assumptions  imply   (see, e.g., \cite[Lemma 3.4.1,  Theorem 3.4.2]{Tartakovsky_Book}) an asymptotic approximation, as $\alpha,\beta\to 0$,  to
$\cL_i(\alpha,\beta)$, $i\in\{0,1\}$, defined in \eqref{optimal}.
Specifically, 
as $\alpha,\beta\to 0$, 
\begin{align}  \label{optimal_rate}
\Exp_0[\SPRT] &\sim \cL_0(\alpha, \beta) \sim \frac{|\log\beta|}{ I_0}  \qquad \text{and} \qquad 
\Exp_1[\SPRT] \sim\cL_1(\alpha, \beta) \sim \frac{|\log\alpha|}{ I_1}, 
\end{align} 
where $\chi'\equiv (\SPRT, \SPRTd)$ is Wald's SPRT, i.e., 
\begin{equation} \label{Definition of SPRT}
    \SPRT\equiv \inf\{n \in \bN: \Lambda_n \notin (-A, B)\} \qquad \text{and} \qquad \SPRTd \equiv 1\{\Lambda_{\SPRT} \geq B\},
\end{equation}
with $A$ and $B$ selected, for example,  as  $A=|\log\beta|$ and  $B=|\log\alpha|$.

\subsubsection{The iid setup} \label{subsec: iid_setup}
When $X$ is  an iid sequence   with common  density  $f_i$ under $\Pro_i$  with respect to some dominating measure $\nu$, $i \in \{0,1\}$,  and  the Kullback-Leibler divergences are positive and finite, i.e.,
\begin{align} \label{KL divergences are positive}
\begin{split}
   D(f_0\| f_1) &\equiv \int \log (f_0/f_1) f_0 \, d\nu,  \\   
       D(f_1\| f_0) &\equiv \int \log (f_1/f_0) f_1 \, d\nu \in (0,\infty),
       \end{split}
\end{align}
then  the log-likelihood ratio statistic in \eqref{def: llr}  becomes 
\begin{align} \label{LLR_iid_setup}\Lambda_n &= \sum_{i=1}^n  \frac{f_1(X_i)}{ f_0(X_i)}, \quad  n\in \bN, \end{align}
and  
\eqref{SLLN on LLR}-\eqref{summability}  hold with $I_0=D(f_0\| f_1)$ and $I_1= D(f_1\| f_0)$ (for more details,  see  Subsection \ref{subsec:iid}). \\


\subsection{Assumptions on the test statistic} \label{subsec: assumptions on the test stat}
With respect to  the test statistic, $T$,  throughout this section we assume  that there are real numbers   $J_0, J_1$,  with  $J_0<J_1$, so that   
    \begin{equation} \label{a.s. limits of eta}
  \Pro_0(\Stat_n\to J_0)=\Pro_1(\Stat_n\to J_1)=1, 
    \end{equation}
and, for every  $\ka \in (J_0, J_1)$, the error probabilities of the fixed-sample-size test that rejects $H_0$ if and only if $T_n > \ka$ go to zero exponentially fast in $n$. Specifically,  we assume that  there are non-negative,  convex,   lower-semicontinuous functions  
$$\psi_0: \bR \to [0,\infty] \qquad \text{and} \qquad
\psi_1: \bR \to [0,\infty],$$
so that 
\begin{itemize}
\item[-]
$[J_0,J_1]$ is a subset of the effective domains of both  $\psi_0$ and $\psi_1$, 
\item[-] $\psi_0(J_0)=0$ and $\psi_0$ is strictly increasing in $[J_0,J_1]$,
\item[-]  $\psi_1(J_1)=0$ and $\psi_1$ is strictly decreasing in $[J_0,J_1]$, 
\item[-]  for every $ \ka\in (J_0,J_1)$,
    \begin{align}
    &  \lim_n \frac{1}{n} \log \Pro_0(\Stat_n >\ka) = -\psi_0(\ka), \label{LD, non-LLR_null}\\
    & \lim_n\frac{1}{n} \log \Pro_1(\Stat_n \leq \ka)  = -\psi_1(\ka). \label{LD, non-LLR_alternative} 
    \end{align}
    \end{itemize}

\noindent \underline{\textbf{Remarks}:} 
1) When $T_n=\bar\Lambda_n$, \eqref{a.s. limits of eta} is the same as \eqref{SLLN on LLR} and \eqref{LD, non-LLR_null}-\eqref{LD, non-LLR_alternative} imply \eqref{summability}, with $J_0=-I_0$ and $J_1=I_1$. \\

\noindent 2)  In Section \ref{sec:sufficient} we  state sufficient conditions for the existence of functions $\psi_0$ and $\psi_1$ that satisfy \eqref{LD, non-LLR_null}-\eqref{LD, non-LLR_alternative}, which we also specify.    In  Section \ref{sec:examples} we show that these sufficient conditions are satisfied in  various  testing problems and for different statistics. The graphs of $\psi_0$ and $\psi_1$  in each of these examples are plotted in Figures \ref{rate 1}, \ref{rate 2},   \ref{rate 3}. \\

\noindent 3) In the iid setup  of Subsection \ref{subsec: iid_setup}, the above   assumptions  hold  when  $\Stat=\bar\Lambda$ as long as \eqref{KL divergences are positive} holds
(see Subsection \ref{subsec:iid}). \\

\noindent 4)  By assumption, the function
\begin{equation} \label{def of g}
    g(\ka) \equiv  \frac{\psi_0(\ka)}{\psi_1(\ka)}, \quad \ka\in (J_0,J_1)
\end{equation}
 is continuous and strictly increasing with $
g(J_0+)=0$ and  $g(J_1-) = \infty$. As a result,  its inverse, $g^{-1}$,  is well-defined in $(0, \infty)$ and  satisfies
\begin{equation} \label{range of  g_inverse}
g^{-1} (0, \infty)=  (J_0,J_1). \\
\end{equation}

\noindent 5)   The above assumptions will suffice for obtaining first-order  asymptotic \emph{upper} bounds on the  expected sample sizes of  the proposed multistage tests under  $\Pro_0$ and $\Pro_1$ as $\alpha,\beta\to 0$. When $T=\bar\Lambda$, 
they will also suffice for obtaining the matching  \emph{lower} bounds.  However, in order to obtain such lower bounds when $T \neq \bar\Lambda$, we will  need to additionally assume that 
\begin{align} \label{extra assumption}
    \begin{split}
        & \exists \; \textit{a neighborhood of $J_1$ in which $\psi_0$ is finite and  \eqref{LD, non-LLR_null} holds} \\
&  \exists    \;  \textit{a neighborhood of $J_0$ in which $\psi_1$ is finite and  \eqref{LD, non-LLR_alternative} holds.}
    \end{split}
\end{align}
In Section \ref{sec:sufficient} we  also  state sufficient conditions  for \eqref{extra assumption}, which  hold for all test statistics, different from $\bar{\Lambda}$, that  we consider in  Section \ref{sec:examples}.\\

\subsection{Asymptotic analysis for the  
 fixed-sample-size test} \label{subsec: FSS}
    
The asymptotic analysis for the proposed  multistage tests is  based on asymptotic bounds and  approximations for $\n(\alpha,\beta)$ 
\textit{as at least one of $\alpha$ and $\beta$ goes to 0, while the other one either goes to 0 as well or remains fixed}. When  any of these asymptotic regimes holds, we  simply write $\alpha \wedge \beta \to 0$.

\subsubsection{Asymptotic bounds}

\begin{theorem} \label{th:asymptotic of FSS with t}
As   $\alpha \wedge \beta \to 0$,
\begin{equation} \label{Eq. ALB and AUB}
\min\left\{ \frac{|\log\beta|}{\psi_1(\ka)}, \, \frac{|\log\alpha|}{\psi_0(\ka)} \right\} \lesssim    \n(\alpha,\beta) \lesssim \max\left\{ \frac{|\log\beta|}{\psi_1(\ka)}, \, \frac{|\log\alpha|}{\psi_0(\ka)} \right\}
\end{equation}
for every $\ka \in (J_0, J_1)$,  and  consequently 
\begin{equation} \label{C}
  \n(\alpha,\beta) \lesssim \frac{|\log(\alpha \wedge \beta)|}{C},\quad \text{ where } \quad  C \equiv \sup_{\ka \in (J_0, J_1) }  \left\{ \psi_1(\ka) \wedge  \psi_0(\ka) \right\}.
\end{equation}
\end{theorem}

\begin{proof}
Appendix  \ref{app:A}. \\
\end{proof}

\myremark In the iid setup of Subsection \ref{subsec: iid_setup},  $C$  is  the well-known  \emph{Chernoff information} (see, e.g.,  \cite[Corollary 3.4.6]{Dembo_Zeitouni_LDPBook}).\\

We present  the following asymptotic lower bounds separately when  $T =\bar{\Lambda}$ and when $T \neq \bar{\Lambda}$, as in the latter case we also need  assumption \eqref{extra assumption}.

\begin{theorem} 
\label{thm: Universal ALB}
\begin{itemize}
\item[(i)]    If  $T= \bar\Lambda$, then
        \begin{equation} \label{Universal ALB, LLR}
            \n(\alpha,\beta) \gtrsim \max\left\{ \frac{|\log\beta|}{I_0},\, \frac{|\log\alpha|}{I_1} \right\} \quad \text{as} \quad \alpha \wedge \beta \to 0.
        \end{equation}
\item[(ii)]    If  $T\neq\bar\Lambda$ and \eqref{extra assumption} holds, then  
        \begin{equation} \label{Universal ALB, non_LLR}
            \n(\alpha,\beta) \gtrsim \max\left\{ \frac{|\log\beta|}{\psi_1(J_0)},\, \frac{|\log\alpha|}{\psi_0(J_1)} \right\} \quad \text{as} \quad \alpha \wedge \beta \to 0.
        \end{equation}
        \end{itemize}
\end{theorem}
\begin{proof}
  Appendix \ref{app:A}. 
\end{proof}

\subsubsection{Asymptotic approximations} 
Unlike the preceding  bounds,  asymptotic approximations to $\n(\alpha,\beta)$ depend on the  relative decay rate of $\alpha$ and $\beta$.    We start with the asymptotic regime where  $\alpha,\beta\to 0$ so that 
\begin{equation} \label{r}
|\log\alpha|\sim r\;  |\log\beta|
\quad \text{for some} \quad  r \in (0,\infty),
\end{equation}
in which case the approximation is expressed in terms of the function $g$, defined in \eqref{def of g}. 

\begin{corollary} \label{coro:ARE} 
As  $\alpha,\beta\to 0$ so that \eqref{r} holds, 
\begin{align}
     \n(\alpha,\beta)  \,  & \sim \,   \frac{|\log\alpha|}{ \psi_{0}(g^{-1}(r))}  \, \sim \,      \frac{|\log \beta|}{ \psi_{1}(g^{-1}(r))}. \label{ARE0, r} 
\end{align}
When in particular,  $r=1$,
\begin{align}
    & \n(\alpha,\beta)  \sim  \frac{|\log\alpha|}{C}  \sim  \frac{|\log\beta|}{C} , \label{ARE0, r=1}
\end{align}
where $C$ is defined in \eqref{C}. 
\end{corollary}

\begin{proof}
Appendix \ref{app:A}.\\
\end{proof}

\myremark   
From the previous corollary and  the optimal asymptotic performance in \eqref{optimal_rate} we  obtain the asymptotic relative efficiency  of the fixed-sample-size test   as  $\alpha,\beta\to 0$ so that \eqref{r} holds. Specifically, 
\begin{align}
     \n(\alpha,\beta)  \,  &  \sim \,  \frac{I_1}{ \psi_{0}(g^{-1}(r))} \, \cL_1(\alpha,\beta) 
\sim      \frac{I_0}{ \psi_{1}(g^{-1}(r))} \, \cL_0(\alpha,\beta), \label{ARE1, r}
\end{align}
and when in particular $r=1$,
\begin{align}
    & \n(\alpha,\beta) \sim\frac{I_1}{C} \, \cL_1(\alpha,\beta) \sim \frac{I_0}{C} \, \cL_0(\alpha,\beta). \label{ARE1, r=1}  
\end{align}
\vspace{0.2cm}

When  $\alpha \wedge \beta\to 0$ so that $|\log\alpha|/|\log\beta|$ either goes to zero or diverges,  the  asymptotic lower bounds in  Theorem \ref{thm: Universal ALB}   turn out to be sharp.

\begin{corollary} \label{coro: ARE when r=0 or infty, LLR}
Let $T=\bar\Lambda$. 
\begin{enumerate}
\item [(i)] 
If  \, $\alpha \wedge \beta\to 0$ so that  $|\log \alpha| <<  |\log\beta|$,  then  \,   $ \n(\alpha,\beta) \sim |\log \beta|/I_0$.
\item[(ii)] If \, $\alpha \wedge \beta\to 0$  so that  $|\log\alpha| >> |\log\beta|$,  then \,   $\n(\alpha,\beta) \sim  |\log \alpha|/ I_1.$
\end{enumerate}
\end{corollary}
\begin{proof}
Appendix \ref{app:A}
\end{proof}

\begin{corollary} \label{coro: ARE when r=0 or infty, non-LLR}
Let  $T\neq\bar\Lambda$ and assume that \eqref{extra assumption} holds. 
\begin{enumerate}
\item [(i)] If \,  $\alpha \wedge \beta\to 0$ so that  $|\log\alpha|<<|\log\beta|$,  then   \,   $ \n(\alpha,\beta) \sim |\log \beta|/\psi_1(J_0)$.
\item[(ii)]  If  \, $\alpha \wedge \beta\to 0$ so that  $|\log\alpha|>>|\log\beta|$,  then  \, $\n(\alpha,\beta) \sim  |\log \alpha|/ \psi_0(J_1)$.
\end{enumerate}
\end{corollary}

\begin{proof}
Appendix \ref{app:A}.\\
\end{proof}

\myremark   When  $T=\bar{\Lambda}$ and one of $\alpha$ and  $\beta$ is fixed,  Corollary \ref{coro: ARE when r=0 or infty, LLR}  is known as  \textit{Stein's lemma} (see, e.g., \cite[Lemma 3.4.7]{Dembo_Zeitouni_LDPBook}). We stress,  however, that both $\alpha$ and  $\beta$ may  go to 0  in the  previous corollaries. \\

When both $\alpha$ and  $\beta$ go  0, 
  Corollary \ref{coro: ARE when r=0 or infty, LLR}, in conjunction with  \eqref{optimal_rate}, implies that the  fixed-sample-size test   is asymptotically optimal under one of the two hypotheses, while being of larger order of magnitude compared to the optimal under the other hypothesis.  This is formalized in the following corollary.

   \begin{corollary} 
Let $T=\bar{\Lambda}$.
\begin{enumerate}
\item [(i)] 
If $\alpha,\beta\to 0$ so that  $|\log \alpha| <<  |\log\beta|$,  then 
\begin{equation*} 
    \cL_1(\alpha,\beta) << \n(\alpha,\beta)  \sim \cL_0(\alpha,\beta).
\end{equation*}
\item[(ii)] If  $\alpha,\beta\to 0$  so that  $|\log\alpha| >> |\log\beta|$,  then
\begin{equation*} 
    \cL_0(\alpha,\beta) << \n(\alpha,\beta) \sim \cL_1(\alpha,\beta).
\end{equation*}
\end{enumerate}
\end{corollary}

We end this subsection with the corresponding result when 
 $T \neq \bar{\Lambda}$.
 
   \begin{corollary} 
 Let   $T \neq \bar{\Lambda}$ for which \eqref{extra assumption} holds.
\begin{enumerate}
\item [(i)] If $\alpha,\beta\to 0$ so that  $|\log\alpha|<<|\log\beta|$,  then 
\begin{equation*}
    \cL_1(\alpha,\beta) << \n(\alpha,\beta)   \sim  
    \frac{I_0}{\psi_1(J_0)} \cL_0(\alpha,\beta).
\end{equation*}
\item[(ii)]  If $\alpha,\beta\to 0$ so that  $|\log\alpha|>>|\log\beta|$,  then
\begin{equation*}
    \cL_0(\alpha,\beta) << \n(\alpha,\beta) \sim  \frac{I_1}{\psi_0(J_1)} \cL_1(\alpha,\beta).
\end{equation*}
\end{enumerate}
\end{corollary}

\subsection{Asymptotic analysis for  multistage tests} \label{sec: robustness}
We now focus on the multistage tests we introduced in Section \ref{sec:multistage} and establish the main theoretical results of this work. We assume that  the test parameters are selected according to  Theorems \ref{th:3ST_error_control} and \ref{th:4ST_error_control}. However, unless otherwise specified, we do not require that the free parameters are selected as in Section \ref{subsec:3ST_free} and  
\ref{subsec:4ST_free}.

\subsubsection{An upper bound on the maximum sample size} \label{sec: upper bound in the middle}
By the definitions of the multistage tests  and the selection of their parameters according to  Theorems \ref{th:3ST_error_control} and \ref{th:4ST_error_control} it follows that, for any $\alpha, \beta \in (0,1)$ and any choice of the   free parameters,     
\begin{equation*} \label{pathwise_upper}
 \TST, \; \FST,\; \FFST \leq \n(\alpha/3,\beta/3),
\end{equation*}
 and consequently,   in view of  Theorem \ref{th:asymptotic of FSS with t}, 
\begin{equation} \label{non-robustness_SPRT}
\TST, \; \FST,\; \FFST \lesssim   \frac{
|\log (\alpha \wedge \beta)|}{C} \quad \text{as} \quad \alpha \wedge \beta\to 0.
\end{equation} 
On  the other hand, it is well known (see, e.g., \cite{Bechhofer60}) that, even when $X$ is an iid sequence, the  SPRT, defined in \eqref{Definition of SPRT}, not only does not have  bounded sample size, but even  its    \textit{expected} sample size can be much larger than  $\n(\alpha,\beta)$. 

To be specific,  consider a $\Pro \in \cP$, different from $\Pro_0$ and $\Pro_1$, under which  $\Lambda$  is a random walk   whose increments have  \textit{zero mean} and finite variance $\sigma^2$. The  expected sample size of the SPRT, with   $A=|\log\beta|$ and $B=|\log\alpha|$, under such a $\Pro$  is
\begin{equation} \label{Inflated ESS of SPRT in the middle}
    \Exp[\SPRT] \approx |\log \alpha | |\log \beta|/ \sigma^2,
\end{equation}
where  $\approx$ is  an equality when there is no  overshoot over the boundaries  (see, e.g.,  \cite[Chapter 3.1.1.2]{Tartakovsky_Book}).  Comparing  with  the upper bound in \eqref{non-robustness_SPRT}  suggests that  all  proposed multistage tests  will perform much better than the SPRT \textit{under such a $\Pro$  when  $\alpha$ and $\beta$  are small enough}. 
This robustness  of the proposed multistage tests is  illustrated in Figure \ref{Figure: performance against parameter}.

\subsubsection{Asymptotic analysis under $\Pro_0$ and $\Pro_1$} \label{sec:asy_multistage}

By the optimal asymptotic performance in \eqref{optimal_rate} it follows that, 
as $\alpha, \beta \to 0$,   
    \begin{align*}
        \Exp_1[\TST], \; \Exp_1[\FST], \; \Exp_1[\FFST] \; \gtrsim \; \frac{|\log\alpha|}{I_1} \quad \text{ and } \quad 
        \Exp_0[\TST], \; \Exp_0[\FST], \; \Exp_0[\FFST] \; \gtrsim \; \frac{|\log\beta|}{I_0},
    \end{align*}
for any selection of the free parameters and any choice of the test-statistic, $T$.  In the next lemma we obtain a sharper asymptotic lower bound when $T$ is not $\bar{\Lambda}$, but satisfies condition \eqref{extra assumption}.

\begin{lemma} \label{lem: ALB, multi}
    Suppose that $T\neq\bar\Lambda$ and \eqref{extra assumption} holds. Then,   for any  selection of the free parameters, as $\alpha, \beta \to 0$,  
    \begin{align*}
        \Exp_1[\TST], \; \Exp_1[\FST], \; \Exp_1[\FFST] \; \gtrsim \; \frac{|\log\alpha|}{\psi_0(J_1)} 
        \quad \text{ and } \quad 
        \Exp_0[\TST], \; \Exp_0[\FST], \; \Exp_0[\FFST] \; \gtrsim \; \frac{|\log\beta|}{\psi_1(J_0)}.
    \end{align*}
\end{lemma}

\begin{proof}
Appendix \ref{app:new}.
\end{proof}

We next state the main results of this section, according to which  the previous asymptotic lower bounds  are attained with an appropriate selection of the free parameters. 
To avoid repetition, we state these results  only when  $|\log\alpha|\gtrsim |\log\beta|$, as  analogous results  hold when $|\log\alpha|\lesssim |\log\beta|$.

\begin{theorem} \label{theorem, main}
Suppose that   $T=\bar\Lambda$ and let the free  parameters be selected according to   \eqref{free_3ST}, \eqref{free_4ST_hat}, \eqref{free_4ST_check}. 
\begin{enumerate}
    \item[(i)] If  $\alpha,\beta\to 0$ so that
$|\log\alpha| \gtrsim  |\log \beta|   $,
then 
    \begin{equation*}
        \Exp_1[\TST] \sim \Exp_1[\FST] \sim  \Exp_1[\FFST] \sim \frac{|\log\alpha|}{I_1} \sim \cL_1(\alpha,\beta).  
    \end{equation*}
    \item[(ii)]  If also   $
        |\log\alpha| \lesssim |\log\beta|/\beta^r$ for some  $r>0$, then 
    $$\Exp_0[\FST] \sim \frac{|\log\beta|}{I_0} \sim \cL_0(\alpha,\beta).$$
    \item[(iii)] If  also $
        |\log\alpha| \lesssim |\log\beta|^r$ for some $r\geq 1$, then 
    $$\Exp_0[\TST]\sim \Exp_0[\FFST] \sim \frac{|\log\beta|}{I_0} \sim \cL_0(\alpha,\beta).$$
\end{enumerate}
\end{theorem}

\begin{proof} 
Appendix \ref{app:new}.
\end{proof}

\begin{theorem} \label{theorem, main, non-LLR}
Suppose that $T\neq\bar\Lambda$,  \eqref{extra assumption} holds,  and let the free parameters  be selected according to   \eqref{free_3ST}, \eqref{free_4ST_hat}, \eqref{free_4ST_check}. 
\begin{enumerate}
 \item[(i)]   If   $\alpha,\beta\to 0$ so that 
$|\log\alpha| \gtrsim  |\log \beta|$, then
    $$ \Exp_1[\TST] \sim \Exp_1[\FST] \sim  \Exp_1[\FFST]  \sim  \frac{|\log\alpha|}{\psi_0(J_1)} \sim \frac{I_1}{\psi_0(J_1)}\cL_1(\alpha,\beta).$$
    \item[(ii)]  If  also 
$        |\log\alpha| \lesssim |\log\beta|/\beta^r$ for some $r>0$, then  
    \begin{equation*}
        \Exp_0[\FST]  \sim  \frac{|\log\beta|}{\psi_1(J_0)}  \sim \frac{I_0}{\psi_1(J_0)}\cL_0(\alpha,\beta).
    \end{equation*}
    \item[(iii)]  If  also 
$ |\log\alpha| \lesssim |\log\beta|^r $ for some $r \geq 1$, then 
    \begin{equation*} 
        \Exp_0[\TST] \sim \Exp_0[\FFST] \sim  \frac{|\log\beta|}{\psi_1(J_0)} \sim \frac{I_0}{\psi_1(J_0)}\cL_0(\alpha,\beta).
    \end{equation*}
 \end{enumerate}
\end{theorem}

\begin{proof}
Appendix \ref{app:new}.\\
\end{proof}

\noindent \underline{\textbf{Remarks}}:   1) As can be seen in the proof of Theorem  \ref{theorem, main, non-LLR}, 
 condition \eqref{extra assumption} is used only in Lemma \ref{lem: ALB, multi}, i.e., it is only needed for establishing the asymptotic lower bounds but not for obtaining the matching upper bounds. \\

\noindent  2) As can be seen from their proofs,  the above theorems hold  even if the free parameters  of the multistage tests are not  selected as suggested in Subsections \ref{subsec:3ST_free} and \ref{subsec:4ST_free}. 
Indeed, part (i) of each theorem  holds as long  as  $\delta \to 0$ and  $|\log\delta|<<|\log\alpha|$ as $\alpha \to 0$.
Similarly, part (ii) (resp. (iii)) of each theorem  holds  as long as the specification of   $\gamma,\gammap$ (resp. $\gamma$)  is such that  \eqref{main thm, condition on gamma, gamma'} (resp. \eqref{main thm, condition on gamma}) is satisfied. \\

\noindent 3) Part (i) in Theorems  \ref{theorem, main}  and \ref{theorem, main, non-LLR}  states that,  under the  \emph{alternative} hypothesis, 
  all multistage tests in this work achieve the optimal performance to a first-order asymptotic approximation when  $T=\bar\Lambda$, and have the same asymptotic relative efficiency when  $T\neq\bar\Lambda$ and \eqref{extra assumption} holds, as  $\alpha,\beta\to 0$ so that $|\log\alpha|\gtrsim |\log\beta|$. 
  On the other hand,  parts (ii) and (iii) 
imply that the corresponding results under the \emph{null} hypothesis hold \textit{as long as $\alpha$ does not go to 0 much faster than $\beta$},  and that this constraint is  much stricter  for $\three$ and $\ffour$ than for $\four$. This suggests that  $\four$ will perform much better than $\three$ and $\ffour$  under the \text{null} hypothesis when $\alpha$ is much smaller than   $\beta$. This insight  is  supported by Figures \ref{Figure: ratios} and \ref{Figure: performance against parameter}. \\

\noindent 3) Analogous results   hold  when  $\alpha,\beta\to 0$ so that $|\log\alpha|\lesssim |\log\beta|$. Indeed, under this asymptotic regime,  all three multistage tests  are asymptotically optimal  when  $T=\bar\Lambda$, and admit the same asymptotic relative efficiency when  $T\neq\bar\Lambda$ and \eqref{extra assumption} holds, under the \textit{null} hypothesis.  The corresponding results under the \textit{alternative} hypothesis hold as long as 
$\beta$ does not go to $0$ much faster than $\alpha$, with this requirement being much stricter for  $\three$ and $\four$ than for  $\ffour$. \\ 

\noindent 4)   The asymptotic optimality    under both hypotheses of the 3-stage test with $T=\bar{\Lambda}$ was established    in \cite[Section 2]{Lorden_1983}, in the iid setup of Subsection \ref{subsec: iid_setup}, as   $\alpha,\beta\to 0$ so that 
$$ |\log\beta|/r \lesssim|\log\alpha| \lesssim r\,|\log\beta| \;\text{ for some }\; r\geq 1. $$
Therefore, apart from extending it to a more general distributional setup, here we generalize this result   even in the iid case. Indeed, from  parts (i) and (iii) of  Theorem  \ref{theorem, main}  and the remark 3) we can conclude that  the asymptotic optimality of the 3-stage test  under both hypotheses  holds 
 as  $\alpha,\beta\to 0$ so that 
$$ |\log\beta|^{1/r} \lesssim |\log\alpha| \lesssim  |\log\beta|^r  \quad \text{ for some }\; r\geq 1.
$$
At the same time, we show how adding a stage  can further relax this asymptotic regime. Specifically,   from  Theorem  \ref{theorem, main}  and remark 3)
we can conclude  that the 4-stage test   $\four$ is asymptotically optimal  under both hypotheses as  $\alpha,\beta\to 0$ so that 
$$ |\log\beta|^{1/r}\lesssim |\log\alpha|\lesssim |\log\beta|/\beta^k, \quad \text{ for some }\; r\geq 1 \; \text{ and } \; k>0, $$
while the 4-stage test  $\ffour$ is asymptotically optimal  under both hypotheses as  $\alpha,\beta\to 0$ so that 
$$ |\log\alpha|^{1/r}\lesssim |\log\beta|\lesssim |\log\alpha|/\alpha^k, \quad \text{ for some }\; r\geq 1 \; \text{ and } \; k>0. $$ 

\noindent 5)   In view of   Theorem \ref{theorem, main, non-LLR}, in what follows  we  use the following notation for the  asymptotic relative efficiencies under $\Pro_0$ and $\Pro_1$,  as $\alpha, \beta \to 0$, of the proposed  multistage tests  when $T \neq \bar\Lambda$ and \eqref{extra assumption} holds, 
\begin{equation} \label{ARE}
\text{ARE}_0\equiv  \frac{\psi_1(J_0)}{I_0} \qquad \text{and} \qquad\text{ARE}_1 \equiv \frac{\psi_0(J_1)}{I_1},
 \end{equation}
  without further reference to  the relative decay rates of $\alpha$ and  $\beta$.

\section{Sufficient conditions} \label{sec:sufficient}
In this section we state  sufficient conditions for the existence of  functions $\psi_0, \psi_1$  that satisfy  \eqref{LD, non-LLR_null}--\eqref{LD, non-LLR_alternative}, which we also specify. To this end, we rely on the  \textit{G\"artner-Ellis theorem} from large deviation theory.  We start by stating  a  version of this  theorem  that focuses on events of form $(\ka,\infty)$ or $(-\infty,\ka)$, where $\ka\in \bR$, and  requires somewhat weaker  conditions compared to standard  formulations in the literature, such as  \cite[Theorem 2.3.6]{Dembo_Zeitouni_LDPBook} or \cite[Theorem 3.2.1]{Bucklew_Book}.

\subsection{The G\"artner-Ellis theorem}
\label{subsec: G-E}
In this subsection we consider  an arbitrary $\Pro \in \cP$ and  for every 
$ \theta\in \bR$   we set   
    $$ \phi_n(\theta)\equiv \frac{1}{n} \log \Exp\left[ \exp\{n\, \theta\,  T_n\} \right],
    \quad n \in \bN,
    $$ 
and  assume that 
    \begin{equation*}
        \phi(\theta)\equiv \lim_n \, \phi_n(\theta) \quad\text{exists in} \quad (-\infty,\infty].
    \end{equation*}
We  denote by $\Theta$ the effective domain of   $\phi$, i.e., 
$\Theta\equiv \{\theta\in \bR: \; \phi(\theta)<\infty\},$
and by    $\phis$  its Legendre-Fenchel transform:
    \begin{equation} \label{def of phis}
        \phis(\ka)\equiv \sup_{\theta\in \bR} \{\theta\ka-\phi(\theta)\}, \quad \ka\in \bR.
    \end{equation}
We further assume that 
 $\Theta^o \neq\emptyset$,   and that
    \begin{equation*} 
     \phi  \text{ is \textit{strictly} convex and continuous in } \Theta \; 
     \text{and differentiable in }  \Theta^o. 
    \end{equation*}
This assumption implies that  $\phi'$ is strictly increasing in $\Theta^o$,  that $\phi' (\Theta^o)$ is a non-trivial open interval, and as a result that 
    \begin{equation*}
        \phis(\ka)=\vartheta(\ka) \, \ka-\phi(\vartheta(\ka)) \quad \forall \; 
        \ka\in \phi'(\Theta^o),
    \end{equation*}
where  $\vartheta$ is    the inverse of  $\phi'$ in  $\Theta^o$.

Finally,  we  assume that for every   $\theta\in  \Theta$  there exists a (unique) distribution of $X$,  $\Qro_\theta$, such that 
\begin{equation} \label{exponential_tilting}
\frac{d\Qro_\theta}{d\Pro}(\cF_n) = \exp\{n \, ( \theta \,  \Stat_n-\phi_{n}(\theta) ) \}, \quad \forall \; n \in \bN.
\end{equation}  
This is known as an \textit{exponential tilting} of $\Pro$, and for  its  existence it suffices, for example,  that $\bS$ be Polish (see, e.g., \cite[p. 144, Theorem 5.1]{MeasureOnMetricSpace}).  \\

\begin{theorem} \label{th:A version of G-E}
Suppose that the above assumptions hold. 
  \begin{itemize}
        \item [(i)] If \, $\Theta^o\cap(0,\infty)\neq\emptyset$, then $\phis(\phi'(0+))=0$, $\phis$ is strictly increasing in $\phi'(\Theta^o\cap(0,\infty))$ and, for every $\ka\in \phi'(\Theta^o\cap(0,\infty))$,
        \begin{equation} \label{>ka}
            \lim_n \,\frac{1}{n}\log \Pro\left( T_n>\ka \right)=-\phis(\ka).
        \end{equation}
        \item [(ii)] If \,  $\Theta^o\cap(-\infty,0)\neq\emptyset$, then $\phis(\phi'(0-))=0$, $\phis$ is strictly decreasing in $\phi'(\Theta^o\cap(-\infty,0))$ and, for every $\ka\in \phi'(\Theta^o\cap(-\infty,0))$, 
        \begin{equation} \label{<=ka}
            \lim_n \,\frac{1}{n}\log \Pro\left( T_n\leq \ka \right)=-\phis(\ka).
        \end{equation}
\item[(iii)] For every $\theta \in \Theta^o$,\; 
$ \Qro_{\theta} \left( T_{n} \to \phi'(\theta) \right)=1.$
    
    \end{itemize}
\end{theorem}
\begin{proof}
 Appendix \ref{app: proof from Assumptions B to Assumptions A}. \\
\end{proof}

\myremark 1) Theorem \ref{th:A version of G-E}  implies that, for any $\epsilon>0$,
 $\Pro\left(T_n-\phi'(0+)>\epsilon\right)$ decays exponentially fast in $n$  if  $\Theta^o$ intersects $(0, \infty)$, and  $\Pro(T_n- \phi'(0-) \leq -\epsilon)$ decays exponentially fast  in $n$  if $\Theta^o$ intersects $(-\infty,0)$.  \\

\noindent 2) In standard formulations of the  G\"artner-Ellis theorem, such as   \cite[Theorem 2.3.6]{Dembo_Zeitouni_LDPBook} or \cite[Theorem 3.2.1]{Bucklew_Book},
it is additionally assumed that $ 0\in \Theta^o$, 
in which case the conditions in both (i) and (ii) of Theorem  \ref{th:A version of G-E}  hold, 
$\phi'(0)$ exists, and thus
 $\Pro(|T_n-\phi'(0)|>\epsilon)$ decays exponentially fast in $n$ for any $\epsilon>0$, and  $\Pro(T_n \to \phi'(0))=1$.
It is also  assumed that  $\phi$ is  \emph{steep}, i.e.,  $\phi'(\Theta^o)=\bR$,  (see, e.g., \cite[Definition 2.3.5]{Dembo_Zeitouni_LDPBook}), in which case $$\phi'(\Theta^o\cap(0,\infty))=(\phi'(0),\infty) \qquad \text{and} \qquad \phi'(\Theta^o\cap(-\infty,0))=(-\infty,\phi'(0)).$$

\subsection{Sufficient conditions for the asymptotic theory of Section \ref{sec:asy} }
We next apply Theorem \ref{th:A version of G-E} to establish sufficient conditions for the asymptotic theory of Section \ref{sec:asy}. To this end,  when the assumptions of Subsection \ref{subsec: G-E}  hold for $\Pro=\Pro_i$, where  $i \in \{0,1\}$,   we write  $\phi_{i, n}, \phi_i, \Theta_i, \phis_i, \vartheta_i$  instead of 
$\phi_n,  \phi, \Theta, \phis, \vartheta$ and, 
for each $\theta \in \Theta_i^o$,   we denote by  $\Qro_{i,\theta}$ the exponential tilting of $\Pro_i$, i.e.,   
\begin{equation} \label{Lambda_n(Q_theta,P_0)}
  \frac{d\Qro_{i,\theta}}{d\Pro_0}(\cF_n)   = \exp\left\{ n \, ( \theta \,  \Stat_n-\phi_{i,n}(\theta) ) \right\}, \quad \forall \; n \in \bN. 
\end{equation}

\begin{corollary} \label{sufficient condition, non-LLR}
Suppose \eqref{a.s. limits of eta} holds 
for some $J_0, J_1 \in \bR$, with  $J_0<J_1$. 
\begin{enumerate}
\item[(i)] If  the assumptions of Subsection \ref{subsec: G-E}  hold for  $\Pro=\Pro_0$ and 
\begin{align} \label{cond1}
\begin{split} 
\Theta_0^o\cap (0,\infty)\neq \emptyset,   \quad \phi'_0(0+)=J_0, \\
\text{and} \quad \exists\; \theta_0\in \Theta_0\cap (0,\infty): \;
\phi'_0(\theta_0-)=J_1,
 \end{split}
\end{align}
then  \eqref{LD, non-LLR_null}   holds, with $\psi_0=\phis_0$,  for every $\ka \in (J_0, J_1)$. 
If also    $ \theta_0\in \Theta_0^o$, then \eqref{LD, non-LLR_null}   holds,   with $\psi_0=\phis_0$,  in a neighborhood of $J_1$.  \\

\item[(ii)]  If the assumptions of Subsection \ref{subsec: G-E}  hold for  $\Pro=\Pro_1$ and 
\begin{align} \label{cond2}
\begin{split}
\Theta_1^o\cap (-\infty,0)\neq \emptyset, \quad \phi'_1(0-)=J_1, \\
\text{and} \quad  \exists\; \theta_1\in \Theta_1\cap (-\infty,0): \;
 \phi'_1(\theta_1+)=J_0,
 \end{split}
\end{align}
then  \eqref{LD, non-LLR_alternative}  holds, with  $\psi_1=\phis_1$,  for every $\ka \in (J_0, J_1)$.    If also  $\theta_1\in \Theta_1^o$,
then   \eqref{LD, non-LLR_alternative}  holds,
  with  $\psi_1=\phis_1$,  in a neighborhood of   $J_0$.  
 
\item[(iii)] 
For $i\in \{0,1\}$, if the assumptions of Subsection \ref{subsec: G-E}  hold for $\Pro=\Pro_i$, then
$$ \Qro_{i,\theta} \left( T_{n} \to \phi'_i(\theta) \right)=1 \quad \forall \;  \theta\in \Theta_i^o.
$$ 
\end{enumerate}
\end{corollary}

\begin{proof}
We only prove (i), as the proof of (ii) is similar, whereas that of  (iii) follows directly from Theorem \ref{th:A version of G-E}.(iii).
 Since $\phi'_0(\Theta_0^o)$ is, by assumption, an  open interval, \eqref{cond1} implies that  
\begin{equation} \label{inclusion 1}
    (J_0,J_1)\subseteq \phi'_0\left(\Theta_0^o\cap(0,\infty)\right),
\end{equation}
and the first claim in (i) follows by an application of Theorem \ref{th:A version of G-E}.(i).  

    If also $\theta_0\in \Theta_0^o$, \eqref{cond1}  implies that 
    \begin{equation} \label{inclusion 2}
        (J_0,J_1]\subseteq \phi'_0(\Theta_0^o\cap(0,\infty)),
    \end{equation}
and  the second claim in (i) follows again by an application of Theorem \ref{th:A version of G-E}.(i). 
\end{proof}

\begin{corollary} \label{further sufficient condition, non-LLR}
Suppose that  the assumptions of Subsection \ref{subsec: G-E}   hold for both  $\Pro=\Pro_0$ and $\Pro=\Pro_1$. 
\begin{enumerate}
\item[(i)] If  $   0\in \Theta_0^o \cap \Theta_1^o$ and   $\phi'_0(0)<\phi'_1(0)$,  then
 \eqref{a.s. limits of eta} holds with  $J_i=\phi'_i(0)$, $i=0,1$.

\item[(ii)] If   also both $\phi_0$ and $\phi_1$ are steep, then 
\eqref{LD, non-LLR_null} holds, with $\psi_0=\phis_0$, for every $\ka>J_0$, and \eqref{LD, non-LLR_alternative} holds, with $\psi_1=\phis_1$, for every $\ka<J_1$. 
\end{enumerate}
\end{corollary}

\begin{proof}    
This  is a direct consequence of the remark following Theorem 
\ref{th:A version of G-E}.
\end{proof}

\myremark In Section \ref{sec:examples} we show that the  assumptions of Corollary   \ref{further sufficient condition, non-LLR} are satisfied in various examples. However,  Corollary \ref{sufficient condition, non-LLR} implies  that, when $T=\bar{\Lambda}$,   for the  asymptotic theory of Section \ref{sec:asy} to  apply,  it suffices  that   \eqref{cond1}-\eqref{cond2} hold,  and the latter can be true  \textit{even if $0$ is not in the interior of either  $\Theta_0$ or $\Theta_1$}.  We explore this point in more detail next.

\subsection{The likelihood ratio case} \label{subsec: the likelihood ratio case}
In what follows, we focus on the case where $T=\bar\Lambda$ and  the assumptions of Subsection \ref{subsec: G-E}  hold for  $\Pro=\Pro_0$. Then, in view of the fact that 
\begin{equation*}
   \Exp_1[\exp\{\theta\Lambda_n\}] =   \Exp_0[\exp\{(\theta+1)\Lambda_n\}],  \quad \forall \; n \in \bN, \quad \theta \in \bR, 
\end{equation*}
 the assumptions of Subsection \ref{subsec: G-E} also  hold for  $\Pro=\Pro_1$, with 
\begin{align} 
    \phi_1(\theta)&=\phi_0(\theta+1), \quad \theta\in \bR, \label{LLR, phi_i} \\
    \Theta_1&=\Theta_0-1,  \label{LLR, phi_iii}\\
    \phis_1(\ka)&=\phis_0(\ka)-\ka, \quad \ka\in \bR. \label{LLR, phis_i}
\end{align}
From \eqref{LLR, phi_i} it follows that 
 $1$ is  the non-zero root of $\phi_0$, and as a result that $  [0,1]\subseteq \Theta_0$, since 
  $\Theta_0$ is an interval.    Since  also  $\phi_0$ is strictly convex and continuous  in $[0,1]$,  and differentiable in $(0,1)$,  we conclude that 
 \begin{align*}
& -\infty< \phi'_0(0+)<0<\phi'_0(1-)<\infty.
 \end{align*}
 From \eqref{LLR, phi_i} and \eqref{LLR, phi_iii} it similarly follows that   $-1$ is the non-zero root of   $\phi_1$, 
$[-1,0]\subseteq \Theta_1$, and
 \begin{align*}
 &-\infty<\phi'_1(-1+)<0<\phi'_1(0-)<\infty.
 \end{align*}
 Based on these observations, we can see that  the  conditions of Corollary \ref{sufficient condition, non-LLR}
 simplify considerably.

\begin{corollary} \label{coro: LLR, sufficient conditions}
If $T=\bar\Lambda$,  the assumptions of Subsection \ref{subsec: G-E} hold  for  $\Pro=\Pro_0$, and 
\eqref{SLLN on LLR} holds with $I_0=-\phi'_0(0+)$ and $I_1=\phi'_0(1-),$
then \eqref{LD, non-LLR_null} and \eqref{LD, non-LLR_alternative} hold for every $\ka\in (-I_0,I_1)$ with $\psi_0=\phis_0$ and $\psi_1=\phis_1$, respectively. Moreover, 
\begin{equation} \label{LLR, C}
    C=\phis_0(0) =\phis_1(0)   =-\inf_{\theta\in \bR}\phi_0(\theta),
\end{equation}
where $C$ is defined in \eqref{C}.
\end{corollary}

\begin{proof}
From the discussion prior to statement of Corollary
\ref{coro: LLR, sufficient conditions} it follows that 
the conditions of  Corollary \ref{sufficient condition, non-LLR} are satisfied. To show \eqref{LLR, C}, we note that the supremum in the definition of $C$ in \eqref{C} is attained when $\psi_0=\psi_1$, or equivalently when $\phi_0^*=\phi_1^*$.   Comparing with \eqref{LLR, phis_i} completes the proof.\\
\end{proof}

\myremark  Suppose that   $T=\bar\Lambda$ and that  the assumptions of Subsection \ref{subsec: G-E} hold  for  $\Pro=\Pro_0$. Then, from Corollary  \ref{further sufficient condition, non-LLR}  it follows that a sufficient condition for  \eqref{SLLN on LLR}  to hold, with 
$I_0=-\phi'_0(0)$ and $I_1=\phi'_0(1)$, is that   $\{ 0, 1\} \subset \Theta^o_0$. 
However,  as we mentioned earlier,  the assumptions of  Corollary
\ref{coro: LLR, sufficient conditions}  may hold  even when $\Theta_0=[0,1]$, in which case   $\Theta_1=[-1,0]$ and  $(-I_0, I_1)=\phi'_i(\Theta_i^o)$, $i\in\{0,1\}$.  
The importance of this observation becomes clear in the iid setup, on which we focus next.

\subsubsection{The iid setup} \label{subsec:iid}
We end this section by showing that the conditions of Corollary \ref{coro: LLR, sufficient conditions}  are satisfied in  the iid setup of Subsection \ref{subsec: iid_setup}  as long as the Kullback-Leibler divergences defined in \eqref{KL divergences are positive} are positive and finite, or equivalently,  the expectation of 
  $\Lambda_1=\log \left(f_1(X_1) / f_0(X_1) \right)$ is non-zero and finite under both $\Pro_0$ and $\Pro_1$.

Indeed, in this case,  \eqref{SLLN on LLR} holds  with $I_0=D(f_0\| f_1)$ and $I_1= D(f_1\| f_0)$ by Kolmogorov's Strong Law of Large Numbers, and clearly
$$\phi_0(\theta)=\log\Exp_0[\exp\{\theta\Lambda_1\}], \quad \theta\in \bR.$$
Since $\phi_0$ is the cumulant generating function of a non-degenerate distribution and $[0,1]\subseteq\Theta_0$, $\phi_0$ is strictly convex in $\Theta_0$, differentiable in $\Theta_0^o$, continuous at 0 and 1, and satisfies
$$\phi'_0(0+)=\Exp_0[\Lambda_1]=-I_0\quad \text{ and } \quad \phi'(1-)=\frac{\Exp_0[\Lambda_1\exp\{\Lambda_1\}]}{\Exp_0[\exp\{\Lambda_1\}]}=I_1$$
(see, e.g. \cite[Excercise 2.2.24]{Dembo_Zeitouni_LDPBook}).

\section{Implementation via importance sampling} \label{sec: IS}
The  proposed designs for  the multistage tests in Section \ref{sec:formulation} require knowledge of  the functions
$\n$   and $\thre$,   defined in  \eqref{n*(alpha,beta)}. These  do not  admit,  
in general, closed-form expressions and need to be approximated. For any given 
$\alpha$ and $\beta$ in $(0,1)$,   $\n(\alpha,\beta)$ and  $\thre(\alpha,\beta)$  can be approximated by  estimating   $\Pro_0(\Stat_n>\ka)$ and   $\Pro_1(\Stat_n\leq \ka)$ for different $n$ and $\ka$,  and finding the  minimum  $n$ for which there exists a $\ka$ so that the first probability does not exceed $\alpha$ and  the second does not exceed  $\beta$.

  If  it is convenient to simulate the sequence $X$ under $\Pro_0$ and $\Pro_1$, 
a simple method for the estimation of   $\Pro_0(\Stat_n>\ka)$ and   $\Pro_1(\Stat_n\leq \ka)$  is plain Monte-Carlo simulation.  However,  when these probabilities are very small,  this approach may  not be efficient, or even feasible. Indeed, if the probability of interest is $10^{-a}$ for some $a>0$,  the minimum number of simulation runs  needed
for the \textit{relative error} of the  Monte-Carlo estimator to be at most $1\%$ is $10^{a+4}$. Therefore, when the probability of interest is very small, a different  method may need to be applied for its estimation, such as \textit{importance sampling} \cite{Bucklew_Book}.

To illustrate this method,   we focus on the estimation of  $\Pro_0(\Stat_n>\ka)$, as a completely analogous discussion applies to the estimation of   $\Pro_1(\Stat_n\leq \ka)$.  We observe that if  $\Qro$   is a distribution of $X$
that is  mutually absolutely continuous with   $\Pro_0$ on $\cF_n$ for every $n \in \bN$, then 
$ \Pro_0(\Stat_n>\ka)  = \Exp_\Qro\left[   Z_{n,\ka}(\Qro)  \right]$, where 
\begin{align} \label{IS_identity}
  Z_{n,\ka}(\Qro) & \equiv \frac{d\Pro_0}{d\Qro}(\cF_n) \cdot   1\{ \Stat_n>\ka \}
\end{align}
and  $\Exp_\Qro$ denotes expectation under $\Qro$. Thus, if  it is possible to simulate  $X$ under $\Qro$,   $\Pro_0(\Stat_n>\ka) $ can be estimated  by  averaging
    $Z_{n,\ka}(\Qro)$ over a  large number of independent realizations of $X$ in which it is distributed according to $\Qro$.

The question then  is how to select the \textit{importance sampling distribution} $\Qro$, so that  the \text{relative error}
    of the induced  estimator is small even when  $\Pro_0(\Stat_n>\ka) $  is  small.   To answer it,  we   assume that the  assumptions of Corollary  \ref{sufficient condition, non-LLR}.(i) (resp. Corollary
\ref{coro: LLR, sufficient conditions}) hold when $T\neq \bar{\Lambda}$ (resp.  $T=\bar{\Lambda}$) and fix  $\ka$ in $(J_0,J_1)$ (resp. $(-I_0, I_1)$), in which case   $\Pro_0(T_n>\ka)$ decays exponentially fast in $n$.  Then,  squaring both sides in \eqref{IS_identity},  applying the  Cauchy-Schwarz inequality,
taking logarithms on both sides, dividing by $n$, letting $n\to \infty$, and applying \eqref{LD, non-LLR_null}, we obtain
    \begin{equation}
      \underset{n}{\underline\lim}\, \frac{1}{n} \log \Exp_\Qro \left[  Z^2_{n,\ka}(\Qro) \right]\geq -2 \, \psi_0(\ka).
    \end{equation}
The latter is essentially a universal asymptotic lower bound on the variance of any importance sampling estimator.  As  it is common in the relevant  literature (see, e.g., \cite[Chapter 5]{Bucklew_Book}), we  refer to   $\Qro$ as  \textit{logarithmically efficient} for the estimation of $\Pro_0(T_n>\ka)$   if it attains this asymptotic  lower bound, i.e., if 
 \begin{equation} \label{logarithmically efficient}
      \underset{n}{\overline\lim}\, \frac{1}{n} \log \Exp_\Qro \left[   Z^2_{n,\ka}(\Qro)  \right] \leq -2 \, \psi_0(\ka). 
    \end{equation}

Recalling the definition of the exponential tilting $\Qro_{0,\theta}$ in \eqref{exponential_tilting},  for every $n \in \bN$ and $\theta \in \Theta_0^o$ we have
    \begin{align*}
      \Exp_{\Qro_{0,\theta}}\left[  Z^2_{n,\ka}(\Qro_{0,\theta})  \right] &= \Exp_{\Qro_{0, \theta} }\left[ \exp\{-2n \left( \theta\, \Stat_n-\phi_{0,n}(\theta) \right) \};\, \Stat_n>\ka \right] \\
     & \leq \exp\{ -2n (\theta \,  \ka-\phi_{0,n}(\theta) )\}.
    \end{align*}
Taking logarithms, dividing by $n$ and letting $n\to \infty$ we obtain
    \begin{align*}
   \underset{n}{\overline\lim} \, \frac{1}{n} \log   \Exp_{\Qro_{0,\theta}} \left[  Z^2_{n,\ka}(\Qro_{0, \theta})  \right] 
    \leq -2(\theta \, \ka -\phi_{0}(\theta)) .
    \end{align*}
Therefore, when $\theta=\vartheta_0(\ka)$, where $\vartheta_0$ is the inverse function of $\phi'_0$,  the right-hand-side is equal to  $-2\psi_0(\ka)$, which proves that 
$\Qro_{0, \vartheta_0(\ka)}$  is  logarithmically efficient for the estimation of  $\Pro_0(T_n> \ka)$.

Working similarly, we can see that if  the assumptions of Corollary  \ref{sufficient condition, non-LLR}.(ii) (resp. Corollary
\ref{coro: LLR, sufficient conditions}) hold when $T\neq \bar{\Lambda}$ (resp.  $T=\bar{\Lambda}$),   a logarithmically efficient importance sampling distribution for the estimation of 
 $\Pro_1(T_n\leq \ka)$  when $n$ is large 
 is  $\Qro_{1, \vartheta_1(\ka)}$, where 
 $\vartheta_1$ is the inverse function of $\phi'_1$.  In Subsection \ref{subsec: expon_family} 
we present an example where $\Qro_{0, \vartheta_0(\ka)}$ and $\Qro_{1, \vartheta_1(\ka)}$  coincide. 

Finally, we observe that by  Corollary \ref{sufficient condition, non-LLR}.(iii) it follows  that 
$$\Qro_{i, \vartheta_i(\ka)}( T_n \to \ka  )= 1, \quad i \in \{0,1\}.$$
This  suggests that if it is not convenient to simulate $X$ under  the logarithmically efficient importance sampling distributions,
 a   potential strategy  for   estimating   $\Pro_0(T_n > \ka)$ and  $\Pro_1(T_n\leq \ka)$,   simultaneously, is to  apply  importance sampling using a distribution  under which \textit{it is convenient to simulate $X$ and  $T_n$ converges almost surely  to $\ka$ as $n \to \infty$}.  
We apply this  strategy successfully in two non--iid testing problems  in Section \ref{sec:numerical}.

\section{Examples} \label{sec:examples}
In this section we focus on  three concrete  testing problems,  with which we illustrate  the  general results of the previous sections.
Specifically, for each of these testing problems we show that the conditions of Subsection \ref{subsec: assumptions on the test prob} hold, and also that the conditions of  Subsection \ref{subsec: assumptions on the test stat} hold for   $T=\bar{\Lambda}$, as well as  for   an alternative test statistic. For the latter, we also compute the induced asymptotic relative efficiency, defined in \eqref{ARE}.


\subsection{Testing in a one-parameter exponential family} \label{subsec: expon_family}

In the first example of this section we let $h$ be  a density with respect to  a $\sigma$-finite measure $\nu$ on $\bS$ such  that  $M \neq \emptyset$, where
\begin{equation} \label{exp family} 
M \equiv \{\mu\in \bR: \varphi(\mu)<\infty\}^o,
\quad \quad \varphi(\mu)\equiv \log\int_{\bS}  e^{\mu x} \, h(x)\,  \nu (dx)
\end{equation} 
and, for each $\mu \in M$,   we set 
$$h_\mu(x) \equiv h(x) \;  e^{\mu x - \varphi(\mu)}, \quad x \in \bS,
$$  noting that $h_\mu$ is also a density with respect to $\nu$, with the same support as $h$. We  denote by $\bP_\mu$  the distribution of $X$, and by $\bE_\mu$ the corresponding expectation,   when $X$ is a  sequence of independent  random elements with common density $h_\mu$, and consider the testing setup of Subsection \ref{subsec: one-sided testing problem}. In this context, the log-likelihood ratio statistic in \eqref{LLR_iid_setup}  becomes 
\begin{equation} \label{llr in exp family}
    \Lambda_n=(\mu_1-\mu_0) \sum_{i=1}^n X_i-n\,  ( \varphi(\mu_1)-\varphi(\mu_0)), \quad n\in \bN,
\end{equation}
and, for each $\mu \in M$,   it is a random walk under $\bP_\mu$  with drift
\begin{equation}\label{mean} \bE_\mu[\Lambda_1]= (\mu_1-\mu_0) \, \varphi'(\mu)-  ( \varphi(\mu_1)-\varphi(\mu_0)). 
\end{equation}
Thus,  setting $\mu$ equal to $\mu_0$ and $\mu_1$, we obtain the following expressions for the Kullback-Leibler divergences  in \eqref{KL divergences are positive}:
\begin{equation*}
    \begin{split}
        & D(f_0|| f_1)=-\big( (\mu_1-\mu_0) \, \varphi^\prime(\mu_0) - \left(\varphi(\mu_1)-\varphi(\mu_0)\right) \big) \\
        & D(f_1|| f_0)= (\mu_1-\mu_0) \, \varphi^\prime(\mu_1) - \left(\varphi(\mu_1)-\varphi(\mu_0)\right).
    \end{split}
\end{equation*}
Since these are positive and finite, by  the discussion in  Subsection \ref{subsec:iid} it follows that all assumptions in Subsections  \ref{subsec: assumptions on the test prob}-\ref{subsec: assumptions on the test stat}   hold with 
\begin{align} \label{psis}
\begin{split}
I_0 &=D(f_0\|f_1),\quad I_1=D(f_1\|f_0), \quad  C=\psi_0(0),\\
        \psi_0(\ka) &=\vartheta_0(\ka)\ka-\phi_0(\vartheta_0(\ka)), \quad \forall \; 
        \ka\in (-I_0, I_1) \\
        \psi_1(\ka) &=\vartheta_1(\ka)\ka-\phi_1(\vartheta_1(\ka)), \quad \forall \; 
        \ka\in (-I_0, I_1),
         \end{split}
\end{align} 
where 
  $\vartheta_i$ is  the inverse  of $\phi'_i$, $i \in \{0,1\}$, and 
\begin{align} \label{phis}
\begin{split}
 \phi_{0}(\theta) &=  \varphi(\mu_0+\theta(\mu_1-\mu_0))- \left( \varphi(\mu_0)+\theta(\varphi(\mu_1)-\varphi(\mu_0)) \right),  \; \theta \in [0,1] \\
 \phi_{1}(\theta) &=  \varphi(\mu_1+\theta(\mu_1-\mu_0))- \left( \varphi(\mu_1)+\theta(\varphi(\mu_1)-\varphi(\mu_0)) \right),  \; \theta \in [-1,0].
 \end{split}
\end{align} 

As a result,  in this context, the asymptotic optimality of the proposed multistage tests holds when  $T=\bar \Lambda$. In fact, it  also holds when 
\begin{equation} \label{average} 
T=\bar{X}\equiv \{\bar{X}_n, n \in \bN\}, \quad \text{where} \quad  \bar X_n\equiv \frac{1}{n} \sum_{i=1}^n X_i, \quad n\in \bN.
\end{equation}
Indeed,  from  \eqref{llr in exp family} it follows that when  $T=\bar{X}$,  then   for any $\mu_0, \mu_1\in M$ we have 
\begin{equation} \label{nice linear transformation}
     \bar\Lambda_n=(\mu_1-\mu_0) \, T_n - (\varphi(\mu_1)-\varphi(\mu_0)), \quad n\in\bN,
\end{equation}
which means that   the values of  $\n(\alpha, \beta)$ and 
$\kappa^*(\alpha, \beta)$, which in general depend on the choice of the test statistic $T$,  coincide when $T=\bar X$  and  $T=\bar  \Lambda$.

\subsubsection{Importance sampling distributions}
In this setup, it is  convenient to obtain an explicit form for  the  logarithmically efficient importance sampling distributions for the estimation of $\Pro_0(\bar \Lambda_n >\ka)$ and 
$\Pro_1(\bar\Lambda_n\leq \ka)$ when 
$T=\bar{\Lambda}$ for any  $\ka\in (-I_0,I_1)$.  Indeed, for any   $\ka\in (-I_0,I_1)$ we have:
$$\Qro_{0, \vartheta_0(\ka)} =\Qro_{1, \vartheta_1(\ka)}=   \bP_\mu,$$
 where  $\mu \in (\mu_0, \mu_1)$ is such that $\bE_\mu[\Lambda_1]=\ka$.
To prove this statement, we first note that for any $n \in \bN$ and $\theta \in (0,1)$, 
by \eqref{llr in exp family} we have
\begin{align*}
   \Lambda_n\left(\bP_{\mu_0+\theta(\mu_1-\mu_0)}, \Pro_0\right) &= \Lambda_n\left(\bP_{\mu_0+\theta(\mu_1-\mu_0)}, \bP_{\mu_0} \right) \\
  & =  \theta(\mu_1-\mu_0)\sum_{i=1}^n X_i - n\left( \varphi(\mu_0+\theta(\mu_1-\mu_0))-\varphi(\mu_0) \right) \\
  &= n \left( \theta\bar\Lambda_n-\phi_0(\theta) \right),
\end{align*}
and similarly, for any $n \in \bN$ and $\theta \in (-1,0)$, 
\begin{align*}
   \Lambda_n\left(\bP_{\mu_1+\theta(\mu_1-\mu_0)}, \Pro_1\right)
  &= n \left( \theta\bar\Lambda_n-\phi_1(\theta) \right).
\end{align*}
Therefore, the exponential tiltings of $\Pro_0$ and $\Pro_1$, defined in \eqref{Lambda_n(Q_theta,P_0)},  are given  by  
 \begin{align*} 
  \Qro_{0, \theta} &=\bP_{\mu_0+\theta(\mu_1-\mu_0)}, \quad \theta \in (0,1), \\
 \Qro_{1, \theta} &=\bP_{\mu_1+\theta(\mu_1-\mu_0)}, \quad \theta \in (-1,0). 
  \end{align*}
 Differentiating the identities in \eqref{phis} and comparing with  \eqref{mean} we obtain 
\begin{align} \label{deriv_expect}
\begin{split}
\bE_{\mu_0+\theta( \mu_1-\mu_0)}[\Lambda_1] & = \phi'_0(\theta), \quad \theta \in (0,1),
\\
\bE_{\mu_1+\theta( \mu_1-\mu_0)}[\Lambda_1] & = \phi'_1(\theta), \quad \theta \in (-1,0).
\end{split}
\end{align}
The statement now follows by the definition of $\vartheta_i$ as the inverse of $\phi'_i$, where $i \in \{0,1\}$.

\subsubsection{A binary statistic}
An approach to the testing problem of this subsection, which can be  motivated by  practical constraints or robustness considerations,  is to binarize the data,  recording only whether each observation is larger, or not, than some  user-specified value in the interior of the support of $h$, say $x_*$. 
Then,  the test statistic can be written as 
\begin{equation}\label{Z}
T=\bar{Z}\equiv \{\bar{Z}_n, n \in \bN\}, \quad 
\bar Z_n\equiv \frac{1}{n} \sum_{i=1}^n Z_i, 
\quad \quad Z_i \equiv  1\{X_i>x_*\}, \quad n\in \bN,
\end{equation}
and all assumptions in Subsection \ref{subsec: assumptions on the test stat}, including  \eqref{extra assumption}, are satisfied with 
\begin{gather*}
    J_i=\Pro_i(X_1>x_*), \quad  C=-\log\sqrt{4 \, J_0 \, J_1}, \\
    \phi_i(\theta)= \log\left( J_i\, e^\theta + (1-J_i) \right), \quad \theta\in \bR, \\
    \psi_i(\ka)=\operatorname{Ber}(\ka || J_i), \quad \ka\in (0,1),
\end{gather*}
where  $i \in \{0,1\}$, and $\operatorname{Ber}(x||y)$ is the Kullback-Leibler divergence between two  Bernoulli distributions with success probabilities $x$ and $y$ respectively, i.e., 
\begin{equation} \label{h, KL of Bernoulli}
    \operatorname{Ber}(x||y) \equiv  x\log (x/y) + (1-x) \log ((1-x)/(1-y)), \quad x,y\in (0,1).
\end{equation}


\subsubsection{Testing the Gaussian mean} \label{subsub: gaussian} 
We next specialize  the above results to the 
special case of testing the mean of a  Gaussian distribution with  unit variance,
i.e., when $M=\bR$  and   $\varphi(\mu)=\mu^2/2$ for every $\mu \in \bR$ in \eqref{exp family}. For simplicity, we assume that the two parameter values under which we control the two error probabilities, $\mu_0$ and $\mu_1$, are opposite,  i.e., $\mu_1=-\mu_0=\eta$   for some $\eta>0$.  

In this case,   $\n(\alpha,\beta)$ and  $\thre(\alpha,\beta)$  in \eqref{n*(alpha,beta)}  can be computed explicitly when $T=\bar\Lambda$ or $T=\bar X$,  for any $\alpha, \beta \in (0,1)$, and do not need to be estimated via  simulation.  Specifically,   by the formulas in the general case  of this subsection we obtain 
\begin{gather*}
        I_0 =I_1=2\eta^2 \equiv I, \quad C=4 I \\
        \phi_0(\theta) =\theta(\theta-1)\, I, \quad 
\phi_1(\theta)=\theta(\theta+1)\, I, \quad \theta\in \bR \\
    \psi_0(\ka) = (I+\ka)^2/(4I), \quad \psi_1(\ka)= (I-\ka)^2/(4I), \quad \ka\in \bR,
\end{gather*}
and,  for any  $\alpha,\beta\in (0,1)$, 
\begin{equation} \label{gaussian_n_star}
    \n(\alpha,\beta)= \frac{(z_\alpha+z_\beta)^2}{2 I }  \qquad \text{and}  \qquad \thre(\alpha,\beta)
    =I \,\frac{z_\alpha-z_\beta}{z_\alpha+z_\beta},
\end{equation}
where $z_p$ is the upper $p$-quantile of the  standard Gaussian distribution.  In Figure \ref{rate 1} we plot  the functions 
$\psi_0$,  $\psi_1$, for  $T=\bar\Lambda$
and  $T=\bar Z$, when $\eta=0.5$.

Finally,  we note that in this case   the  asymptotic  relative efficiencies  in \eqref{ARE} coincide when  $T=\bar{Z}$, since 
\begin{equation} \label{h}
\text{ARE}_0=   \frac{\operatorname{Ber}(\Phi(-\eta) || \Phi(\eta))}{2\eta^2} = \frac{\operatorname{Ber}(\Phi(\eta) || \Phi(-\eta))}{2\eta^2}= \text{ARE}_1,
 \end{equation}
 where $\Phi$ denotes the cumulative distribution function of the standard Gaussian distribution and the function 
 $\operatorname{Ber}(x||y)$ is defined in \eqref{h, KL of Bernoulli}. We  note also that this quantity converges to $0.25$ as $\eta \to \infty$ and to $2/\pi$ as $\eta \to 0$.
 In  Figure \ref{ARE1} we plot the asymptotic relative efficiency in \eqref{h}  as a function of $\eta$ in $\in(0,5)$.

\begin{table}
    \begin{tabular}{cc}
        \begin{subfigure}{0.5\textwidth}\centering\includegraphics[width=0.9\textwidth]{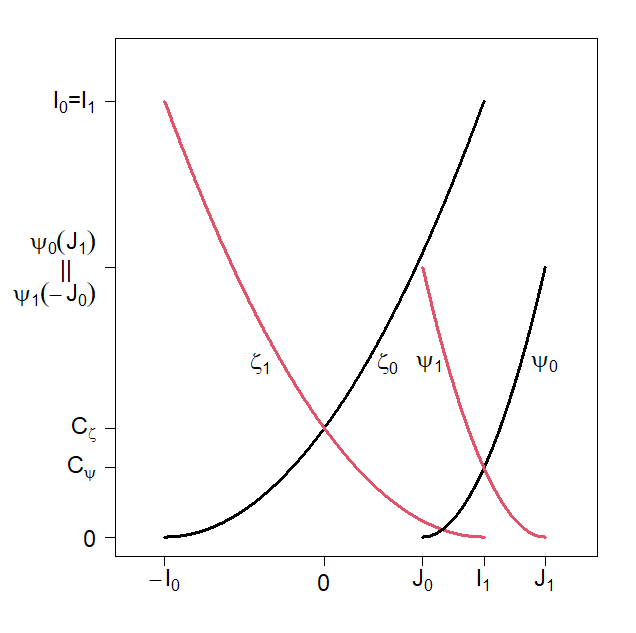}\caption{
        }\label{rate 1}\end{subfigure} &
        \begin{subfigure}{0.5\textwidth}\centering\includegraphics[width=0.9\textwidth]{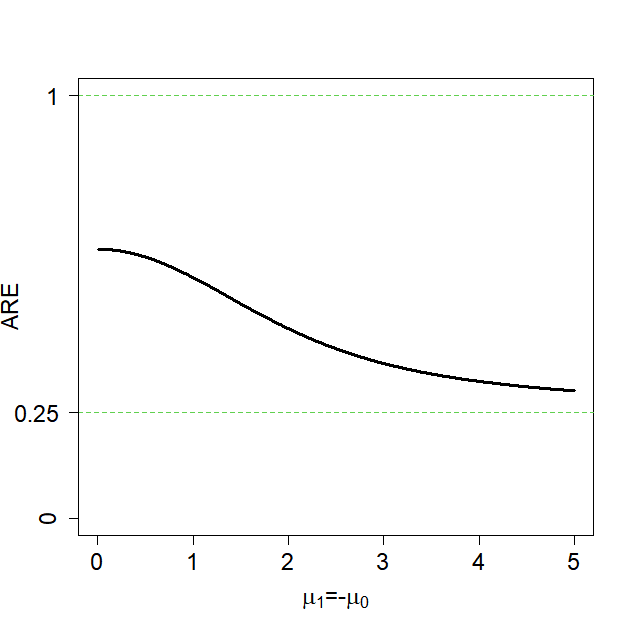}\caption{
        }\label{ARE1}\end{subfigure} \\
        \begin{subfigure}{0.5\textwidth}\centering\includegraphics[width=0.9\textwidth]{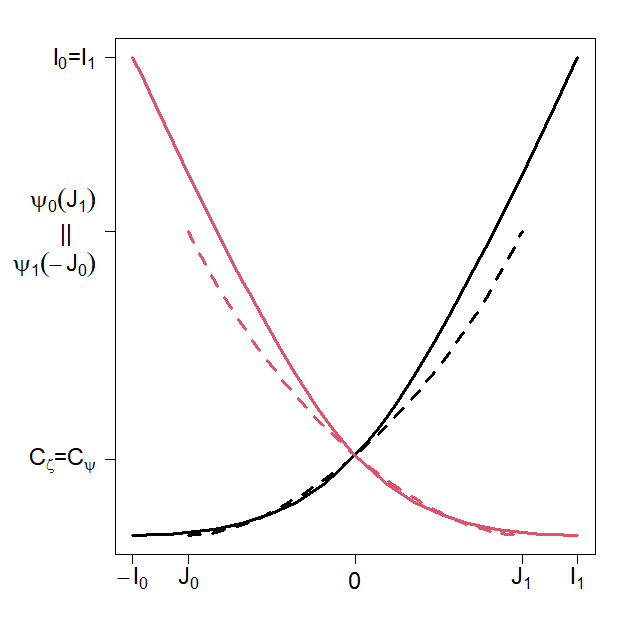}\caption{
        }\label{rate 2}\end{subfigure} &
        \begin{subfigure}{0.5\textwidth}\centering\includegraphics[width=0.9\textwidth]{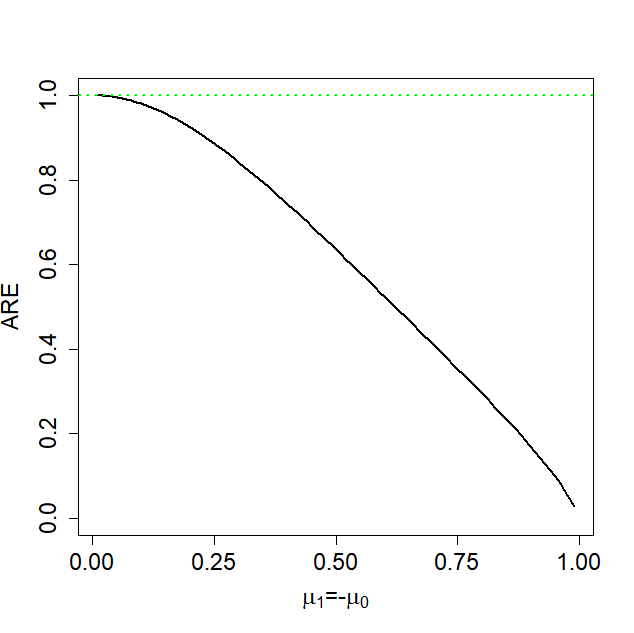}\caption{
        }\label{ARE2}\end{subfigure} \\
        \begin{subfigure}{0.5\textwidth}\centering\includegraphics[width=0.9\textwidth]{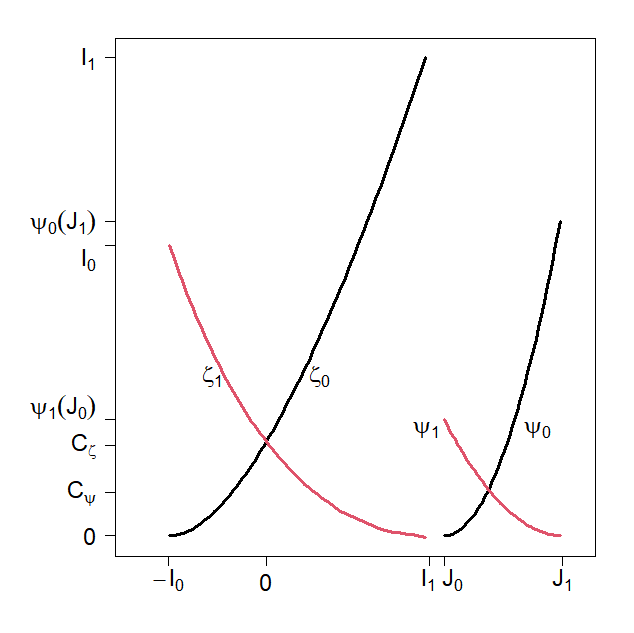}\caption{
        }\label{rate 3}\end{subfigure} &
        \begin{subfigure}{0.5\textwidth}\centering\includegraphics[width=0.9\textwidth]{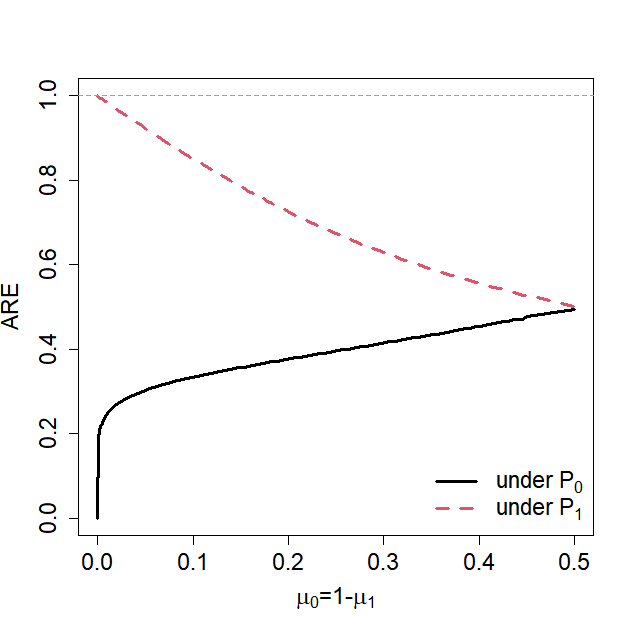}\caption{
        }\label{ARE3}\end{subfigure} \\    \end{tabular}
        \caption{In the left column we plot the functions images of $\psi_0$ and $\psi_1$ when $T=\bar\Lambda$ and when $T$ is the alternative test statistic considered in each of these examples of Section \ref{sec:examples}.  To distinguish, we write  $\zeta_i$ instead of  $\psi_i$ when $T=\bar\Lambda$, $i \in \{0,1\}$ and 
write  $C_\zeta$ and $C_\psi$ for the quantity $C$ defined in \eqref{C}. In the     right column  we plot the corresponding asymptotic relative efficiencies,   $\text{ARE}_0$ and $\text{ARE}_1$, defined in \eqref{ARE}.}
    \label{Figure: function images}
\end{table}

\subsection{Testing the coefficient of a first-order autoregressive model}  
\label{subsec: AR(1)}
In the second example of this section we assume that   $X$ follows a Gaussian first-order autoregressive  model, i.e., 
\begin{equation*}  
 X_{n} =  \mu  X_{n-1} +\epsilon_n,   \quad n \in \bN,
\end{equation*}
where   $X_0=0$,  $\{\epsilon_n, n\in \bN\}$ are iid standard Gaussian, and  $\mu$  is an unknown parameter  taking values in  $M= (-1,1)$.  We denote by $\bP_\mu$ the distribution and by $\bE_\mu$ the corresponding expectation when the true parameter is $\mu$, and consider the  testing problem  of Subsection \ref{subsec: one-sided testing problem}.

In this setup, the  log-likelihood ratio statistic in \eqref{def: llr}  becomes
    \begin{equation} \label{llr: AR1}
       \Lambda_n = (\mu_1-\mu_0) \; \left( \sum_{i=1}^n X_{i-1} X_i - \frac{\mu_1+\mu_0}{2} \sum_{i=1}^n X_{i-1}^2 \right) , \quad n\in \bN.
\end{equation}

 For any $\mu \in M$, 
from \cite[Chapter 3]{Brockwell_TimeSeriesBook} it follows that
\begin{align} \label{SLLN_AR}
 &\frac{1}{n}\sum_{i=1}^n X_{i-1}^2 \rightarrow 
\frac{1}{1-\mu^2} \qquad \text{and} \qquad  \frac{1}{n} \sum_{i=1}^n X_{i-1} X_i   \rightarrow  \frac{\mu}{1-\mu^2}   \quad \bP_\mu-\text{a.s.},
\end{align}
and consequently 
\begin{equation} \label{a.s. limit of average LLR of AR1}
    \bar{\Lambda}_n  \to   \frac{\mu_1-\mu_0}{1-\mu^2} \left( \mu-\frac{\mu_1+\mu_0}{2} \right)  \qquad \bP_\mu-\text{a.s.}
\end{equation}
 Moreover,   from \cite{Bercu_1997} it follows that, for every $\mu \in M$,  
\begin{align} \label{ARlim}
    \begin{split} \frac{1}{n} \log \bE_\mu\left[e^{\theta\Lambda_n}\right] & \to -\frac{1}{2} \log \left( \frac{1}{2}  \, p_\mu(\theta)+ \frac{1}{2}  \, \sqrt{ p^2_\mu(\theta)-4q^2_\mu(\theta)} \right)\equiv \phi(\theta; \mu), 
    \end{split}
    \end{align}
where $\theta \in \cD_{\mu} \equiv \cD_{\mu,1}\cup \cD_{\mu,2}\cup\cD_{\mu,3}$, 
\begin{align*}
    \begin{split}
         \cD_{\mu,1} &\equiv \left\{\theta\in \bR: \, \mu^2 < p_\mu(\theta) \leq 2\mu^2, \;  q_\mu^2(\theta) \leq \mu^2 (p_\mu(\theta)-\mu^2) \right\} , \\
         \cD_{\mu,2} &\equiv \left\{\theta\in \bR: \, 2\mu^2 < p_\mu(\theta)<2, \; p_\mu(\theta) > 2 |q_\mu(\theta)|\right\} ,\\
        \cD_{\mu,3} &\equiv\left\{ \theta\in \bR: \, p_\mu(\theta) \geq 2, \; q_\mu^2(\theta)\leq p_\mu(\theta)-1 \right\}, \\
 \text{and} \quad  p_\mu(\theta) &\equiv 1+\mu^2 + (\mu_1-\mu_0)(\mu_1 + \mu_0) \, \theta,  \\
 q_\mu(\theta)  &\equiv  -\mu-(\mu_1-\mu_0) \theta.
    \end{split}
 \end{align*}
 the function $\phi(\cdot; \mu)$  in \eqref{ARlim}  is differentiable in $\cD_\mu^o$,  and  $$0\in\cD^o_{\mu_0,2}, \quad \quad 1\in  (\cD_{\mu_0,1}\cup\cD_{\mu_0,2})^o.$$
Thus, setting  $\mu$ equal to  $\mu_0$ and  $\mu_1$  in \eqref{a.s. limit of average LLR of AR1}-\eqref{ARlim}, we conclude  that all assumptions in Corollary \ref{coro: LLR, sufficient conditions} are satisfied with 
$$I_0=\frac{(\mu_1-\mu_0)^2}{2(1-\mu_0^2)},  \qquad  I_1=\frac{\mu^2_1-\mu^2_0}{2(1-\mu_1^2)}, \qquad \phi_i=\phi(\cdot;\mu_i),  \quad i\in \{0,1\}.$$
Moreover,  from \eqref{LLR, C} it follows, by minimizing $\phi(\cdot; \mu_0)$, that  
\begin{equation} \label{AR(1), LLR, C}
    C=\log \sqrt{
    \frac{1-\mu_0\, \mu_1}
    {1-(\mu_0+\mu_1)^2/4} }.
\end{equation}

The functions $\psi_0$ and $\psi_1$ in this context are computed numerically and are  plotted in Figure   \ref{rate 2} when $\mu_1=-\mu_0=0.5$. We note that, in this case, they   are symmetric about the y-axis, a property that does  not hold, in general, when $\mu_1\neq-\mu_0$.

\subsubsection{The Yule-Walker estimator}
An alternative  test statistic for this testing problem   is the  Yule-Walker estimator, i.e.,
$T=\hat{\mu} \equiv \{\hat\mu_n,  n \in \bN\}$, where 
\begin{equation} \label{def: Yule Walker}
\hat\mu_n\equiv \frac{\sum_{i=1}^n X_{i-1}X_i} {\sum_{i=1}^n X_{i}^2}, \quad n \in \bN.
\end{equation}

From \eqref{SLLN_AR} it follows that   $\hat\mu_n$ is a strongly consistent estimator of $\mu$, i.e., for every $\mu \in M$, 
\begin{equation} \label{consi_Yule_Walker}
\bP_\mu( \hat{\mu}_n \to \mu)=1.
\end{equation}
  Moreover, from \cite{Bercu_1997} it follows  that, for any $\mu\in M$,
\begin{align} \label{ldp_Yule_Walker}
    \begin{split}
        & -\frac{1}{n} \log \bP_\mu(\hat\mu_n>\ka) \to 
      \psi(\ka; \mu) , \quad \forall \, \ka\in (\mu,1) \\
        & -\frac{1}{n} \log \bP_\mu(\hat\mu_n\leq \ka) \to \psi(\ka; \mu) , \quad \forall \, \ka\in (-1,\mu),
    \end{split}
\end{align}
where  the function 
$$ \psi(\ka; \mu) \equiv   \log \sqrt{ \frac{1+\mu^2-2\mu \ka}{1-\ka^2} }, \quad \ka\in (-1,1)$$
 is   strictly convex, has a unique root at $\mu$, goes to $\infty$ as $\ka$ goes to $-1$ or $1$.  Therefore, setting $\mu$ equal to $\mu_0$ and $\mu_1$ in \eqref{consi_Yule_Walker} and \eqref{ldp_Yule_Walker},  we conclude that    assumptions \eqref{a.s. limits of eta} \eqref{LD, non-LLR_null}, \eqref{LD, non-LLR_alternative}, \eqref{extra assumption}  hold with
$$J_i=\mu_i, \quad \psi_i=\psi(\cdot; \mu_i), \quad i\in \{0,1\}.$$

Interestingly,  equating  $\psi_{0}$ and $\psi_{1}$ we obtain the same value for $C$ as in \eqref{AR(1), LLR, C}.  In view of 
\eqref{ARE0, r=1},  this  implies that   using $\hat{\mu}$,  instead of $\bar{\Lambda}$,  as the test statistic,  does not reduce  the asymptotic relative efficiency of the \textit{fixed-sample-size} test  as  $\alpha, \beta \to 0$ so that  $|\log\alpha|\sim |\log\beta|$. 
 This is not the case for the proposed multistage tests, as can be seen in   Figure \ref{ARE2}, where we plot $\text{ARE}_0$ and $\text{ARE}_1$
when $\mu_0=-\mu_1$, in which case they coincide,  for different values of $\mu_1$ in  $(0,1)$. 

\subsection{Testing the transition matrix of a  Markov chain} \label{example, two-state Markov model}

In the third example of this section we assume that  $X$ is an irreducible and recurrent  Markov chain  with state space $[I]= \{0,1,\ldots, I\}$, where $I\in \bN$,  initial value $X_0=0$,  transition matrix $\Pi$, and stationary distribution $\pi$.   Moreover,  
we  note that (see, e.g.,  \cite[Theorem 5.5.9]{Durret_Book})
$$Y \equiv \left\{Y_n \equiv (X_{n-1},X_n), \, n \in \bN \right\}$$
is also  an irreducible and recurrent Markov chain, with state space $[I]^2$,  transition matrix $\Pi^\cY$ whose  $\left((i_1,i_2), \, (i_3,i_4)\right)$-th element is
$$ \Pi(i_3,i_4) \;  1\{i_2=i_3 \}, \quad 
(i_1,i_2), \, (i_3,i_4)\in [I]^2,$$
and stationary distribution 
$$\pi^\cY(i,j)= \pi(i) \, \Pi(i,j), \quad i,j\in [I].$$

For simplicity, we identify the  family of all possible distributions of $X$,  $\cP$,  with the class of all  irreducible and recurrent transition matrices of dimension  $I+1$. For each   $\Pi \in \cP$, we denote by $\bP_\Pi$ the distribution of $X$, and by  $\bE_\Pi$ the corresponding expectation,   when the transition matrix of $X$  is $\Pi$.   We consider the general testing setup of Section \ref{sec:formulation}, where $\cP_0$ and $\cP_1$ are two arbitrary subclasses of $\cP$, and  
 $$\Pro_i \equiv \bP_{\Pi_i}, \quad i\in\{0,1\}
$$
for some  arbitrary  $\Pi_i\in \cP_i$, $i\in\{0,1\}$.  In this setup, the  log-likelihood ratio statistic  in \eqref{def: llr} takes the form:
\begin{align*}  
        \Lambda_n &= \sum_{(i,j)\in [I]^2} r(i,j) \, N_n (i,j)=  \sum_{m=1}^n \f(Y_m), \quad n\in\bN,
        \end{align*}
where, for each $(i,j), y \in [I]^2$, 
\begin{align*}  
    r(i,j) &\equiv  \log \left( \frac{\Pi_1(i,j)}{\Pi_0(i,j)} \right), \quad
    N_n(i,j) \equiv 
    \sum_{m=1}^n 1\{Y_m=(i,j)\}, \\
&   \f(y) \equiv \sum_{(i,j)\in [I]^2} r(i,j) \cdot  1\{y=(i,j)\}. 
\end{align*}

For any  $\Pi \in \cP$, from \cite[Example 6.2.4]{Durret_Book} it follows that, 
for every  $ (i,j)\in [I]^2 $, 
$$\frac{1}{n} N_n(i,j)\to \pi^\cY(i,j)   \qquad \bP_\Pi-\text{a.s.}  $$
and, as a result, 
\begin{equation} \label{a.s. limt of LLR, Markov chain}
    \bar\Lambda_n \to \sum_{(i,j)\in [I]^2} r(i,j) \, \pi^\cY(i,j) \qquad \bP_\Pi-\text{a.s.}
\end{equation}
Moreover, by \cite[Theorem 3.1.1 \& 3.1.2]{Dembo_Zeitouni_LDPBook},
it follows that, for any  $\Pi \in \cP$, 
 \begin{equation} \label{set}
    \frac{1}{n} \log \Exp_\Pi[\exp\{\theta\Lambda_n\}] \to \log \lev \left( \Pi^\cY_{\theta,\f} \right) \equiv 
    \phi(\theta;\Pi), \quad \text{for every} \quad \theta \in \bR,
\end{equation}
    where $\lev$ is the functional that maps a matrix to its greatest eigenvalue,  
      $\Pi^\cY_{\theta,\f}$ is a matrix of the same dimension as $\Pi^\cY$ whose  $\left((i_1,i_2), \, (i_3,i_4)\right)$-th element is
$$ 
\Pi^\cY((i_1,i_2),(i_3,i_4)) \,  \exp\{\theta \, \f((i_3,i_4))\}, \quad 
(i_1,i_2), \, (i_3,i_4)\in [I]^2,
$$
and  the limit in \eqref{set} is a finite and differentiable function of $\theta$. Therefore, setting $\Pi$ equal to  $\Pi_0$ and $\Pi_1$ in \eqref{a.s. limt of LLR, Markov chain}-\eqref{set} we
conclude  that all assumptions in Corollary \ref{coro: LLR, sufficient conditions} are satisfied, and $I_i,\phi_i$, $i\in\{0,1\}$ can be computed accordingly. 

\subsubsection{The two-state case} 
 \label{subsubsec: two-state Markov model}
We next specialize the previous setup to the case that $I=1$, where  the transition  matrix and stationary distribution of $X$ are of the form 
$$ \Pi=\begin{pmatrix}
            p & 1-p \\
            1-\mu & \mu
        \end{pmatrix}, \quad \pi=\left( \frac{1-\mu}{2-p-\mu}, \frac{1-p}{2-p-\mu} \right), \quad \text{where} \quad p, \mu \in (0,1).$$
We fix $p\in (0,1)$, so that the only unknown parameter is   $\mu$, which  takes values in $M= (0,1)$.  Thus, we  now denote by $\bP_\mu$  the distribution,  and by $\bE_\mu$  the corresponding expectation,
of $X$ when the unknown parameter is  $\mu$, and consider the testing setup of Subsection \ref{subsec: one-sided testing problem}.  In this case,  \eqref{a.s. limt of LLR, Markov chain}  reduces to 
\begin{equation} \label{a.s. limit of average LLR, two-state Markov}
\bar{\Lambda}_n \to \frac{1-p}{2-p-\mu} \big( \operatorname{Ber}(\mu ||\mu_0)- \operatorname{Ber}(\mu ||\mu_1) \big) \quad \bP_\mu-\text{a.s.}, 
\end{equation}
where $\operatorname{Ber}(x||y)$ is defined in \eqref{h, KL of Bernoulli}, whereas $I_0$ and  $I_1$  become:
$$ I_0=\frac{1-p}{2-p-\mu_0} \operatorname{Ber} (\mu_0\|\mu_1), \quad  I_1= \frac{1-p}{2-p-\mu_1} \operatorname{Ber}(\mu_1\|\mu_0). $$

An alternative test statistic in this setup is the sample average in \eqref{average},  or equivalently,
$$T_n= \bar{X}_n \equiv \frac{1}{n}  \sum_{m=1}^n \ff(x_m), \quad \text{where} \quad \ff(x)= x. 
$$ 
Unlike the first example of this section, however, this test statistic does not lead to asymptotic optimality, as it does not admit 
a bijection with the log-likelihood ratio,  as in \eqref{nice linear transformation}.  To compute the resulting asymptotic relative efficiency, \eqref{ARE}, we note that, by  \cite[Example 6.2.4]{Durret_Book},   for any $\mu\in M$,
 \begin{equation} \label{sett} 
 \bar{X}_n \to \sum_{i\in [I]} i \, \pi(i) = \frac{1-p}{2-p-\mu} \qquad \bP_\mu-\text{a.s.}
 \end{equation}
Moreover,   by \cite[Theorem 3.1.1 \& 3.1.2]{Dembo_Zeitouni_LDPBook} it follows that,  for any $\mu\in M$,  
 \begin{equation} \label{l}
 \frac{1}{n} \log \Exp_\mu[\exp
 \{\theta \, n \, \bar{X}_n\}] \to \log \lev(\Pi_{\theta,\ff})   \equiv  \phi(\theta;\Pi),\quad \forall \; \theta \in \bR,
 \end{equation}
where $\Pi_{\theta,\ff}$ is a matrix of the same dimension as $\Pi$, whose $(i,j)$-th element is 
$$ \Pi(i,j) \, e^{\theta \, \ff(j)}, \quad i,j\in [I],$$
and   the limit  is finite, differentiable and steep in $\bR$ as a function of $\theta$.  Therefore, setting $\mu$ equal to  $\mu_0$ and $\mu_1$  in \eqref{sett}-\eqref{l}  we  conclude  that  all assumptions in Corollary \ref{further sufficient condition, non-LLR} are satisfied  with
$$J_i= \frac{1-p}{2-p-\mu_i} \qquad \text{and}\qquad \phi_i(\theta)=   \phi(\theta;\Pi_i), \quad \text{for every}  \quad \theta \in \bR, \quad  i\in \{0,1\}.$$

In Figure \ref{rate 3} we plot the functions $\psi_0$, $\psi_1$ for $T=\bar\Lambda$ and   $T=\bar X$  when $\mu_0=1-\mu_1=0.25$. 
In Figure \ref{ARE3} we plot  the asymptotic relative efficiencies  in \eqref{ARE} against $\mu_0=1-\mu_1$ for different values of 
$\mu_0$ in $(0,0.5)$.

\section{Numerical studies} \label{sec:numerical}
In this section we present the results of two numerical studies in which we compare the 3-stage test, $\three$,  the 4-stage test, $\four$, both  with  $\Stat=\bar{\Lambda}$,  against  the SPRT, $\chi'$,  when 
\begin{itemize}
    \item   testing the mean of an iid Gaussian sequence with unit variance  (Subsection \ref{subsub: gaussian}), with  $\mu_1=-\mu_0=0.5$,
\item   testing the coefficient of an first-order autoregression  (Subsection \ref{subsec: AR(1)}), when   $\mu_1=-\mu_0=0.5$,
\item   testing an entry in the transition matrix of a two-state Markov chain (Subsection \ref{subsubsec: two-state Markov model}), with $p=0.5$ and $\mu_0=1-\mu_1=0.25$.
\end{itemize}
Before we describe the two studies and present the main findings, we discuss how the tests  are designed and how their average sample sizes are computed.

\subsection{Design of tests}
In all cases, the  SPRT  in \eqref{Definition of SPRT} is designed with  $B=|\log\alpha|$ and $A=|\log\beta|$, whereas  the multistage tests are designed according to  Theorems \ref{th:3ST_error_control} and \ref{th:4ST_error_control}, with the  free parameters selected according to \eqref{free_3ST} and \eqref{free_4ST_hat}. 
The functions  $\n$ and $\thre$,  defined in \eqref{n*(alpha,beta)},
are evaluated using   the closed-form expressions in \eqref{gaussian_n_star}   in the first  testing problem  and the importance sampling method of  Section \ref{sec: IS}  in  the other two.
Specifically,  the importance sampling distribution  employed 
in the second (resp. third) testing problem is  the distribution $\bP_\mu$  for which the limit in  \eqref{a.s. limit of average LLR of AR1} (resp.  \eqref{a.s. limit of average LLR, two-state Markov}) is  equal to $\ka$.     Moreover, grid search is used for the determination of  the free parameters of the multistage tests.

\subsection{Computation of the expected sample sizes}
The expected sample sizes of the multistage tests are computed  using the formulas \eqref{3ST, first case, ESS}-\eqref{3ST, second case, ESS} and \eqref{4ST, first case, ESS}-\eqref{4ST, first case, ESS, n1<=n0} in the first testing problem, as it is possible to compute the multivariate Gaussian probabilities in these expressions, and   plain Monte Carlo in the other two. The expected sample size of the SPRT is estimated with  plain Monte Carlo in all cases.    In each  Monte Carlo application,   $10^4$ replications are utilized,  leading in all cases to \textit{relative}  errors below $5\%$. 

\subsection{The first study}

\begin{table}
    \begin{tabular}{cc}
        \begin{subfigure}{0.41\textwidth}\centering\includegraphics[width=0.9\textwidth]{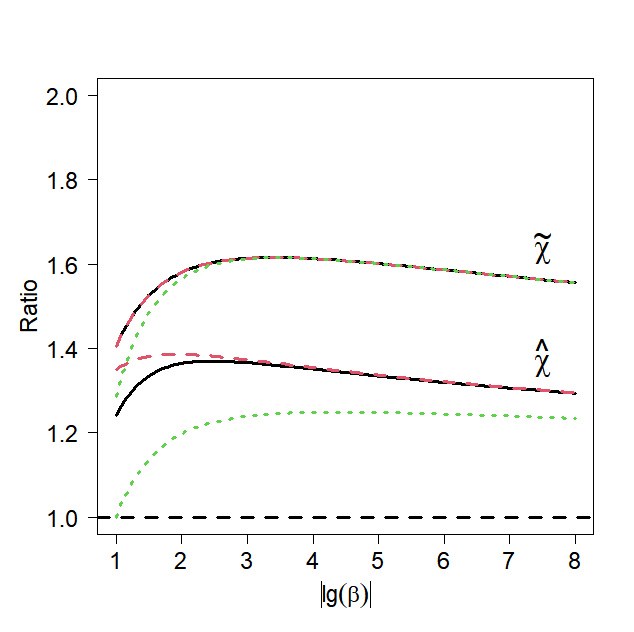}
       \caption{$\alpha=\beta$}\label{2a}\end{subfigure} &
        \begin{subfigure}{0.41\textwidth}\centering\includegraphics[width=0.9\textwidth]{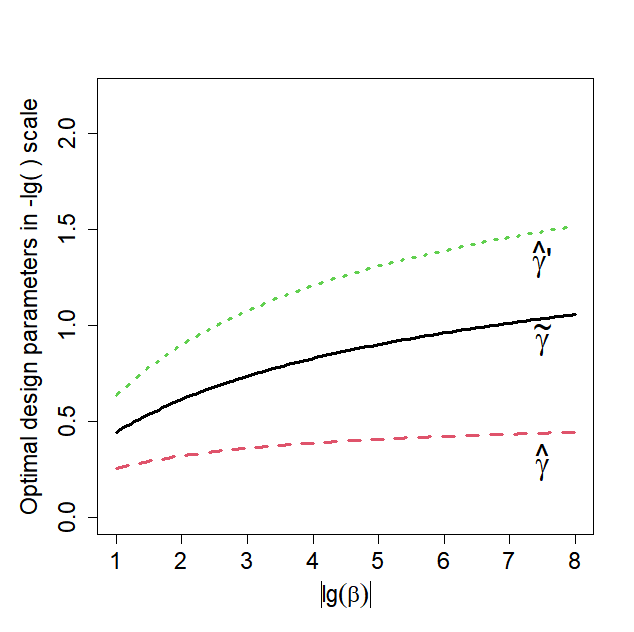}\caption{$\alpha=\beta$}\label{3a}\end{subfigure} \\
        \begin{subfigure}{0.41\textwidth}\centering\includegraphics[width=0.9\textwidth]{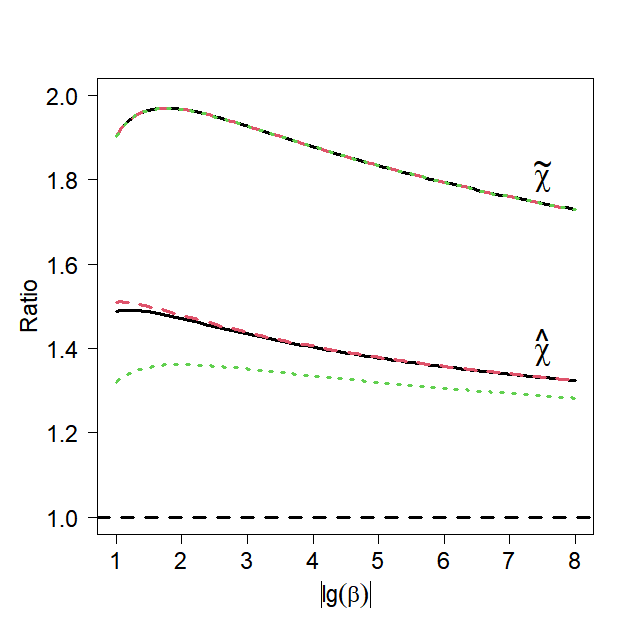}\caption{$|\log\alpha|=4\,|\log\beta|$}\label{2b}\end{subfigure} & 
        \begin{subfigure}{0.41\textwidth}\centering\includegraphics[width=0.9\textwidth]{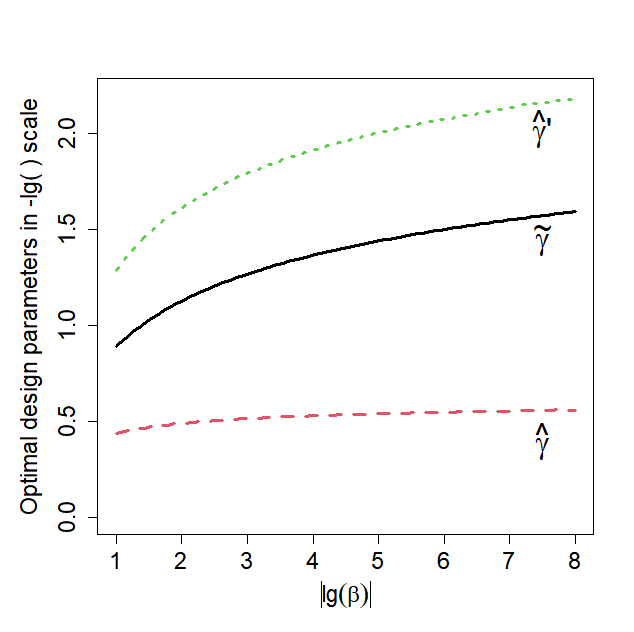}\caption{$|\log\alpha|=4\,|\log\beta|$}\label{3b}\end{subfigure}\\
        \begin{subfigure}{0.41\textwidth}\centering\includegraphics[width=0.9\textwidth]{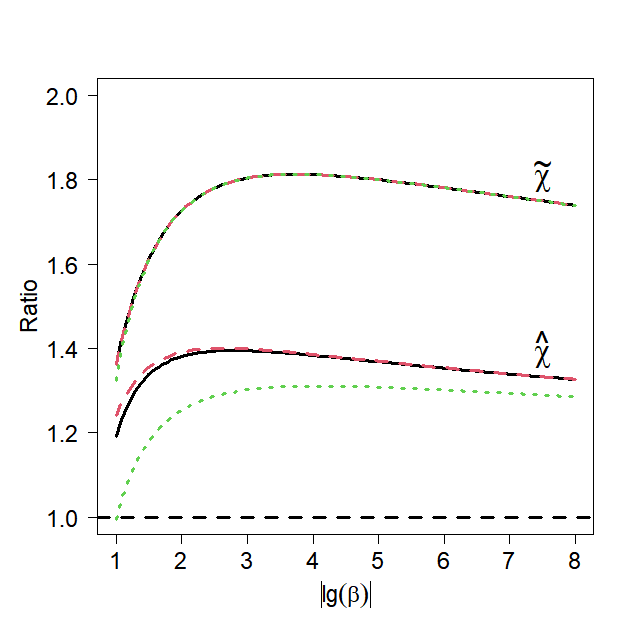}\caption{$|\log\alpha|=|\log\beta|^{1.5}$}\label{2d}\end{subfigure} &
        \begin{subfigure}{0.41\textwidth}\centering\includegraphics[width=0.9\textwidth]{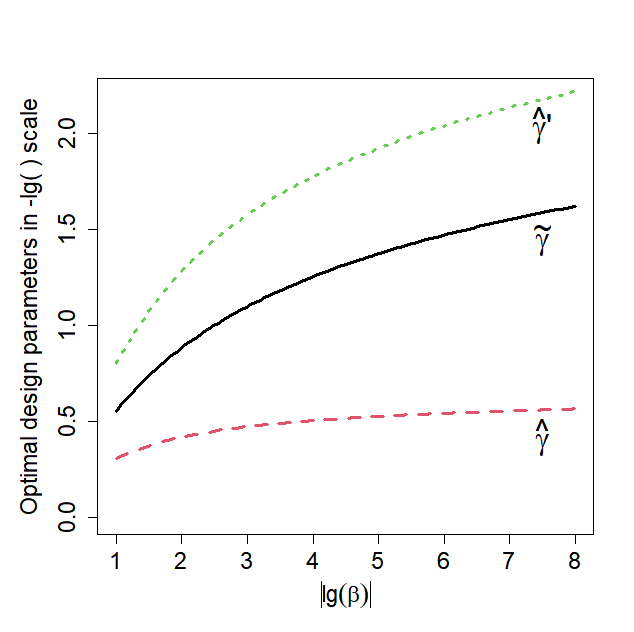}\caption{$|\log\alpha|=|\log\beta|^{1.5}$}\label{3d}\end{subfigure} \\
        \begin{subfigure}{0.41\textwidth}\centering\includegraphics[width=0.9\textwidth]{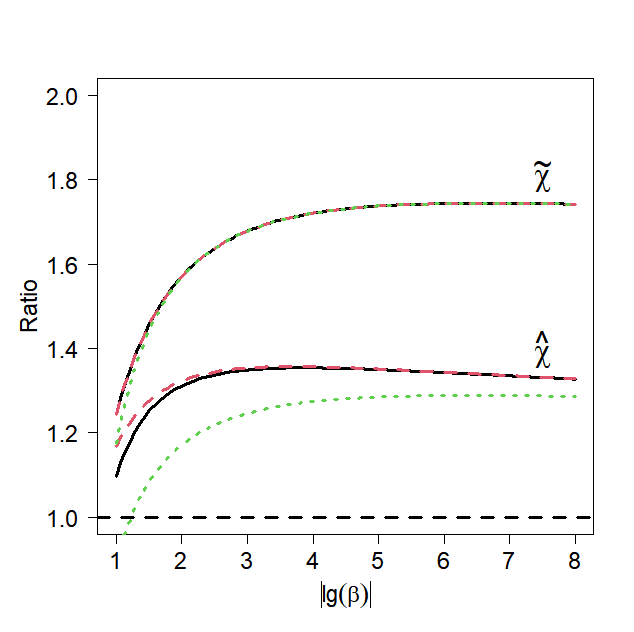}\caption{$|\log\alpha|=|\log\beta|/\beta^{0.08}$}\label{2f}\end{subfigure} &
        \begin{subfigure}{0.41\textwidth}\centering\includegraphics[width=0.9\textwidth]{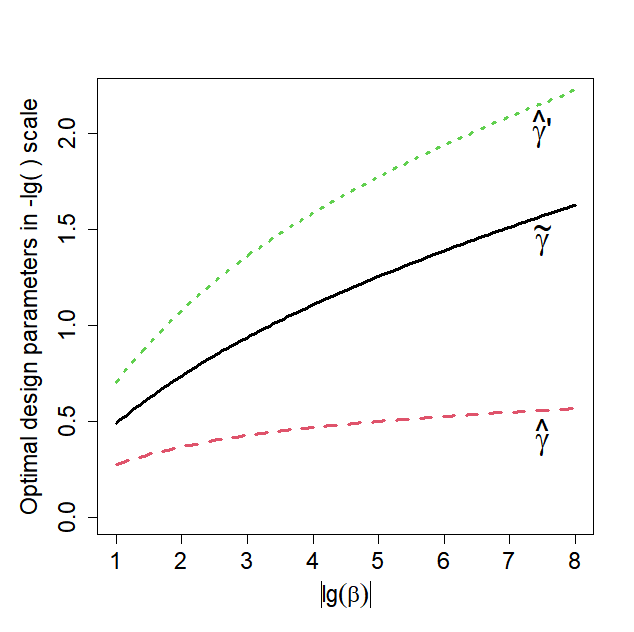}\caption{$|\log\alpha|=|\log\beta|/\beta^{0.08}$}\label{3f}\end{subfigure}
        \end{tabular}
    \caption{In the left column we plot  $\Exp_0[\TST]/\Exp_0[\SPRT]$ and $\Exp_0[\FST]/\Exp_0[\SPRT]$, along with the corresponding  bounds from Section \ref{sec:formulation}, against $|\lg \beta|$, when $\alpha$ follow the given pattern, in testing the mean of iid Gaussian sequence with unit variance.
    In the right column we plot the corresponding $|\lg \tgamma|$ in $\three$ and  
   $|\lg \hgamma|, |\lg \hgammap|$ in $\four$ against  $|\lg \beta|$. }
    \label{Figure: ratios}
\end{table}

In the first study we compare the expected sample sizes of $\three$, $\four$ and $\chi'$ under $\Pro_0$,
with the understanding that analogous results can be obtained when comparing $\three$, $\ffour$ and $\chi'$ under $\Pro_1$.  Specifically, we evaluate  $\Exp_0[\TST]/ \Exp_0[\SPRT]$ and $ \Exp_0[\FST]/ \Exp_0[\SPRT]$,  i.e., the ratio of the expected sample sizes  under $H_0$ of  $\three$ and $\four$ over that of  $\chi'$,
in the context  of the  first testing problem,
for different values of $\beta$, when $\alpha$  is given by one of the following  relationships:
\begin{align}  \label{regimes}
 & \alpha = \beta, \quad 
\alpha= \beta^4 , \quad 
 |\log\alpha|=|\log\beta|^{1.5}, \quad  
 |\log\alpha|=|\log\beta| / \beta^{0.08}.
\end{align}  

In the left column of Figure \ref{Figure: ratios} we present these ratios, together with the  non-asymptotic bounds implied by \eqref{3ST, ESS bound, null, after plugin}-\eqref{3ST, ESS bound, alternative, after plugin} and \eqref{4ST, ESS bound, after plugin, null}-\eqref{4ST, ESS bound, after plugin, alternative}. 
In these graphs we observe  a  slow,  downward trend,  as $\alpha$ and $\beta$ decrease,  in all ratios  but the one that corresponds to \textit{$\three$ in the last asymptotic regime}. This is consistent with  Theorem \ref{theorem, main}, in  which $\four$ is shown to achieve asymptotic optimality under $\Pro_0$ in all  asymptotic regimes in \eqref{regimes}, whereas $\three$  only  in the first three.  

 From these graphs we also  see that, under  $\Pro_0$,  the average sample of the 4-stage test, $\four$,  is   substantially smaller than that of the 3-stage test,  $\three$,  in all cases, and  does not exceed that of the SPRT by more than 50\%. 

Finally,  we see that the  upper bounds are very accurate approximations of the  expected sample sizes  in all cases, even for  large values of $\alpha$ and $\beta$. On the other hand, the lower bounds are similarly accurate for $\three$, but  relatively conservative for $\four$. 
To illustrate the selection of the free parameters of the two multistage tests,  
in the right column of Figure \ref{Figure: ratios}  we plot  $\tgamma$ in $\three$ and  $\hgamma,\hgammap$ in $\four$, against  $\beta$,   all of them in the $|\lg (\cdot)|$ scale.


\subsection{The second study}

\begin{table}
    \begin{tabular}{cc}
        \begin{subfigure}{0.5\textwidth}\centering\includegraphics[width=0.9\textwidth]{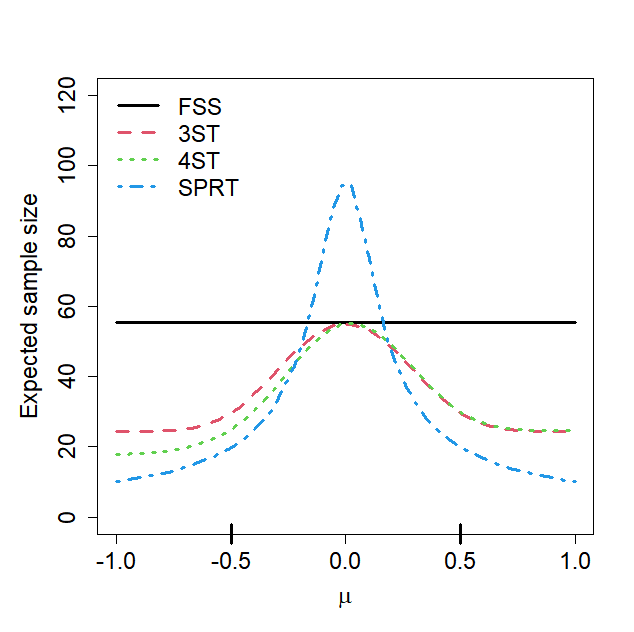}
        \begin{center}\caption{
        IID Gaussian, $\alpha=\beta=10^{-4}$}\end{center}
        \label{4a}\end{subfigure} &
        \begin{subfigure}{0.5\textwidth}\centering\includegraphics[width=0.9\textwidth]{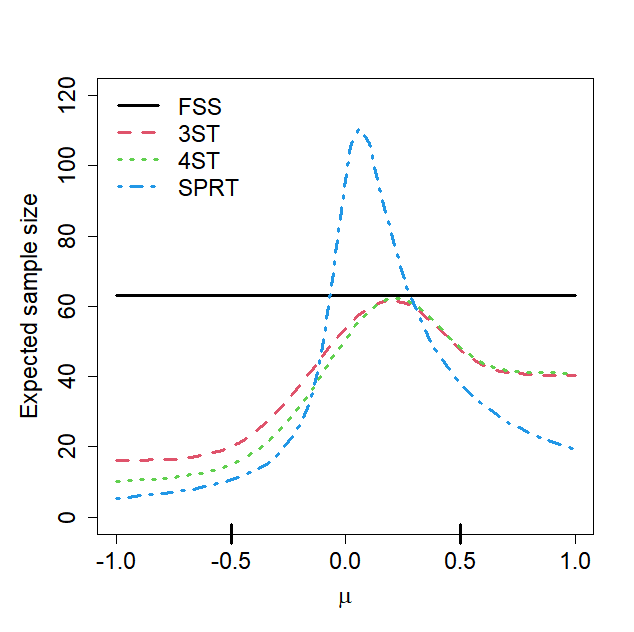}\caption{$\alpha=10^{-8}, \beta=10^{-2}$}\label{4b}\end{subfigure} \\
        \begin{subfigure}{0.5\textwidth}\centering\includegraphics[width=0.9\textwidth]{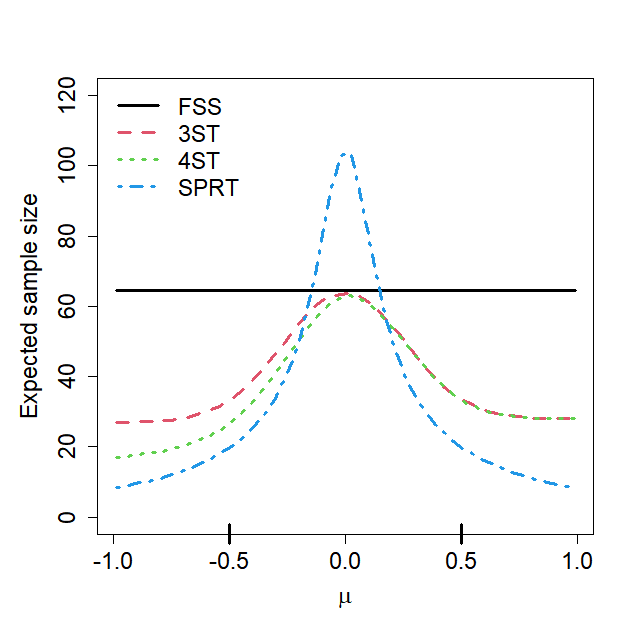}\caption{AR(1), $\alpha=\beta=10^{-4}$}\label{4c}\end{subfigure} &
        \begin{subfigure}{0.5\textwidth}\centering\includegraphics[width=0.9\textwidth]{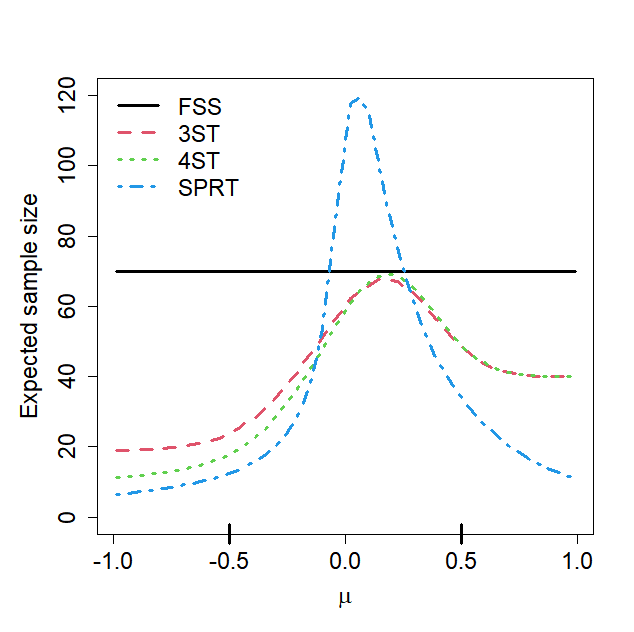}\caption{$\alpha=10^{-8}, \beta=10^{-2}$}\label{4d}\end{subfigure} \\
        \begin{subfigure}{0.5\textwidth}\centering\includegraphics[width=0.9\textwidth]{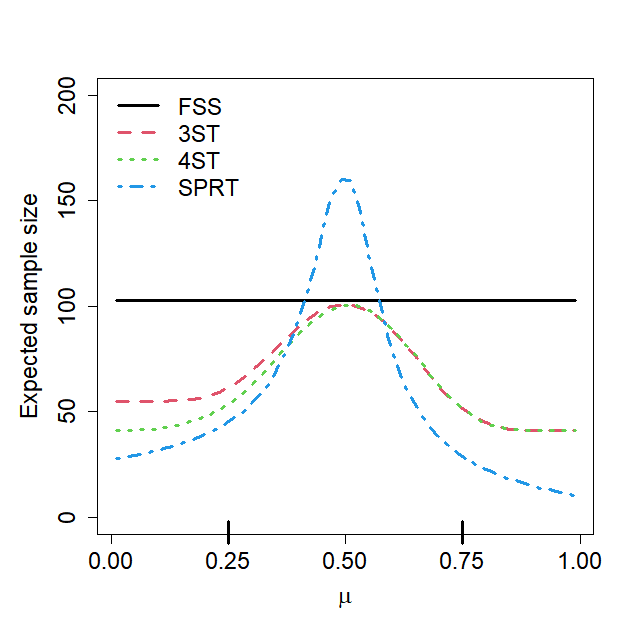}\caption{Two-state Markov, $\alpha=\beta=10^{-4}$}\label{4e}\end{subfigure} &
        \begin{subfigure}{0.5\textwidth}\centering\includegraphics[width=0.9\textwidth]{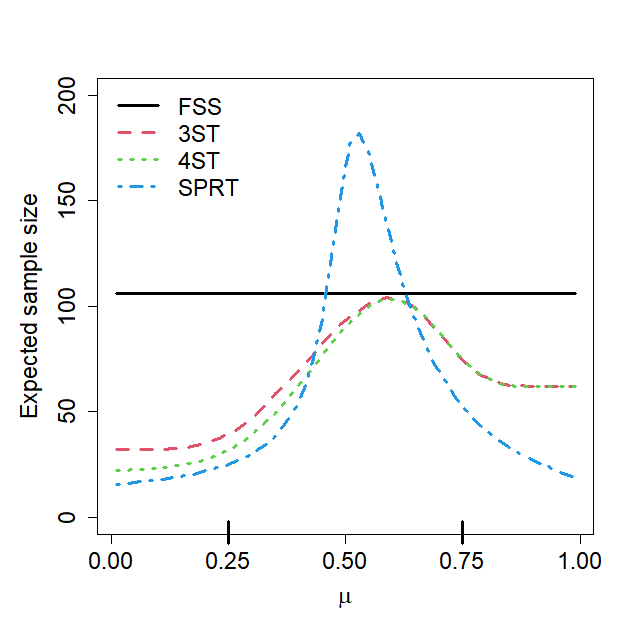}\caption{$\alpha=10^{-8}, \beta=10^{-2}$}\label{4f}\end{subfigure} \\    \end{tabular}
    \caption{We plot the expected sample sizes of the fixed-sample-size test, the two multistage tests and the SPRT  against the true value of the parameter. The values of the parameter at which we control the two types of error probabilities are highlighted on the x-axis. 
    Each row corresponds to  one of the testing  problems considered in Section \ref{sec:examples}. The first column corresponds to $\alpha=\beta=10^{-4}$, and the second  $\alpha=10^{-8}, \beta=10^{-2}$.}
    \label{Figure: performance against parameter}
\end{table}

In the second study we  compare the expected sample sizes of the various tests when the true distribution is not necessarily $\Pro_0$ or $\Pro_1$. 
Specifically, we compute $\bE_\mu[\tau]$ for different values of $\mu$, in each of the three testing problems, when   $\alpha=\beta=10^{-4}$ and when  $\alpha=10^{-8}, \, \beta=10^{-2}$. The results are presented  in  Figure \ref{Figure: performance against parameter}.  
Consistently with our discussion in Subsection  \ref{sec: upper bound in the middle}, we can see that when the true parameter  is  close to the middle  of $\mu_0$ and $\mu_1$,  the expected sample size of the  SPRT is much larger  than
those  of the  multistage tests. On the other hand, the  expected sample sizes of the multistage tests are not much larger than that of the SPRT when the  true parameter is smaller than  $\mu_0$ or  larger than  $\mu_1$.

\section{Conclusion} \label{sec:conclusions}
Given a fixed-sample-size test  that controls the error probabilities at  two specific distributions, in this paper we design and analyze a 3-stage and two 4-stage tests, with \textit{deterministic} stage sizes, 
which guarantee the same error control. 
Under some additional assumptions, which hold for many testing problems beyond the iid setup, we also conduct an asymptotic analysis for these tests.  Specifically,   we obtain asymptotic approximations for their expected sample sizes  under the two distributions with respect to which we control the  error probabilities, as the latter go to 0.
In particular, when the test statistic is the average log-likelihood ratio between these two distributions, their expected sample sizes under these two distributions are asymptotically the optimal among all sequential tests with the same error control.
Moreover, we obtain a universal asymptotic upper bound, which reveals robustness   in comparison to the corresponding SPRT.     

The above  asymptotic optimality properties require certain constraints on how asymmetrically the two error probabilities go to 0. These constraints are  removed in  \cite{ours_conf}, in an iid setup, using multistage tests in which the number of stages is fixed, but  increases, without a bound, with the   asymmetry between the two error probabilities. An interesting direction is the extension of these results beyond the iid setup, using similar ideas as in the present paper. 

In order to have  multistage tests that achieve asymptotic optimality
under  every distribution of the null and the alternative hypotheses,  at least some stage   sizes need to be  random,  as in   \cite[Section 3]{Lorden_1983}, \cite{Hayre_1985, Bartroff_2008_adaptive}. In these works, such a  uniform asymptotic optimality property was established in the case of iid data that belong to an exponential family and under the assumption of symmetric error probabilities. Ideas from the present work can be useful for extending these results  to more general distributional setups and more asymmetric error probabilities.
 
Finally, another  direction of interest is the application of  multistage tests, as the ones we consider in this work, in a multiple testing setup, similarly to  \cite{Malloy_Nowak_2014}.

\appendix

\section{} \label{app:A}
In this Appendix we prove  the results in Subsection \ref{subsec: FSS}.  To this end,  we start with a preliminary lemma, which holds under only some of the assumptions of   Section \ref{sec:asy}.

\begin{lemma} \label{Lemma: n* to infty}
\begin{enumerate}
\item[(i)] If, for every $n\in \bN$, $\Pro_1$ and $\Pro_0$ are mutually absolutely continuous when restricted to $\cF_n$, then
  $$\n(\alpha,\beta) \to \infty \quad \text{ as } \quad \alpha \wedge  \beta \to 0.$$
\item[(ii)]    If  also \eqref{a.s. limits of eta} holds, then 
   \begin{equation*} 
        J_0\leq \underline\lim \; \thre(\alpha,\beta) \qquad \text{and} \qquad  \overline{\lim}\; \thre(\alpha,\beta) \leq J_1
        \quad \text{ as } \quad \alpha \wedge  \beta \to 0.
    \end{equation*}
    \end{enumerate}
\end{lemma}
\begin{proof} 
(i)  Since $\n$ is decreasing in both its arguments,  it suffices to show $\n(\alpha,\beta)$ goes to infinity  when only one of  $\alpha$  and $\beta$ goes to $0$, while the other one is fixed. Without loss of generality, we assume that  $\alpha$   is fixed and $\beta \to 0$.  We argue by contradiction and suppose that $\n(\alpha,\beta)\not\to \infty$ as $\beta\to 0$.  From this assumption and the fact that 
$\n$ is decreasing  in both  its arguments 
we conclude that there exists an $m\in \bN$ and a sequence 
$(\beta_n)$  with  $\beta_n \to 0$ such that  $\n(\alpha,\beta_n) = m$, $\forall\,n\in \bN$. Then,  for every $n \in \bN$ we have $\thre(\alpha,\beta_n)\geq z_{\alpha}$, where 
$$ z_{\alpha}\equiv \inf\{ z\in \bR: \, \Pro_0(\Stat_{m}>z) \leq \alpha \}>-\infty,$$
and subsequently 
$$ \beta_n \geq \Pro_1(\Stat_{\n(\alpha,\beta_n)}\leq \thre(\alpha,\beta_n)) = \Pro_1(\Stat_{m}\leq \thre(\alpha,\beta_n)) \geq \Pro_1(\Stat_{m}\leq z_{\alpha}).$$
Letting $n \to \infty$ we obtain 
$\Pro_1(\Stat_{m}\leq z_{\alpha})=0$.
By the definition of $z_\alpha$ we also have 
$\Pro_0(\Stat_{m}\leq z_{\alpha})\geq 1-\alpha>0$.
This violates the assumption that $\Pro_0$ is absolutely continuous to $\Pro_1$ when restricted to $\cF_m$, thus, we have reached a contradiction. 

(ii)  We only prove  the first inequality, as the proof of the second  is similar.  Without loss of generality, we assume that $\beta \to 0$, while  $\alpha$ is either fixed or goes to 0.
 We argue by contradiction and suppose that $\underline\lim\,\thre(\alpha,\beta)<J_0$.  Then, there exists an  $\epsilon>0$  so that 
$\underline\lim\, \thre(\alpha,\beta)\leq J_0-2\epsilon$
and  we can find a sequence $(\alpha_n, \beta_n)$, such that $\beta_n \to 0$, $(\alpha_n)$ is either constant or also goes to 0,  and  $\thre(\alpha_n,\beta_n)\leq J_0-\epsilon$ for every $n \in \bN$.  Then, for every $n \in \bN$,
$$\alpha_n \geq \Pro_0(T_{\n(\alpha_n, \beta_n)}>\thre(\alpha_n, \beta_n))\geq \Pro_0(T_{\n(\alpha_n, \beta_n)}>J_0-\epsilon).$$ 
In view of (i) and assumption \eqref{a.s. limits of eta},
 the lower bound goes to 1 as  $n \to \infty$,  which contradicts the fact that the sequence 
 $(\alpha_n)$ is bounded away from 1. 
\end{proof}

\begin{proof}[Proof of Theorem \ref{th:asymptotic of FSS with t}]
The upper bound in \eqref{Eq. ALB and AUB}  implies  that 
    $$ \n(\alpha,\beta) \lesssim \frac{|\log\alpha| \vee  |\log\beta|} {\psi_1(\ka)\wedge \psi_0(\ka)} = \frac{|\log(\alpha\wedge \beta)|}{\psi_1(\ka)\wedge \psi_0(\ka)} \quad \text{for every} \; \ka\in (J_0,J_1),$$ 
and optimizing with respect to $\ka$ we obtain   \eqref{C}.  Therefore, it suffices to show  \eqref{Eq. ALB and AUB}.  To lighten the notation, we set  $\n\equiv \n(\alpha,\beta)$ and $\thre\equiv \thre(\alpha,\beta)$. By the definitions of these quantities we have 
\begin{equation*}
     \Pro_0(\Stat_{\n}>\thre)  \leq 
        \alpha \quad \text{ and } \quad
         \Pro_1(\Stat_{\n}\leq \thre) \leq 
         \beta=\alpha^{\frac{|\log\beta|}{|\log\alpha|}},
\end{equation*}
and as a result 
$$  \max\left\{ \Pro_0(\Stat_{\n}>\thre), \; \Pro_1(\Stat_{\n}\leq \thre)^{\frac{|\log\alpha|}{|\log\beta|}} \right\} \leq \alpha. $$
Since for any $n \in \bN$ and $\ka_1,\ka_2\in \bR$ we have $$\text{either } \quad
\Pro_0(\Stat_n>\ka_1)\geq \Pro_0(\Stat_n>\ka_2) \quad \text{ or } \quad \Pro_1(\Stat_n\leq \ka_1)\geq \Pro_1(\Stat_n\leq \ka_2),$$ 
for any $\ka\in (J_0,J_1)$ we obtain 
\begin{equation*}
    \min\left\{ \Pro_0(\Stat_{\n}>\ka), \; \Pro_1(\Stat_{\n}\leq \ka)^{\frac{|\log\alpha|}{|\log\beta|}} \right\} \leq  \alpha,
\end{equation*} 
and consequently
\begin{equation*} 
\frac{\min\left\{ \frac{1}{\n} \log  \Pro_0(\Stat_{\n}>\ka), \; \frac{1}{\n} \log \Pro_1(\Stat_{\n}\leq \ka) \cdot \frac{|\log\alpha|}{|\log\beta|} \right\} }{\frac{1}{\n} \log\alpha} \geq 1. 
\end{equation*}
Then,  from Lemma \ref{Lemma: n* to infty} and \eqref{LD, non-LLR_null}-\eqref{LD, non-LLR_alternative}   we conclude that,  as $\alpha\wedge\beta\to 0$, 
\begin{equation*}
    \underline\lim \; \frac{\n}{\log\alpha} \min\left\{ -\psi_0(\ka), \; -\psi_1(\ka) \cdot \frac{|\log\alpha|}{|\log\beta|} \right\} \geq 1,  \quad \text{for every} \; \ka\in (J_0,J_1),
\end{equation*}
  which proves the asymptotic upper  bound in \eqref{Eq. ALB and AUB}.  On the other hand,  the definition of $\n$ and $\thre$ implies that, for any $\alpha, \beta \in (0,1)$,  
\begin{equation*}
    \text{either } \quad
    \alpha<\Pro_0(\Stat_{\n-1}>\thre) \quad \text{ or } \quad
    \beta<\Pro_1(\Stat_{\n-1}\leq \thre), 
\end{equation*}
and consequently
\begin{equation*}
    \alpha < \max\left\{ \Pro_0(\Stat_{\n-1}>\thre), \; \Pro_1(\Stat_{\n-1}\leq \thre)^{\frac{|\log\alpha|}{|\log\beta|}} \right\}.
\end{equation*}
Working as before we conclude that, for any  $\ka \in (J_0, J_1)$, 
\begin{equation*} 
    \frac{\max\left\{ \frac{1}{\n-1} \log  \Pro_0(\Stat_{\n-1}>\ka), \; \frac{1}{\n-1} \log \Pro_1(\Stat_{\n-1}\leq \ka) \cdot \frac{|\log\alpha|}{|\log\beta|} \right\}}{\frac{1}{\n-1} \log \alpha} < 1
\end{equation*}
for every  $\alpha, \beta \in (0,1)$, and letting $\alpha\wedge\beta\to 0$  
we obtain 
\begin{equation*}
    \overline\lim \; \frac{\n}{\log\alpha} \max\left\{ -\psi_0(\ka), \; -\psi_1(\ka) \cdot \frac{|\log\alpha|}{|\log\beta|} \right\} \leq 1.
\end{equation*}
Thus, we have established  the asymptotic lower bound in \eqref{Eq. ALB and AUB}, and the proof is complete.
\end{proof}

\begin{proof}[Proof of Theorem \ref{thm: Universal ALB}]

(i)   When both $\alpha$ and $\beta$ go to 0, this follows from the universal asymptotic lower bound in \eqref{optimal_rate}. Therefore,  it suffices to consider the case that only one of them goes to 0, while the other one is fixed. Without loss of generality, we assume that  $\beta \to 0$, while $\alpha$ is fixed, in which case it suffices to show that, for every  $\epsilon>0$, 
$$ \underline\lim \; \frac{\log \beta }{\n(\alpha, \beta)} \geq -I_0 - \epsilon. 
$$
To this end, we fix   $\epsilon>0$ and observe that,
by  Lemma \ref{Lemma: n* to infty}.(ii),    for $\beta$ small enough  we have $\thre(\alpha, \beta)>-I_0- \epsilon$    and consequently
\begin{align}    \label{qqqqq}
\begin{split}
    \beta & \geq \Pro_1(\bar\Lambda_{\n(\alpha, \beta)} \leq \thre(\alpha, \beta))
    \\ &\geq \Pro_1(-I_0-\epsilon<\bar\Lambda_{\n(\alpha, \beta)}\leq \thre(\alpha, \beta)) \\
    & = \Exp_0\left[ \exp\{ \Lambda_{\n(\alpha, \beta)} \}; \; -I_0-\epsilon<\bar\Lambda_{\n(\alpha, \beta)}\leq \thre(\alpha, \beta) \right] \\
    &\geq \exp\{-\n(\alpha, \beta) \, (I_0+\epsilon)\} \; \; \Pro_0(-I_0-\epsilon<\bar\Lambda_{\n(\alpha, \beta)} \leq \thre(\alpha, \beta)) .
    \end{split}
\end{align} 
Moreover, for any $\alpha, \beta \in (0,1)$ we have 
\begin{align}\label{qqqqq2}
\begin{split}
&\Pro_0(-I_0-\epsilon<\bar\Lambda_{\n(\alpha, \beta)} \leq \thre(\alpha, \beta))\\
 &= 1-\Pro_0(\bar\Lambda_{\n(\alpha, \beta)} \leq -I_0-\epsilon) - \Pro_0(\bar\Lambda_{\n(\alpha, \beta)}>\thre(\alpha, \beta))
 \\
& \geq \;  1-\Pro_0(\bar\Lambda_{\n(\alpha, \beta)} \leq -I_0-\epsilon) - \alpha,
 \end{split}
\end{align}
and the  probability in the lower bound of \eqref{qqqqq2} goes to zero  as $\beta\to 0$,  because of Lemma \ref{Lemma: n* to infty}.(i)  and  assumption \eqref{SLLN on LLR}. Therefore, taking logarithms on both sides of \eqref{qqqqq}, dividing by $\n(\alpha, \beta)$ and letting $\beta\to 0$ completes the proof.

(ii) We only prove that, as $\alpha \wedge \beta\to 0$,
 $$\n(\alpha,\beta)\gtrsim 
 \frac{|\log\beta|}{\psi_1(J_0)},$$ 
as the proof that   $\n(\alpha,\beta)\gtrsim |\log\alpha|/\psi_0(J_1)$  is similar. By  assumption \eqref{extra assumption}, there is an $\epsilon>0$ so that $\psi_1$ is finite and  \eqref{LD, non-LLR_alternative} holds  in $(J_0-2\epsilon,J_1)$.
From Lemma \ref{Lemma: n* to infty}.(ii) it follows that, when at least one of  $\alpha$ and $\beta$ is  small enough,  $\thre(\alpha, \beta)  >J_0-\epsilon$ and consequently 
$$ \beta\geq \Pro_1(T_{\n(\alpha,\beta)}\leq \thre(\alpha, \beta) )\geq \Pro_1(T_{\n(\alpha,\beta)}\leq J_0-\epsilon).$$
Thus, taking logarithms, dividing by $\n(\alpha,\beta)$ and   letting $\alpha \wedge \beta\to 0$ we obtain
$$ \underline\lim \; \frac{ \log \beta}{\n(\alpha,\beta)} \geq \underline\lim \; \frac{1}{\n(\alpha,\beta)} \log \Pro_1(T_{\n(\alpha,\beta)}\leq J_0-\epsilon) = -\psi_1(J_0-\epsilon),$$
where the equality follows from  Lemma   \ref{Lemma: n* to infty}.(i) and assumption \eqref{LD, non-LLR_alternative}. Since $\psi_1$ is convex, it is continuous on  the interior of its effective domain.  By assumption, $\psi_1$ is finite in a neighborhood of $J_0$,  thus, letting $\epsilon\downarrow 0$  completes the proof. \\
\end{proof}

\begin{proof}[Proof of Corollary \ref{coro:ARE}] 
The first asymptotic approximation in \eqref{ARE0, r}  follows by  setting $\ka= g^{-1}(r)$ in \eqref{Eq. ALB and AUB}, whereas the second by   \eqref{range of  g_inverse}, which implies 
$$ \psi_0\left( g^{-1}(r) \right) = r\, \psi_1\left( g^{-1}(r) \right) \quad 
\forall \; r \in (0, \infty).
$$ 
To  prove  \eqref{ARE0, r=1}  it suffices to show that 
 $$C=\psi_0(g^{-1}(1))=\psi_1(g^{-1}(1)).$$ 
Indeed,  the strict  monotonicity of $\psi_0$ and $\psi_1$ in $(J_0,J_1)$ implies that  the supremum in  \eqref{C} is attained when $\psi_0=\psi_1$, or equivalently when $g=1$.
\end{proof}

\begin{proof}[Proofs of Corollaries \ref{coro: ARE when r=0 or infty, LLR} and \ref{coro: ARE when r=0 or infty, non-LLR}]

In view of the asymptotic lower bounds in Theorem \ref{thm: Universal ALB}, it satisfies to establish only the corresponding upper bounds.  We only prove part  (i) of each Corollary, 
as the proof of (ii) is similar. 

We  show first that, for any test statistic $T$, \textit{even if  \eqref{extra assumption} does not hold}, 
\begin{align*}
\n(\alpha,\beta) &\lesssim \frac{|\log\beta|}{\psi_1(J_0)}, \quad 
\text{or equivalently  }\quad 
 \psi_1(J_0) \lesssim \frac{|\log\beta|} {\n(\alpha,\beta) } , 
 \end{align*} 
 as   $\alpha \wedge \beta\to 0$ so that  $|\log \alpha|<<|\log\beta|$. 

By assumption,   $\psi_1$ is convex and  lower-semicontinuous, thus,  it   is continuous in its effective domain, and as a result in  $[J_0, J_1]$. Therefore, to prove the above claim
  it suffices to show that, as   $\alpha \wedge \beta\to 0$ so that  $|\log \alpha|<<|\log\beta|$,
\begin{align*}
 \psi_1(\ka) &\lesssim \frac{|\log\beta|} {\n(\alpha,\beta) }  \quad \forall \; \ka\in (J_0, J_1),
\end{align*} 
which follows directly by Theorem \ref{th:asymptotic of FSS with t}.  When  $T \neq \bar{\Lambda}$, the proof is complete.  When $T=\bar{\Lambda}$, it remains to show that   $\psi_1(-I_0)\geq I_0$. Since   $\psi_1$ is  continuous in  
$[-I_0, I_1]$,  it suffices to show that    $\psi_1(\ka) \geq -\ka$ for every $\ka\in (-I_0,0)$. Indeed, for any $\ka<0$, by Markov's inequality we have
    $$ \Pro_1(\bar\Lambda_n \leq \ka) 
 \leq e^{n \ka}  \; \Exp_0[ \exp\{ \Lambda_n\}] =  e^{n \ka}  \quad \text{for all} \quad n \in \bN.$$
Taking logarithms, dividing by $n$, letting $n \to \infty$,  and applying    \eqref{LD, non-LLR_alternative}  for $\ka$ in  $(-I_0,0)$  completes the proof. 
\end{proof}

\section{} \label{app:new}
In this Section we prove
Lemma \ref{lem: ALB, multi} and 
Theorem \ref{theorem, main, non-LLR}.   The proof of Theorem
 \ref{theorem, main}  is  omitted, as it is almost identical to that of Theorem \ref{theorem, main, non-LLR}.
 
\begin{proof}[Proof of Lemma \ref{lem: ALB, multi}]
We only prove the asymptotic lower bounds under $\Pro_0$, as the  proofs of the corresponding  lower bounds under  $\Pro_1$ are similar. We first  prove the result for $\three$, in which case it   suffices to show that, 
for all $\epsilon \in (0,1)$,
\begin{align} \label{show1}
 \inf_{\gamma\in (\alpha/2,1)} \Exp_0[\TST] \gtrsim (1-\epsilon) \,\frac{|\log\beta|}{\psi_1(J_0)} \quad \text{as} \quad  \alpha, \beta \to 0.
 \end{align}
 Fix  $\epsilon \in (0,1)$.    By the non-asymptotic lower bound in \eqref{3ST, ESS bound, null, after plugin} it follows that, for any $\alpha,\beta\in (0,1)$ and $\gamma\in (\alpha/2,1)$,
$$ \Exp_0[\TST]\geq \max\big\{ \n(\gamma,\beta/2) \cdot (1-\alpha/2), \; \n(\alpha/2,\beta/2)\cdot (\gamma-\alpha/2) \big\}.$$
When, in particular,  $\gamma \leq 1-\epsilon$,
\begin{align*}
\Exp_0[\TST] 
&\geq     \n(\gamma,\beta/2) \cdot (1-\alpha/2) \geq      \n(1-\epsilon, \, \beta/2) \cdot (1-\alpha/2).
\end{align*}
and when 
$\gamma> 1-\epsilon$,
\begin{align*}
\Exp_0[\TST] &\geq \n(\alpha/2, \, \beta/2) \cdot (\gamma-\alpha/2)\geq \n(\alpha/2, \, \beta/2) \cdot (1-\epsilon- \alpha/2).
\end{align*}
By Theorem \ref{thm: Universal ALB}.(ii) it then follows that,  as $\alpha,\beta\to 0$, 
\begin{equation*}
    \begin{split}
         \inf_{\gamma\in (\alpha/2, \, 1-\epsilon)}\Exp_0[\TST] &\gtrsim \n(1-\epsilon, \, \beta/2) \gtrsim \frac{|\log\beta|}{\psi_1(J_0)}, \\
         \inf_{\gamma\in (1-\epsilon, 1)}\Exp_0[\TST] &\gtrsim  (1-\epsilon)  \cdot \n(\alpha/2,\beta/2) \gtrsim (1-\epsilon)\,\frac{|\log\beta|}{\psi_1(J_0)},
    \end{split}
\end{equation*}
and this implies \eqref{show1}. The proof for $\ffour$ is similar and omitted.
To prove the result for $\four$, it suffices to show that, for every  $\epsilon \in (0,1)$, 
\begin{align} \label{show2}
\inf_{\alpha/2<\gamma'<\gamma<1} \Exp_0[\FST] & \gtrsim (1-2\epsilon) \, \frac{|\log\beta|}{\psi_1(J_0)} \quad \text{as} \quad  \alpha,\beta\to 0. 
\end{align}
Fix $\epsilon \in (0,1)$. By the non-asymptotic lower bound in \eqref{4ST, ESS bound, after plugin, null} it follows that,  for any $\alpha,\beta\in (0,1)$  and $\alpha/2<\gamma'<\gamma<1$,
\begin{equation*}
    \begin{split}
        \Exp_0[\FST]\geq \max\big\{ &  \n(\gamma, \, \beta/3) \cdot (1-\alpha/2),        \;  \n(\gamma',\beta/3)\cdot (\gamma-\alpha/2), \\
        &  \n(\alpha/2,\beta/3) \cdot
         ( (1- \alpha/2)- (1- \gamma)-
        (1-\gamma') )\big\}.
    \end{split}    
\end{equation*}
When, in particular, $\gamma<1-\epsilon$, 
\begin{equation*}
    \begin{split}
   \Exp_0[\FST]   &\geq  \n(\gamma, \, \beta/3) \cdot (1-\alpha/2)     \\
   & \geq \n(1-\epsilon,\, \beta/3) \cdot (1-\alpha/2),
    \end{split}
\end{equation*}
when  $   \gamma' < 1-\epsilon< \gamma$,
\begin{align*}
   \Exp_0[\FST]
  & \geq     \n(\gamma', \, \beta/3)
  \cdot (\gamma-\alpha/2) \\
  &\geq  \n(1-\epsilon,\, \beta/3) \cdot (1-\epsilon-\alpha/2),   
\end{align*}
 and when  $\gamma' > 1-\epsilon$, 
\begin{align*}
   \Exp_0[\FST] & \geq           \n(\alpha/2,\, \beta/3) \cdot 
   ( (1 -\alpha/2) - (1-\gamma)- (1-\gamma'))  \\
   &\geq    \n(\alpha/2,\, \beta/3)\cdot (1-2\epsilon-\alpha/2)   .
\end{align*}

By Theorem \ref{thm: Universal ALB}.(ii) it then follows that, as $\alpha,\beta\to 0$,
\begin{equation*}
    \begin{split}
        \underset{\gamma\in (\alpha/2, 1-\epsilon)}{\inf}\Exp_0[\FST] & \gtrsim \n(1-\epsilon,\, \beta/3) \gtrsim \frac{|\log\beta|}{\psi_1(J_0)}, \\
        \underset{\substack{\gamma\in (1-\epsilon, 1) \\ \gamma'\in (\alpha/2, 1-\epsilon)}}{\inf}\Exp_0[\FST] & \gtrsim  \n(1-\epsilon,\, \beta/3) \cdot (1-\epsilon)   \gtrsim (1-\epsilon) \, \frac{|\log\beta|}{\psi_1(J_0)}, \\
        \underset{\gamma'\in(1-\epsilon, 1)}{\inf}\Exp_0[\FST] & \gtrsim \n(\alpha/2, \, \beta/3)\cdot (1-2\epsilon) \gtrsim (1-2\epsilon)\,\frac{|\log\beta|}{\psi_1(J_0)}.
    \end{split}
\end{equation*}
and this implies \eqref{show2}. 
\end{proof}

\begin{proof}[Proof of Theorem \ref{theorem, main, non-LLR}]
In view of Lemma \ref{lem: ALB, multi},  it remains to prove in each case  the corresponding  asymptotic upper bounds.

(i) Let  $\delta$ be  a function of $\alpha, \beta$ such that 
$\delta\in (\beta/2,1)$  for every  $\beta \in (0,1)$,  and 
 \begin{align} \label{main thm, condition on delta}
&\delta \to 0 \quad \text{and} \quad    |\log\delta|<<|\log\alpha|  \quad \text{as} \quad   \alpha \to 0,
\end{align}
e.g.,  $\delta=|\log \alpha|^{-\epsilon} \vee \beta$ for some $\epsilon \in (0,1)$. By the non-asymptotic upper bound in \eqref{3ST, ESS bound, alternative, after plugin} and the selection of the free parameters according to  \eqref{free_3ST}  it  follows that,  for any $\alpha, \beta \in (0,1)$, 
\begin{align*}
\Exp_1[\TST] 
& \leq \n(\alpha/2,\delta)+\left( \n(\alpha/2,\beta/2)-\n(\alpha/2,\delta) \right)\cdot  \delta \\
&\leq \n(\alpha/2,\delta)+\n(\alpha/2,\beta/2)\cdot \delta.
\end{align*}
Then, by  Corollary \ref{coro: ARE when r=0 or infty, LLR}.(ii),  Theorem \ref{th:asymptotic of FSS with t}  and \eqref{main thm, condition on delta} we conclude that 
\begin{align*}
 \Exp_1[\TST] 
&\lesssim    \frac{|\log \alpha|}{\psi_0(J_1)} +  \frac{|\log(\alpha\wedge\beta)|}{C}\,\delta  \sim    \frac{|\log \alpha|}{\psi_0(J_1)} \\ &\text{as} \quad  \alpha,\beta\to 0 \quad  \text{so that}  \quad |\log \alpha| \gtrsim |\log \beta|,
\end{align*}
and this completes the proof for $\three$. The proof for $\four$ is similar and  omitted. To prove the result for    $\ffour$,  we observe that by the non-asymptotic upper bound in \eqref{44ST, ESS bound, after plugin, alternative} and the selection of the free parameters according to  \eqref{free_4ST_check} it follows that 
\begin{align*} 
 \Exp_1[\FFST] 
  &\leq \n(\alpha/3,\delta) + (\n(\alpha/3,\delta')- \n(\alpha/3,\delta) ) \cdot  \delta
 \\
 &+ (\n(\alpha/3,\beta/2)-\n(\alpha/3,\delta') \cdot \deltap \\ &\leq \n(\alpha/3,\delta) + \n(\alpha/3,\beta/2)\cdot \delta
\end{align*}
for any $\alpha,\beta\in (0,1)$ and  $\delta', \delta$ such that   $\beta/2<\delta'<\delta<1$.
The proof  then  continues in exactly the same way as for  $\three$, i.e., by selecting $\delta$ to satisfy   \eqref{main thm, condition on delta}.

(ii) Let $\gamma,\gammap$ be functions of $\alpha$ and $\beta$ such that $\alpha<\gammap<\gamma<1$ and
\begin{align} \label{main thm, condition on gamma, gamma'}
\begin{split}
  &  |\log\gamma|  <<|\log\beta|, \quad  |\log(\gammap\wedge\beta)|\,\gamma <<|\log\beta|,\quad  |\log\alpha|\,\gammap<<|\log\beta|,\\
    & \text{as }\; \alpha,\beta\to 0 \;\text{ so that }\; |\log\beta| \lesssim |\log\alpha|\lesssim |\log\beta|/\beta^r \; \text{ for some }\; r\geq 1.
    \end{split}
\end{align}
e.g.,    $ \gamma = \; |\log\beta|^{-\epsilon}\vee \alpha$ and  $\gammap = (\beta^{r+\epsilon'}\wedge\gamma)\vee \alpha$ for some $\epsilon\in (0,1)$ and  $\epsilon'>0$.  By the non-asymptotic upper bound in \eqref{4ST, ESS bound, after plugin, null} and the selection of the free parameters according to  \eqref{free_4ST_hat} it follows that,  for any $\alpha,\beta\in (0,1)$,
\begin{align*} 
   \Exp_0[\FST]  
      & \leq \n(\gamma,\beta/3)+\big(\n(\gammap,\beta/3)-\n(\gamma,\beta/3)\big)\cdot \gamma \\
      &+ \big(\n(\alpha/2,\beta/3)-\n(\gammap,\beta/3)\big)\cdot  \gammap \\
    & \leq \n(\gamma,\beta/3) + \n(\gamma',\beta/3)\cdot \gamma+\n(\alpha/2,\beta/3)\cdot \gamma'.
\end{align*}
Then, by    Corollary \ref{coro: ARE when r=0 or infty, LLR}.(i),  Theorem \ref{th:asymptotic of FSS with t} and  \eqref{main thm, condition on gamma, gamma'}  we conclude that 
\begin{align*} 
  \Exp_0[\FST] 
&\lesssim    \frac{|\log \beta|}{\psi_1(J_0)} +  \frac{|\log(\gammap\wedge\beta)|}{C}\,\gamma + \frac{|\log(\alpha \wedge \beta)|}{C}\,\gammap \sim \frac{|\log \beta|}{\psi_1(J_0)}\\ 
    & \text{as }\; \alpha,\beta\to 0 \;\text{ so that }\; |\log\beta| \lesssim |\log\alpha|\lesssim |\log\beta|/\beta^r \; \text{ for some }\; r\geq 1.
\end{align*}

(iii) To prove the result for $\Exp_0[\TST]$, we let $\gamma$ be a  function of $\alpha$ and $\beta$ such that $\gamma\in(\alpha,1)$ and 
\begin{equation} \label{main thm, condition on gamma}
\begin{split}
&    |\log\gamma|  <<|\log\beta| \quad \text{and} 
    \quad |\log\beta|^{r-1}\, \gamma  \to 0 \\
    & \text{as }\; \alpha,\beta\to 0 \;\text{ so that }\; |\log\beta| \lesssim |\log\alpha|\lesssim |\log\beta|^r \; \text{ for some }\; r\geq 1,
\end{split}
\end{equation}
e.g., $    \gamma = |\log\beta|^{-r+\epsilon}\vee\alpha$  for some  $\epsilon\in (0,1)$. By the non-asymptotic upper bound in \eqref{3ST, ESS bound, null, after plugin} and the selection of the free parameters according to \eqref{free_3ST} it follows that,    for any $\alpha,\beta\in (0,1)$,
\begin{equation*}
    \begin{split}
        \Exp_0[\TST] 
        & \leq \n(\gamma,\beta/2) + \big(\n(\alpha/2,\beta/2)-\n(\gamma,\beta/2)\big)\cdot \gamma \\
        & \leq \n(\gamma,\beta/2) + \n(\alpha/2,\beta/2)\cdot \gamma.
    \end{split}
\end{equation*}
Then,   by  Corollary \ref{coro: ARE when r=0 or infty, LLR}.(i), Theorem \ref{th:asymptotic of FSS with t} and 
\eqref{main thm, condition on gamma} we conclude that
\begin{align*} 
  \Exp_0[\TST] 
&\lesssim    \frac{|\log \beta|}{\psi_1(J_0)} + \frac{|\log(\alpha \wedge \beta)|}{C}\,\gamma  \sim  \frac{|\log \beta|}{\psi_1(J_0)}   \\
  & \text{as }\; \alpha,\beta\to 0 \;\text{ so that }\; |\log\beta| \lesssim |\log\alpha|\lesssim |\log\beta|^r \; \text{ for some }\; r\geq 1.
\end{align*}
\end{proof}

\section{} \label{app: proof from Assumptions B to Assumptions A}
In this Appendix, we prove Theorem 
\ref{th:A version of G-E}, a version of the G\"artner-Ellis Theorem.  The proof is essentially the same as in \cite[Theorem 2.3.6]{Dembo_Zeitouni_LDPBook} or \cite[Theorem 3.2.1]{Bucklew_Book}, and is presented only for completeness.  Specifically,  we establish first  the  asymptotic upper bounds in (i) and (ii). Using these, we establish (iii). Finally, using (iii),  we establish the asymptotic lower bounds in (i) and (ii).

\begin{proof} [Proof of Theorem 
\ref{th:A version of G-E}]
We  establish the asymptotic upper bound only for (i), as the corresponding proof for (ii) is similar.  Thus, we assume that $\Theta^o\cap(0,\infty) \neq\emptyset$. For any $\ka_1,\ka_2\in \phi'(\Theta^o\cap(0,\infty))$  such that 
$\ka_1<\ka_2$ and  $\vartheta(\ka_1), \vartheta(\ka_2)>0$,
\begin{align*}
    \phis(\ka_1) =\vartheta(\ka_1) \ka_1-\phi(\vartheta(\ka_1)) < \vartheta(\ka_1) \ka_2 - \phi(\vartheta(\ka_1)) \leq \phis(\ka_2),
\end{align*}
which  proves that $\phis$ is strictly increasing in $\phi'(\Theta^o\cap(0,\infty))$.  From  \cite[Lemma 2.2.5]{Dembo_Zeitouni_LDPBook})
it follows that  $\phis$  is  non-negative and lower-semicontinuous,
and these properties imply that
$$ 0\leq \phis\left( \phi'(0+)\right) \leq \lim_{\theta\downarrow 0} \, \phis\left( \phi'(\theta)\right) = \lim_{\theta\downarrow 0} \, \left\{\theta \phi'(\theta)-\phi(\theta)\right\} = 0. $$

Since $ \phi'(\Theta^o\cap(0,\infty))$ is an open interval, whose right endpoint may be infinity, 
to show that \eqref{>ka} holds for every  $\ka\in
\phi'(\Theta^o\cap(0,\infty))$  it suffices to show that it  holds for every  $\ka\in \phi'((0,\thetas))$, where  $\thetas\in \Theta^o\cap(0,\infty)$.
Thus, we fix  $\thetas\in \Theta^o\cap(0,\infty)$ and denote $\phi'((0,\thetas)) \equiv (a,b)$, where $a \equiv \phi'(0+)$ and $b \equiv \phi'(\thetas)$. 

For any $\ka\in (a,b)$ and $\theta\in (0,\thetas)$, we have
\begin{equation*}
    \begin{split}
        \Pro(\Stat_n>\ka) & \leq \exp\{-n \, \theta\, \ka \} \, \Exp \left[ \exp\{n \, \theta \, \Stat_n\} \right]  =\exp\{-n (\theta \ka  - \phi_{n}(\theta))\},
    \end{split}
\end{equation*}
which, after taking logarithm, dividing by $n$ and letting $n\to \infty$, gives
$$ \overline{\lim_n} \; \frac{1}{n} \log \Pro(\Stat_n >\ka) \leq - \left(\theta \ka  -\phi(\theta)\right). $$
Optimizing the right-hand-side with respect to $\theta\in (0,\thetas)$, we obtain $-\phis(\ka)$.

Note that this asymptotic upper bound is non-trivial for every $\ka\in (a,\infty)$, since $\phis(a)=0$ and $\phis$ is strictly increasing in $(a,b)$. Therefore, it implies that $\Pro\left(T_n- \phi'(0+) > \epsilon\right)$ is an exponentially decaying sequence   for every $\epsilon>0$.  Similarly  it follows that  if  $\Theta^o\cap(-\infty,0)\neq\emptyset$, 
then $ \Pro\left(T_n-  \phi'(0-)  \leq -\epsilon\right)$  is an exponentially decaying sequence   for every $\epsilon>0$. From these observations  we conclude that if   $0\in \Theta^o$,   then $\Pro \left( |T_n-\phi'(0)|>\epsilon \right)$ is exponentially decaying for every $\epsilon>0$, and as a result  $\Pro \left( T_n \to \phi'(0) \right)=1$. Therefore,  (iii) follows using  exactly the same argument  as long as the sequence of functions
 $$ \lambda\in \bR  \to  \frac{1}{n} \, \log \Exp_{\Qro_\theta} \left[  \exp\{n\lambda T_{n}\} \right],  \quad  \,  n \in \bN $$
satisfies the  assumptions of  the theorem,  0 belongs to the interior of the effective domain of its limit, and the derivative of its limit at 0 is $\phi'(\theta)$. 
To show this, we fix $\theta\in (0,\thetas)$. Then, for any $\lambda\in \bR$,
\begin{equation*}
    \begin{split}
        \Exp_{\Qro_\theta} \left[\exp\{n\lambda T_n\}\right] & = \Exp\left[ \exp\{n((\lambda+\theta)T_n-\phi_n(\theta))\} \right] \\
        &= \exp\{ n(\phi_n(\lambda+\theta)-\phi_n(\theta)) \},
    \end{split}
\end{equation*}
and consequently
\begin{align*}
    \lim_n \, \frac{1}{n} \log \Exp_{\Qro_\theta} \left[\exp\{n\lambda T_n\}\right] =  \phi(\lambda+\theta) - \phi(\theta).
\end{align*}
The limit  is finite for $\lambda\in (-\theta,\thetas-\theta)$, which contains 0 in its interior,  inherits all the smoothness properties of  $\phi$,
and its derivative at $\lambda=0$ is  $\phi'(\theta)$.  This completes the proof of (iii).

It remains to prove the asymptotic lower bounds in (i) and (ii). Again, we only do so  for (i), as the proof for (ii) is similar.  Fix  $\ka\in (a,b)$. For any $n\in \bN$, $\theta\in (0,\thetas)$ and $\epsilon\in (0,b-\ka)$,
\begin{equation*}
    \begin{split}
        \Pro(T_n>\ka) & = \Exp_{\Qro_\theta} \left[ \exp\{-n(\theta T_n-\phi_n(\theta))\}; \, T_n>\ka \right] \\
        & \geq \Exp_{\Qro_\theta} \left[ \exp\{-n(\theta T_n-\phi_n(\theta))\};\, \ka<T_n\leq\ka+\epsilon \right] \\
        & \geq \exp\{ -n(\theta(\ka+\epsilon)-\phi_n(\theta)) \} \, \Qro_\theta(\ka<T_n\leq \ka+\epsilon).
    \end{split}
\end{equation*}
If we now set $\theta=\vartheta(\ka+\epsilon/2)$, take logarithms, divide by $n$ and let $n \to \infty$,  by  (iii) we obtain
\begin{align*}
    \underset{n}{\underline\lim} \; \frac{1}{n} \log \Pro(\Stat_n >\ka) 
    \geq - \vartheta(\ka+\epsilon/2) (\ka+\epsilon) + \phi(\vartheta(\ka+\epsilon/2)).
 \end{align*}
To   complete the proof, we let  $\epsilon\downarrow 0$ and observe that the right-hand-side converges to $-\big(\vartheta(\ka) \ka- \phi(\vartheta(\ka))\big)=-\phis(\ka)$, since $\vartheta$ and $\phi$ are both continuous in the corresponding neighborhoods.
\end{proof}


\begin{acks}[Acknowledgments]
\end{acks}

\bibliographystyle{abbrv}
\bibliography{new.bib}

\begin{thebibliography}{10}

\bibitem{Armitage_1969}
P.~Armitage, C.~K. McPherson, and B.~C. Rowe.
\newblock Repeated significance tests on accumulating data.
\newblock {\em Journal of the Royal Statistical Society. Series A (General)},
  132(2):235--244, 1969.

\bibitem{Barber_Jennison_2002}
S.~Barber and C.~Jennison.
\newblock Optimal asymmetric one-sided group sequential tests.
\newblock {\em Biometrika}, 89(1):49--60, 2002.

\bibitem{Bartroff_2007}
J.~Bartroff.
\newblock Asymptotically optimal multistage tests of simple hypotheses.
\newblock {\em The Annals of Statistics}, 35(5):2075--2105, 2007.

\bibitem{Bartroff_2008_adaptive}
J.~Bartroff and T.~L. Lai.
\newblock Efficient adaptive designs with mid-course sample size adjustment in
  clinical trials.
\newblock {\em Statistics in Medicine}, 27(10):1593--1611, 2008.

\bibitem{Bartroff_2008_anotheradaptive}
J.~Bartroff and T.~L. Lai.
\newblock Generalized likelihood ratio statistics and uncertainty adjustments
  in efficient adaptive design of clinical trials.
\newblock {\em Sequential Analysis}, 27(3):254--276, 2008.

\bibitem{Bartroff_Book_Clinicaltrials}
J.~Bartroff, T.~L. Lai, and M.-C. Shih.
\newblock {\em Sequential experimentation in clinical trials: design and
  analysis}, volume 298.
\newblock Springer Science \& Business Media, 2012.

\bibitem{Bechhofer60}
R.~Bechhofer.
\newblock A note on the limiting relative efficiency of the wald sequential
  probability ratio test.
\newblock {\em Journal of the American Statistical Association},
  55(292):660--663, 1960.

\bibitem{Bercu_1997}
B.~Bercu, F.~Gamboa, and A.~Rouault.
\newblock Large deviations for quadratic forms of stationary gaussian
  processes.
\newblock {\em Stochastic Processes and their Applications}, 71:75--90, 1997.

\bibitem{Brockwell_TimeSeriesBook}
P.~J. Brockwell and R.~A. Davis.
\newblock {\em Time Series: Theory and Methods}.
\newblock Springer-Verlag, Berlin, Heidelberg, 1986.

\bibitem{Bucklew_Book}
J.~Bucklew.
\newblock {\em Introduction to Rare Event Simulation}.
\newblock Springer Publishing Company, Incorporated, 1st edition, 2010.

\bibitem{Chernoff_book}
H.~Chernoff.
\newblock {\em Sequential analysis and optimal design}.
\newblock SIAM, 1972.

\bibitem{Dembo_Zeitouni_LDPBook}
A.~Dembo and O.~Zeitouni.
\newblock {\em Large Deviations Techniques and Applications}.
\newblock Springer, Berlin, Heidelberg, 1998.

\bibitem{Dodge_Romig_1929}
H.~F. Dodge and H.~G. Romig.
\newblock A method of sampling inspection.
\newblock {\em The Bell System Technical Journal}, 8(4):613--631, 1929.

\bibitem{Durret_Book}
R.~Durrett.
\newblock {\em Probability: Theory and Examples}.
\newblock Cambridge University Press, USA, 4th edition, 2010.

\bibitem{Eales_Jennison_1992}
J.~D. Eales and C.~Jennison.
\newblock An improved method for deriving optimal one-sided group sequential
  tests.
\newblock {\em Biometrika}, 79(1):13--24, 1992.

\bibitem{Emerson_Fleming_1989}
S.~S. Emerson and T.~R. Fleming.
\newblock Symmetric group sequential test designs.
\newblock {\em Biometrics}, 45(3):905--923, 1989.

\bibitem{Hayre_1985}
L.~S. Hayre.
\newblock Group sequential sampling with variable group sizes.
\newblock {\em Journal of the Royal Statistical Society: Series B
  (Methodological)}, 47(1):90--97, 1985.

\bibitem{Jennison_1987}
C.~Jennison.
\newblock Efficient group sequential tests with unpredictable group sizes.
\newblock {\em Biometrika}, 74(1):155--165, 1987.

\bibitem{Jennison_Turnbull_Book}
C.~Jennison and B.~W. Turnbull.
\newblock {\em Group sequential methods with applications to clinical trials}.
\newblock CRC Press, 1999.

\bibitem{Kim_DeMets_1987}
K.~Kim and D.~L. DeMets.
\newblock Design and analysis of group sequential tests based on the type i
  error spending rate function.
\newblock {\em Biometrika}, 74(1):149--154, 1987.

\bibitem{Lai_Shih_2004}
T.~L. Lai and M.-C. Shih.
\newblock Power, sample size and adaptation considerations in the design of
  group sequential clinical trials.
\newblock {\em Biometrika}, 91(3):507--528, 2004.

\bibitem{Lan_DeMets_1983}
K.~K.~G. Lan and D.~L. DeMets.
\newblock Discrete sequential boundaries for clinical trials.
\newblock {\em Biometrika}, 70(3):659--663, 1983.

\bibitem{Lorden_1983}
G.~Lorden.
\newblock Asymptotic efficiency of three-stage hypothesis tests.
\newblock {\em Annals of Statistics}, 11:129--140, 1983.

\bibitem{Malloy_Nowak_2014}
M.~L. Malloy and R.~D. Nowak.
\newblock Sequential testing for sparse recovery.
\newblock {\em IEEE Transactions on Information Theory}, 60(12):7862--7873,
  2014.

\bibitem{OBrien_Fleming_1979}
P.~C. O'Brien and T.~R. Fleming.
\newblock A multiple testing procedure for clinical trials.
\newblock {\em Biometrics}, 35(3):549--556, 1979.

\bibitem{PAMPALLONA1994}
S.~Pampallona and A.~A. Tsiatis.
\newblock Group sequential designs for one-sided and two-sided hypothesis
  testing with provision for early stopping in favor of the null hypothesis.
\newblock {\em Journal of Statistical Planning and Inference}, 42(1):19--35,
  1994.

\bibitem{MeasureOnMetricSpace}
K.~R. Parthasarathy.
\newblock {\em Probability measures on metric spaces}.
\newblock American Mathematical Soc., 2005.

\bibitem{Pocock_1977}
S.~J. Pocock.
\newblock Group sequential methods in the design and analysis of clinical
  trials.
\newblock {\em Biometrika}, 64(2):191--199, 1977.

\bibitem{Pocock_1982}
S.~J. Pocock.
\newblock Interim analyses for randomized clinical trials: The group sequential
  approach.
\newblock {\em Biometrics}, 38(1):153--162, 1982.

\bibitem{Tartakovsky_Book}
A.~Tartakovsky, I.~Nikiforov, and M.~Basseville.
\newblock {\em Sequential Analysis: Hypothesis Testing and Changepoint
  Detection}.
\newblock Chapman \& Hall/CRC, 1st edition, 2014.

\bibitem{Wald_Book}
A.~Wald.
\newblock {\em Sequential Analysis}.
\newblock John Wiley \& Sons, New York, 1947.

\bibitem{Wald1948OptimumCO}
A.~Wald and J.~Wolfowitz.
\newblock Optimum character of the sequential probability ratio test.
\newblock {\em Annals of Mathematical Statistics}, 19:326--339, 1948.

\bibitem{Wang_Tsiatis_1987}
S.~K. Wang and A.~A. Tsiatis.
\newblock Approximately optimal one-parameter boundaries for group sequential
  trials.
\newblock {\em Biometrics}, 43(1):193--199, 1987.

\bibitem{ours_conf}
Y.~Xing and G.~Fellouris.
\newblock Asymptotically optimal multistage tests for iid data.
\newblock {\em In 2022 IEEE International Symposium on Information Theory
  (ISIT)}, page to appear, 2022.

\end{thebibliography}

\end{document}